%% file: contact_structures_Seifert_surface_complements.tex
\documentclass{article}

\usepackage{hyphenat}
\usepackage{graphicx}
\usepackage{hyperref}
\usepackage{amsfonts, amsmath, amsthm, bm}
\usepackage[margin=3cm]{geometry}
\usepackage{epstopdf}
\usepackage{verbatim}

\usepackage{tikz}
\usetikzlibrary{calc, arrows, decorations.markings, decorations.pathmorphing, positioning, decorations.pathreplacing, knots}
\makeatletter
\tikzset{nomorepostaction/.code={\let\tikz@postactions\pgfutil@empty}}
\makeatother

\newtheorem{thm}{Theorem}[section]
\newtheorem{cor}[thm]{Corollary}
\newtheorem{lem}[thm]{Lemma}
\newtheorem{prop}[thm]{Proposition}

\newtheorem{defn}[thm]{Definition}

\newcommand{\To}{\longrightarrow}

\DeclareMathOperator{\Bip}{Bip}

\newcommand{\Conf}{\mathcal{C}}

\newcommand{\F}{\mathcal{F}}

\newcommand{\G}{\mathcal{G}}
\newcommand{\GG}{\mathbb{G}}
\newcommand{\HH}{\mathcal{H}}

\DeclareMathOperator{\Int}{Int}

\renewcommand{\L}{\mathbb{L}}
\DeclareMathOperator{\lk}{lk}

\newcommand{\M}{\mathcal{M}}

\newcommand{\QQ}{\mathcal{Q}}

\newcommand{\R}{\mathbb{R}}

\newcommand{\s}{\mathfrak{s}}

\DeclareMathOperator{\Spin}{Spin}

\DeclareMathOperator{\Supp}{Supp}
\newcommand{\T}{\mathcal{T}}

\DeclareMathOperator{\tb}{tb}

\newcommand{\U}{\mathcal{U}}

\newcommand{\V}{\mathcal{V}}

\newcommand{\x}{{\bf x}}

\newcommand{\Z}{\mathbb{Z}}

\let\oldmarginpar\marginpar
\renewcommand\marginpar[1]{\oldmarginpar[\raggedleft\footnotesize #1]%
{\raggedright\footnotesize #1}}

\begin{document}

\title{Tight contact structures on Seifert surface complements}

\author{Tam\'{a}s K\'{a}lm\'{a}n and Daniel V. Mathews}

\date{}

\maketitle

\begin{abstract}
We consider complements of standard Seifert surfaces of special alternating links. On these handlebodies, we use Honda's method to enumerate those tight contact structures whose dividing sets are isotopic to the link, and find their number to be the leading coefficient of the Alexander polynomial. The Euler classes of the contact structures are identified with hypertrees in a certain hypergraph. Using earlier work, this establishes a connection between contact topology and the Homfly polynomial. We also show that the contact invariants of our tight contact structures form a basis for sutured Floer homology. Finally, we relate our methods and results to Kauffman's formal knot theory.
\end{abstract}

\tableofcontents

\section{Introduction}

\subsection{Overview}

This paper classifies tight contact structures on a particular family of sutured 3-manifolds arising out of bipartite plane graphs. We find interesting coincidences between the numbers of such contact structures for certain related graphs. We also connect our results to knot theory in several ways.

In theory, the classification of tight contact structures up to isotopy on a given 3-manifold can be reduced to a combinatorial question about dividing sets on surfaces, via work of Eliashberg, Giroux, Honda and others (see, e.g., \cite{ElOT, ElMartinet, Gi91, Hon02}). However, the set of 3-manifolds for which a full classification is known remains rather small (see, e.g., \cite{Gi00, GiBundles, Hon00I, Hon00II}). In particular, no explicit classification has been given for sutured solid tori in general (more precisely when there are more than 2 boundary dividing curves, cf.\ \cite{Hon00I, Hon01}); handlebodies, even less so.

In this paper we provide a full and explicit classification of tight contact structures on an infinite family of sutured handlebodies. Moreover, we explicitly calculate the invariants of all these contact structures in sutured Floer homology.

This family of sutured 3-manifolds was studied by the first author, along with Juh\'{a}sz and Rasmussen, in \cite{Juhasz-Kalman-Rasmussen12}. From a finite connected bipartite plane graph $G$, a sutured 3-manifold $(M_G, L_G)$ is constructed as follows. Thicken $G$ into a ribbon in the plane and insert a negative half-twist over each edge; this yields a minimal genus Seifert surface $F_G$ for a non-split special alternating link $L_G$. Then $M_G$ is obtained by splitting $S^3$ along $F_G$, and $L_G$ gives a set of sutures. See figure \ref{fig:median_construction}. The details of this construction, as with everything mentioned in this introduction, will be described in more detail as we proceed.

We will show that the tight contact structures on $(M_G, L_G)$ are closely related to the combinatorics of the bipartite plane graph $G$. Such graphs have been studied at least since Tutte's 1948 work \cite{Tutte48}. A bipartite plane graph $G$ naturally yields a \emph{trinity}, a 
3-coloured triangulation of the sphere containing three bipartite plane graphs $G_V$, $G_E$, and $G_R=G$. A bipartite graph can also be regarded as a \emph{hypergraph}. In fact, a bipartite graph yields two dual or ``transpose'' hypergraphs (and any hypergraph derives from a bipartite graph). Hence, a trinity naturally contains \emph{six} hypergraphs.

The three bipartite plane graphs $G_V, G_E, G_R$ are closely related. Tutte's `tree trinity theorem' says that the planar duals $G_V^*, G_E^*, G_R^*$ all have the same arborescence number $\rho$. Then, the matrix-tree theorem implies that the sandpile groups of these three directed graphs \cite{HLMPPW} all have order $\rho$. The three groups are in fact isomorphic \cite{Blackburn-McCourt}. In \cite{Kalman13_Tutte} the first author showed that $\rho$ is also equal to the number of \emph{hypertrees} in any of the six hypergraphs associated to the trinity. (A hypertree, essentially, is a vector that may arise as the degree sequence of a spanning tree, at those vertices of the bipartite graph that correspond to hyperedges in the hypergraph. See section \ref{sec:hyper_background}.) The same number $\rho$ is given by a determinant formula of Berman \cite{Berman80}. We will call $\rho$ the \emph{magic number} of the trinity.

In this paper we add to the list of questions which yield the magic number as their answer.

\begin{thm}
\label{thm:classification_of_tight_contact_structures}
The number of isotopy classes of tight contact structures on $(M_G, L_G)$ is equal to the number of hypertrees in either hypergraph of $G$.
\end{thm}

Thus, each of the three graphs in a trinity yields a sutured manifold with the same number of (isotopy classes of) tight contact structures, namely the magic number of the trinity. Our proof of theorem \ref{thm:classification_of_tight_contact_structures} provides an explicit construction of each tight contact structure via spanning trees. 

In \cite{Postnikov09} Postnikov proved that the set of hypertrees in a hypergraph equals the set of lattice points of a convex polytope. We will show directly that the Euler classes of tight contact structures on $(M_G, L_G)$ are equivalent to hypertrees in $(E,R)$ in a strong sense. Here $E$ is one of the two colour classes of $G$ (the other is $V$) and $R$ is the set of its regions. The pair $(E,R)$ is one of the six hypergraphs of the trinity, with $R$ as its set of hyperedges.

\begin{thm}
\label{cor:hypertree_Euler_bijection}
There is an affine bijection between hypertrees of $(E,R)$ and Euler classes of tight contact structures on $(M_G, L_G)$.
\end{thm}

In \cite{Juhasz-Kalman-Rasmussen12}, the first author together with Juh\'{a}sz and Rasmussen studied the sutured Floer homology of the manifolds $(M_G, L_G)$. They proved that $SFH(M_G, L_G)$ has a single $\Z$ summand at each \mbox{spin-c} structure in its support, and that the support is affine isomorphic to the set of hypertrees of $(E,R)$ (or of $(V,R)$). We will show that the spin-c structures of the tight contact structures on $(M_G, L_G)$ coincide with the support of $SFH$. Moreover, the contact invariants \cite{HKM09} of the tight contact structures essentially provide a basis for $SFH(M_G,L_G)$.

\begin{thm}
\label{thm:classification_of_contact_invariants}
Each isotopy class of tight contact structures on $(M_G, L_G)$ has a distinct spin-c structure. A tight contact structure $\xi_\s$ for the spin-c structure $\s$ exists exactly when $SFH(-M_G, -L_G, \s)\not\cong0$. In this case the contact invariant $c(\xi_\s)$ is $\{\pm x\}$, where $x$ generates $SFH(-M_G, -L_G, \s) \cong \Z$.
\end{thm}

Contact invariants in $SFH$ are only defined up to sign \cite{HKM08}, so this is in fact a complete description of them. This theorem quickly follows from theorem \ref{thm:classification_of_tight_contact_structures} and the TQFT property of sutured Floer homology \cite{HKM08}, because every tight contact structure on $(M_G, L_G)$ extends to the unique tight contact structure on $S^3$, cf.\ Proposition \ref{prop:inclusion_into_S3}.

Our results are related to knot theory in at least two ways. To describe the first way, we recall that the leading coefficient of the Alexander polynomial of $L_G$ is also given by the magic number $\rho$ of the trinity. This fact follows immediately from \cite[thm.\ 2]{Murasugi-Stoimenow03} and it is also a consequence of \cite[thm.\ 1.3]{K-Murakami}. A short proof can be given based on Kauffman's state expansion formula (cf.\ section \ref{sec:FKT} and proposition \ref{prop:alexander_magic}). We also mention the following corollary.

\begin{cor}
\label{cor:Alexander_magic}
The leading coefficients of the Alexander polynomials of the three special alternating links $L_{G_V}, L_{G_E}, L_{G_R}$ of a trinity are all equal, given by the number of tight contact structures on each of $(M_{G_V}, L_{G_V})$, $(M_{G_E}, L_{G_E})$, $(M_{G_R}, L_{G_R})$.
\end{cor}

The Homfly polynomials of the three links (via the identity $\Delta(t)=P(1,t^{1/2}-t^{-1/2})$, where $\Delta$ is the Alexander and $P$ is the Homfly polynomial) induce partitions of this coefficient which do not coincide, but each can be derived from the appropriate set of hypertrees using the \emph{interior polynomial} introduced by the first author in \cite{Kalman13_Tutte}. (See \cite[Corollary 1.2]{K-Postnikov}. The proof also uses results from \cite{K-Murakami}.) Hence theorem \ref{thm:classification_of_contact_invariants} establishes a direct connection between certain Homfly coefficients and tight contact structures on a naturally constructed sutured manifold, namely the complement of a minimal genus Seifert surface.

The second connection to knot theory is an application of theorem \ref{thm:classification_of_tight_contact_structures} to the \emph{formal knot theory} of Kauffman \cite{Kauffman_FKT83}. From a universe $\U$, we construct a bipartite plane graph $G_\U$ as shown in figure \ref{fig:figure_8_red_graph}. We relate the states of $\U$ to the contact topology of $(M_{G_\U}, L_{G_\U})$ to obtain the following result.

\begin{thm}
\label{thm:Kauffman_states}
The number of states of a universe $\U$ is equal to the number of isotopy classes of tight contact structures on $(M_{G_\U}, L_{G_\U})$.
\end{thm}

In particular, the number of states of $\U$ is the magic number of the trinity of $G_\U$. This follows relatively easily after noticing that the two Tait graphs (the black and white graphs constructed from a checkerboard colouring) of the universe are among the hypergraphs of the trinity. But we obtain our result not just as a combinatorial, but as a geometric correspondence: in fact, the Euler--Jordan trails describing the Kauffman states turn out to be the dividing curves of the corresponding contact structures.

\begin{figure}
\begin{center}
\input{red_graph.tikz.tex}
\input{median_construction.tikz.tex}
\end{center}
\caption{A bipartite plane graph $G$, and the construction yielding the link $L_G$ and surface $F_G$.}
\label{fig:median_construction}
\end{figure}
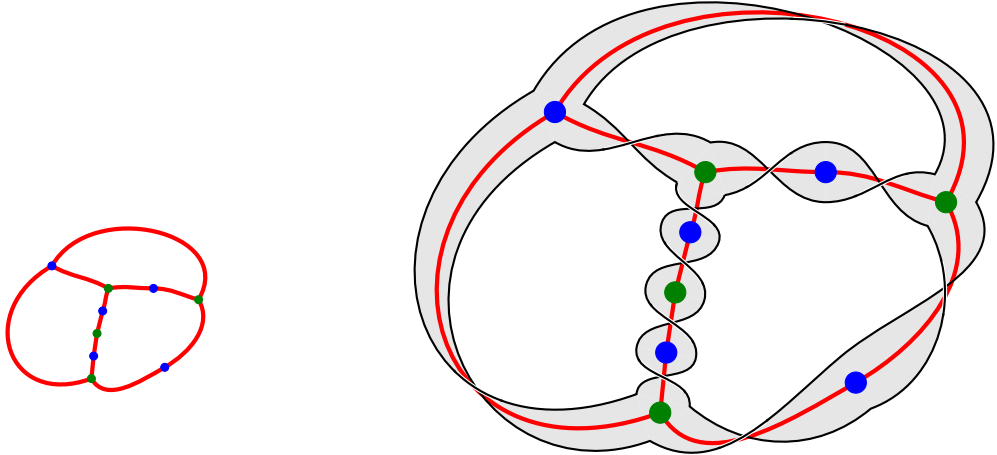

\begin{figure}
\begin{center}
\begin{minipage}{0.3\textwidth}
\input{fig-8_universe_small.tikz.tex}
\end{minipage}
\begin{minipage}{0.5\textwidth}
\input{fig-8_red_graph.tikz.tex}
\end{minipage}
\end{center}
\caption{A formal knot theory universe $\U$, and the bipartite plane graph $G_\U$ constructed from it.}
\label{fig:figure_8_red_graph}
\end{figure}
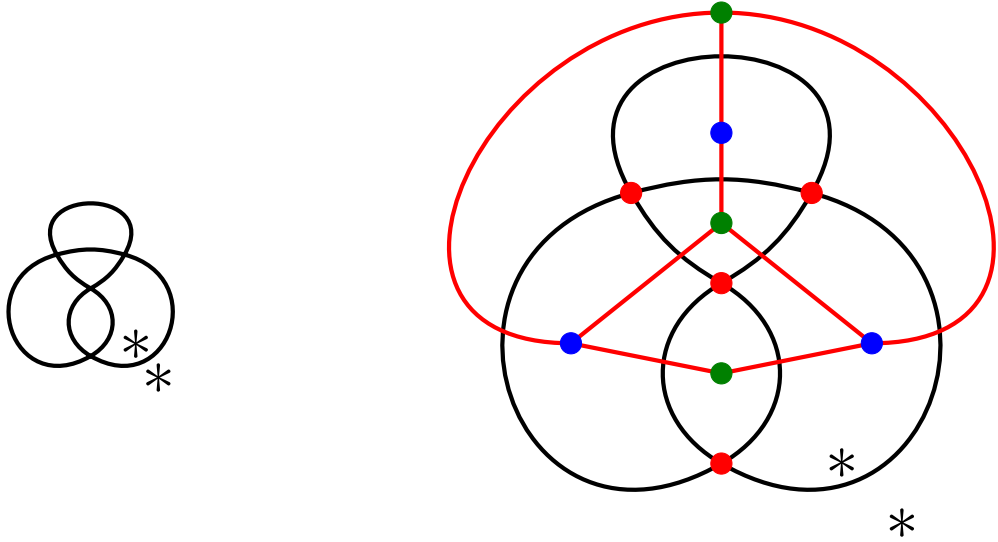

An exposition of some of the mathematics in this paper, and related background, is also available as \cite{Me17_polytopes}.

\subsection{Organisation of the paper}

In sections \ref{sec:contact_background}--\ref{sec:hyper_background} we recall some background from the disparate fields required: in section \ref{sec:contact_background}, contact topology, including sutured Floer homology and spin-c structures; in section \ref{sec:FKT}, formal knot theory, including universes and states; and in section \ref{sec:hyper_background} graph theory, including results on bipartite plane graphs, trinities, the magic number, hypergraphs, and hypertrees.

In section \ref{sec:trinities_SFH} we recall recent work of the first author and others which connects the combinatorics of trinities and hypertrees with sutured Floer homology of the associated sutured manifolds.

Section \ref{sec:contact_structures} is the main part of the paper and classifies the tight contact structures on the sutured manifold $(M_G, L_G)$. Using Giroux's theory of convex surfaces \cite{Gi91} and the gluing theorem of Honda \cite{Hon02}, among other work, we reduce the classification of contact structures to a combinatorial problem about dividing curves on discs in the complement $S^2\setminus G$. We apply results on bypass surgeries, hypertrees and arborescences to show that tight contact structures are bijective with hypertrees in an appropriate hypergraph.

Having classified contact structures on $(M_G, L_G)$, in section \ref{sec:properties_of_contact_structures} we investigate some further details. We compute their Euler classes in terms of the corresponding hypertrees, and show that all the contact structures include into the tight contact structure on $S^3$. Applying the TQFT property of $SFH$ \cite{HKM08}, we can then prove theorem \ref{thm:classification_of_contact_invariants}.

Finally, in section \ref{sec:alexander} we consider the Alexander polynomial, and we construct the bipartite plane graph of a formal knot theory universe. We observe that its states correspond precisely to configurations found in section \ref{sec:contact_structures}, which we then show are bijective with isotopy classes of tight contact structures, proving theorem \ref{thm:Kauffman_states}.

\subsection{Acknowledgments}

During the course of this work the first author was supported by a Japan Society for the Promotion of Science Grant-in-Aid for Young Scientists (B), no.\ 25800037. He was also the recipient of a ``Challenging Research Award'' by the Tokyo Institute of Technology which made it possible for the second author to visit Japan so that we could carry out this research. The second author is supported by Australian Research Council grant DP160103085.

\section{Contact topology background}
\label{sec:contact_background}

We briefly recall concepts we will need from 3-dimensional contact topology. We refer generally to \cite{Et02} and \cite{Geiges_Introduction} for an introduction to the subject.

\subsection{Contact 3-manifolds}

Let $M$ be a smooth oriented $3$-manifold. (Since a contact structure always induces an orientation, we lose no generality here.) A \emph{contact structure} $\xi$ on $M$ is a non-integrable 2-plane distribution on $M$. Such a pair $(M, \xi)$ is called a \emph{contact 3-manifold}. Locally, a contact structure is the kernel of a 1-form $\alpha$. We only consider contact structures which are (co-)orientable, i.e., in this paper all contact structures $\xi$ have globally-defined contact forms $\alpha$. Given a contact form $\alpha$, we can regard each contact plane as having an ``upward'' and ``downward''-facing side: a vector $X$ points ``upward'' or ``downward'' with respect to $\xi$ accordingly as $\alpha(X)$ is positive or negative.

The non-integrability of $\xi$ is equivalent to the non-vanishing of $\alpha \wedge d\alpha$. Any contact form $\alpha$ for $\xi$ has $\alpha \wedge d\alpha$ of the same constant sign, relative to the orientation. In this paper we consider only \emph{positive} contact structures, that is, those for which $\alpha \wedge d\alpha > 0$. 
Our choice of $\xi$ thus determines the orientation of $M$ and our choice of $\alpha$ determines a coorientation, and hence an orientation, for $\xi$. In the sequel we will usually assume that such structures have been fixed.

A smooth curve $C$ in $(M, \xi)$ is \emph{Legendrian} if it is everywhere tangent to $\xi$. The contact planes along $C$ define a framing of $C$. If $C$ is homologically trivial then it bounds an orientable surface $S$; this surface defines another framing of $C$. The \emph{Thurston--Bennequin number} of $C$, denoted by $\tb(C)$, is the twisting of the contact framing relative to the surface framing. More specifically, if $C'$ is a pushoff of $C$ along a vector field transverse to $\xi$, then $\tb(C)$ is given by the linking number $\lk(C,C')$. The Thurston--Bennequin number does not depend on the choice of surface or pushoff or orientation of $C$.

An \emph{overtwisted disc} is a 2-dimensional disc embedded in $M$ with Legendrian boundary of Thurston--Bennequin number $0$. A contact 3-manifold or contact structure containing an overtwisted disc is called \emph{overtwisted}. Eliashberg in \cite{ElOT} showed that the classification of overtwisted contact structures on a 3-manifold $M$ is equivalent to the classification of 2-plane fields on $M$ up to homotopy. A contact structure which is not overtwisted is called \emph{tight}. 

A smoothly embedded surface $S$ in $M$ has a singular 1-dimensional foliation defined by $TS \cap \xi|_S$, called its \emph{characteristic foliation}. (Note that it is impossible for a surface to be everywhere tangent to $\xi$.) For a generic surface, the singularities of this foliation are isolated points of two possible types: elliptic (where leaves spiral into the singular point) and hyperbolic (saddle points). The characteristic foliation on $S$ determines the germ of the contact structure near $S$ \cite{Gi91}.  

As an oriented 2-plane bundle on $M$, the contact structure $\xi$ has an \emph{Euler class} $e(\xi) \in H^2(M)$. If $S$ is an oriented embedded closed surface in $M$, then $e(\xi)$ evaluates on $S$ and we write $e(\xi)[S] \in \Z$. The Euler class is the obstruction to the existence of a non-vanishing section of $\xi|_S$. It evaluates to zero on the boundary: 
\begin{equation}
\label{eq:eulerboundary}
e(\xi)[\partial M] = \langle e(\xi), [\partial M] \rangle = \langle \delta e(\xi), [M] \rangle = 0.
\end{equation} 
Hence $\xi|_{\partial M}$ is a trivial bundle over $\partial M$; let $s$ be a nowhere-vanishing section. Then we also have the \emph{relative Euler class} $e(\xi, s) \in H^2 (M, \partial M)$. The map $H^2 (M, \partial M) \To H^2 (M)$ sends $e(\xi, s)$ to $e(\xi)$.  If $S$ is an oriented properly embedded surface, hence with $\partial S \subset \partial M$, then $e(\xi,s)[S]$ is the obstruction to extending the section $s|_{\partial S}$ to a section of $\xi|_S$ without zeroes. When the section $s$ is understood we write $e(\xi)$ instead of $e(\xi, s)$.

\subsection{Convex surfaces} 
\label{sec:convex_surfaces}

Convex surfaces are crucial to the contact geometry in this paper. We refer to \cite{Gi91, Hon00I, Hon02} for background on the subject. Let $(M,\xi)$ be a contact manifold.

A vector field on $M$ is called a \emph{contact vector field} if its flow preserves $\xi$. Given a contact form $\alpha$, contact vector fields on $M$ are naturally in bijective correspondence with smooth functions $H\colon M \To \R$: for each such $H$, there is a unique contact vector field $X_H$ such that $\alpha(X_H) = H$ (see, e.g., \cite[sec.\ 2.3]{Geiges_Introduction}). Consequently, any contact vector field defined on a submanifold of $M$ extends over $M$.

A closed surface $S$ smoothly embedded in $M$ is called \emph{convex} if there is a contact vector field on $M$ which is transverse to $S$. Equivalently, by the extensibility of contact vector fields, $S$ is convex if there is a contact vector field, defined in a neighbourhood of $S$, which is transverse to $S$. If $S$ has boundary, we require it to be Legendrian: $S$ is convex if its boundary is Legendrian and there is a contact vector field transverse to $S$. Any convex surface is coorientable, hence orientable. The contact vector field transverse to $S$ defines a homeomorphism from a neighbourhood of $S$ in $M$ to $S \times \R$, such that $S$ maps to $S \times \{0\}$ and $\xi$ is invariant under translations in the $\R$ direction. 

A convex surface $S$ in $(M, \xi)$ with transverse contact vector field $X$ has a \emph{dividing set} $\Gamma$, which is defined to be the set of points on $S$ where $X$ is tangent to $\xi$; equivalently, if $\alpha$ is a contact form, it is the set of points on $S$ where $\alpha(X) = 0$. In symbols,
\[
\Gamma = \{x \in S \mid X_x \in \xi_x \} 
= \{x \in S \mid \alpha_x (X_x) = 0\}.
\]
It can be shown that $\Gamma$ is a properly embedded 1-dimensional submanifold of $S$, transverse to the characteristic foliation, and its isotopy class does not depend on the choice of $X$ \cite{Gi91}. If $S$ has nonempty 
boundary, then $\partial S$ consists of several leaves (possibly with singularities) of the characteristic foliation, and $\Gamma$ intersects $\partial S$ transversely. At a point $x$ of $S$ not on $\Gamma$, the vector $X_x$ is not tangent to $\xi_x$ and so $\alpha_x (X_x)$ is either positive or negative. Accordingly, we say $x$ lies in the \emph{positive region} $R_+$ or \emph{negative region} $R_-$ of $S$; together we will call $R_\pm$ the \emph{signed regions}. Thus
\[
S = R_+ \sqcup R_- \sqcup \Gamma, \quad
R_+ = \{ x \in S \mid \alpha_x (X_x) > 0 \}, \quad
R_- = \{ x \in S \mid \alpha_x (X_x) < 0 \}.
\]
Note that if we use the contact vector field $-X$ instead of $X$, or the contact form $-\alpha$ instead of $\alpha$, the regions $R_+$ and $R_-$ exchange roles. This ambiguity can be removed by fixing a coorientation for $S$ (and one for $\xi$, which we have already done), and requiring $X$ to point towards the positive side of $S$. Thus, for a cooriented convex surface in a cooriented convex structure, $R_+$, $R_-$, and $\Gamma$ are well-defined up to isotopy.

Away from $\Gamma$, that is for $x\in R_{\pm}$, the tangent plane $T_x S$ can be identified with $\xi_x$ via projection along $X_x$. In this way, the orientation of $\xi$ induces orientations on $R_\pm$. These in turn induce an orientation on $\Gamma$ (the same by $R_+$ and by $R_-$). With the coorientation conventions above, the orientation of $R_+$ as a subset of $S$ agrees with the orientation induced by $\xi$, whereas along $R_-$ the two are opposite.

Any embedded surface in a contact manifold can be approximated by convex surfaces. This involves isotoping the boundary so that it becomes Legendrian. On the other hand if there is a boundary that is already Legendrian and we wish to keep it fixed, then care needs to be taken. Suppose we have a compact, oriented, properly embedded surface $S$ in $M$, with Legendrian boundary all of whose components are such that the difference between the framings from the contact planes and from $S$ is non-positive.
Then the surface $S$ can be made convex by a $C^0$-small isotopy of a neighbourhood of its boundary (fixing $\partial S$), followed by a $C^\infty$-small isotopy of $S$ fixing this neighbourhood of the boundary (\cite[prop.\ 3.1]{Hon00I}, see also \cite[thm.\ 3.8]{KanTB}). The framing condition here is essential: note that, by equation (\ref{eq:tb_from_Gamma}) below, a Legendrian curve with positive Thurston--Bennequin number cannot be the boundary of a convex surface.

If $S$ is convex with the dividing set $\Gamma$ and signed regions $R_\pm$, then $\Gamma$ is not only transverse to the characteristic foliation $\F$: it \emph{divides} $\F$, in the sense that $\F$ can be directed by a vector field which dilates an area form on $R_+$ and contracts it on $R_-$. In particular, elliptic singularities in $R_+$ become sources while those in $R_-$ become sinks.
See \cite{Gi91} for details.

The dividing set on a convex surface $S$ ``essentially'' determines the contact structure near $S$. More precisely, first use the contact vector field $X$ to describe a neighbourhood of $S$ in $M$ as $S \times \R$, such that $S = S \times \{0\}$, the vector field $X$ points in the positive $\R$ direction, and $\xi$ is invariant under translations in the $\R$ direction. Let $\F$ be any singular $1$-dimensional foliation on $S$ divided by $\Gamma$. \emph{Giroux's Flexibility Theorem} says that there exists a $C^0$-small isotopy $\phi_t$ ($0\le t\le1$) of $S$ in $S \times \R$, starting from the identity $\phi_0\colon S \To S \times \{0\}$, at the end of which the characteristic foliation on $\phi_1 S$ 
is $\phi_1 \F$. This 
means that any characteristic foliation on $S$ divided by $\Gamma$ can be achieved by a small isotopy of $S$; this characteristic foliation then determines the germ of a contact structure along $S$.

Since any characteristic foliation divided by $\Gamma$ can be achieved by a small isotopy of $S$, many collections of curves on $S$ can be made into leaves of the characteristic foliation, hence Legendrian curves, via a small isotopy of $S$. The \emph{Legendrian realisation principle} \cite[thm.\ 3.7]{Hon00I} says precisely which collections of curves $C$ can be made Legendrian in this way. Namely, a properly embedded 1-submanifold $C$ of $S$ can be made Legendrian if and only if it is transverse to $\Gamma$, and \emph{nonisolating} in the sense that every connected component of $S \setminus (\Gamma \cup C)$ has boundary intersecting $\Gamma$. The nonisolating condition roughly means that no point is isolated from $\Gamma$ by $C$: any point can escape to $\Gamma$ without crossing $C$.

As noted in \cite{EH05_Cabling}, the proof of the Legendrian realisation principle in \cite{Hon00I} applies equally to embedded \emph{graphs} in $S$. Namely, if a graph $G$ embedded in $S$ is nonisolating (meaning, again, that every component of $S \setminus (\Gamma \cup G)$ has boundary intersecting $\Gamma$), then $G$ can be made Legendrian via a small isotopy of $S$. Obviously, $G$ fails to be a $1$-manifold at each vertex of degree $3$ or more, so the vertices become singularities of the characteristic foliation. (In fact, even in the case of Legendrian realisations of smooth curves there are in general singularities of the characteristic foliation along each curve.)

A closed Legendrian curve $C$ on a convex surface $S$ has two natural framings: one from $\xi$, and one from $S$ or, equivalently, the contact vector field $X$. When $C$ bounds a subsurface of $S$, the twisting of the former with respect to the latter is $\tb(C)$. As we proceed along $C$, the contact planes (always tangent to $C$) spin so that they are tangent to $X$ at the points of $C \cap \Gamma$. Furthermore, since the flow of $X$ preserves $\xi$ and $\xi$ induces the orientation of $M$ (recall that we assumed $\xi$ to be a positive contact structure), at all these instances the spinning of $\xi$, relative to $X$, is in the same direction. The two framings will coincide at every other point of $C \cap \Gamma$ (and will be opposite at the rest).
Hence the twisting of $\xi$ with respect to $S$ is given in absolute value by $\frac{1}{2} | C \cap \Gamma|$, and it is not hard to check that it is in fact $-\frac{1}{2} |C \cap \Gamma|$. By the same token, for the Legendrian boundary $\partial S$ of the convex surface $S$ with dividing set $\Gamma$, we have 
\begin{equation}\label{eq:tb_from_Gamma}
\tb(\partial S) = -\frac{1}{2} |\partial S \cap \Gamma|.
\end{equation}

Since $\Gamma$ essentially determines the contact structure near $S$, we can ask whether this contact structure near $S$ is tight. \emph{Giroux's criterion} gives us a precise answer in terms of the dividing set $\Gamma$:
\begin{itemize}
\item if $S$ is a sphere, then $S$ has a tight neighbourhood if and only if $\Gamma$ is connected;
\item otherwise, $S$ has a tight neighbourhood if and only if $\Gamma$ has no contractible curves.
\end{itemize}
The second case covers all convex surfaces (with or without boundary) other than spheres, including closed surfaces with positive genus. In either case, when the criterion fails, one can Legendrian realise a contractible curve parallel to a contractible dividing curve (possibly after applying the ``folding'' technique of \cite[sec.\ 5.3]{Hon00I} if necessary), resulting in an overtwisted disc.

If $S$ is an oriented embedded closed convex surface in $(M, \xi)$ with dividing set $\Gamma$ and signed regions $R_\pm$, then the evaluation of the Euler class $e(\xi)$ on the homology class of $S$ is given by
\begin{equation}
\label{eq:eulerclass}
e(\xi)[S] = \chi(R_+) - \chi(R_-),
\end{equation}
where $\chi$ denotes Euler characteristic. More generally, consider an oriented properly embedded convex surface $S$, with Legendrian boundary $\partial S \subset \partial M$. Then $\partial S$ is an oriented curve tangent to $\xi$, and so tangent vectors to $\partial S$ give a natural (homotopy class of) nowhere-vanishing section $s_\partial$ of $\xi|_{\partial S}$. If we take a nowhere-vanishing section $s$ of $\xi|_{\partial M}$ which agrees with this section $s_\partial$ over $\partial S$ (i.e., $s|_{\partial S} = s_\partial$), then similarly, the evaluation of $e(\xi, s)$ on $[S]$ is given by
\[
e(\xi, s)[S] = \chi(R_+) - \chi(R_-).
\]
Note that such an extension $s$ of $s_\partial$ over $\partial M$ need not always exist. However, if $\partial S$ is a non-separating curve on $\partial M$, then such an extension $s$ does always exist. In this case, cutting $\partial M$ along the non-separating $\partial S$ gives a connected surface $U$ with an even number of boundary components, and the section $s_\partial$ gives a nowhere-vanishing section of $\xi|_{\partial U}$ which extends to a nowhere-vanishing section of $\xi|_U$ (the obstructions to extensions from pairs of boundary components cancel); gluing $U$ back into $\partial S$ we obtain the desired section $s$. This will be sufficient for our purposes. See \cite[sec.\ 4.2]{Hon00I} (also \cite{KanTB}) for details. More generally, if a nowhere-vanishing section $s$ of $\xi|_{\partial M}$ differs over $\partial S$ by $k$ full twists from $s_\partial$, then $e(\xi,s)$ differs from $\chi(R_+) - \chi(R_-)$ by $2k$ (see also section \ref{sec:spin-c_structures} below).

\subsection{Contact structures on sutured manifolds}
\label{sec:sutured_manifolds}

Rather than starting from a contact structure and considering convex surfaces with their dividing sets and signed regions, we can start instead from surfaces and their data of $\Gamma, R_\pm$, and consider contact structures related to them.

Suppose we are given $(S, \Gamma, R_+, R_-)$, where $S$ is a smooth compact oriented surface (with or without boundary), $\Gamma$ is a properly embedded smooth oriented 1-submanifold of $S$, 
and every component of $S$ contains at least one component of $\Gamma$.
Suppose further that $R_+$ (resp.\ $R_-$) is a sub-surface of $S$ oriented the same as (resp.\ opposite to) $S$, such that $S \setminus \Gamma = R_+ \sqcup R_-$ and $\Gamma \subset \partial R_+, \partial R_-$ as oriented 1-manifolds. Clearly not every smooth oriented 1-submanifold $\Gamma$ of $S$ has such an $R_+$ and $R_-$; the existence of $R_+$ and $R_-$ places strong restrictions on $\Gamma$. However, when such $R_+$ and $R_-$ do exist they can be deduced from the orientations on $S$ and $\Gamma$. We call $(S, \Gamma)$ a \emph{sutured surface}. 

The structure of a sutured surface, being identical to the structure of a dividing set, thus describes a contact structure in a product neighbourhood of $S$. See \cite{HKM02} for details.

We define a \emph{sutured 3-manifold} $(M, \Gamma)$ to be a smooth, oriented, compact 3-manifold $M$, together with an oriented 1-submanifold $\Gamma$ of $\partial M$ such that $(\partial M, \Gamma)$ is a sutured surface. So a sutured structure on $M$ provides boundary conditions for a contact structure on $M$. The notion of sutured 3-manifold originated with Gabai's study of foliations on 3-manifolds \cite{Gabai83}, though his definition was slightly different.

Given a sutured 3-manifold $(M, \Gamma)$ one has a contact structure $\xi_\partial$ defined near $\partial M$ such that $\partial M$ is convex and $\Gamma$ is the dividing set. The orientations of $\Gamma$ and $M$ determine a (co)orientation for $\xi_\partial$.
One can then ask how many cooriented, tight contact structures exist on $M$ extending $\xi_\partial$, up to isotopy fixing the boundary. This is the problem of \emph{classification of tight contact structures} on $(M, \Gamma)$. 

Of course, if the dividing set $\Gamma$ on $\partial M$ fails Giroux's criterion, then the contact structure $\xi_\partial$ near $\partial M$ has an overtwisted disc, so there are no tight contact structures near $\partial M$, let alone on $M$. Thus, we may assume $\Gamma$ satisfies Giroux's criterion and $\xi_\partial$ is tight. 
Another natural assumption is that $\chi(R_+)=\chi(R_-)$; indeed if there is any extension of $\xi_\partial$, then this follows by (\ref{eq:eulerboundary}) and (\ref{eq:eulerclass}) above. Sutured manifolds with $\chi(R_+)=\chi(R_-)$ are called \emph{balanced}.

Classifications of tight contact structures are known for some sutured 3-manifolds. The most fundamental result, due to Eliashberg, provides a classification for the 3-ball. When $(M, \Gamma)$ is a 3-ball, $M=B^3$, with a connected dividing set (so that Giroux's criterion is satisfied and $\xi_\partial$ is tight), Eliashberg showed there is a unique extension of $\xi_\partial$ to a tight contact structure on $B^3$ (up to isotopy relative to the boundary): see \cite[thm.\ 2.1.3]{ElMartinet}, also \cite[thm.\ 4.1]{Hon00I}. When $\Gamma$ is disconnected, Giroux's criterion fails, $\xi_\partial$ is overtwisted and there are no tight contact structures on $(B^3, \Gamma)$.

In a similar vein, Eliashberg proved that there is a unique (\emph{standard}) tight contact structure $\xi$ on $S^3$, up to isotopy \cite[thm.\ 2.1.1]{ElMartinet}. We take a ball neighbourhood of a point $p \in S^3$; by Darboux's theorem $p$ has a standard neighbourhood, which is tight. We can take the ball to have convex boundary $S$, which must then have connected dividing set. Now $S$ bounds balls on either side, both tight. Thus, the standard tight contact structure can be obtained by gluing together two balls bounded by convex spheres with connected dividing sets.

More generally, contact 3-manifolds can be decomposed into simpler 3-manifolds, eventually into 3-balls, by cutting them along convex surfaces: this is the idea of the \emph{convex decomposition theory} of Honda--Kazez--Mati\'{c} \cite{HKM02}. This theory parallels the theory of sutured manifolds, producing a sutured manifold hierarchy as in \cite{Gabai83}; it is closely connected to the theory of taut foliations.

This process can also be reversed, so that a general contact 3-manifold can be obtained by gluing together simple 3-manifolds (such as 3-balls) with convex boundary.

Since it will be useful, we state precisely how this gluing works. Let $(M, \xi)$ be a contact 3-manifold with convex boundary; let $\Gamma$ be the dividing set on $\partial M$ and $R_\pm$ the signed regions. Let $U_1, U_2$ be disjoint subsurfaces of $\partial M$ with (possibly empty) Legendrian boundary, and let $\Gamma_i = U_i \cap \Gamma$, $R_{\pm,i} = U_i \cap R_\pm$ for $i =1,2$. Suppose there is a homeomorphism $\phi\colon U_1 \To U_2$ such that $\phi(\Gamma_1) = - \Gamma_2$ and $\phi(R_{\pm,0}) = -R_{\mp,1}$. (The minus signs refer to reversal of orientation.) Then we may glue $U_1$ to $U_2$ using $\phi$ to obtain a new 3-manifold $M'$. If $U_1, U_2$ have empty boundary then we may use Giroux flexibility to obtain foliations $\F_1, \F_2$ on $U_1, U_2$ related by $\phi$; these define germs of contact structures which glue together to obtain a contact structure on $M'$. If $U_1, U_2$ have nonempty boundary then the same procedure applies, but we must first ``crease'' $\partial M$ along $\partial U_1$ and $\partial U_2$ so that upon gluing $M'$ has smooth boundary. (This yields interleaving dividing sets, as discussed in the next section). Note that the requirement that $U_1, U_2$ have Legendrian boundary means that $\partial U_1 \cup \partial U_2$ must be Legendrian realisable, that is, transverse to $\Gamma$ and nonisolating.

\subsection{Edge-rounding and bypasses}
\label{sec:edge_rounding_bypasses}

We will need to consider what happens when convex surfaces intersect transversely. Let $S_1, S_2$ be convex surfaces in the contact manifold $(M, \xi)$, with dividing sets $\Gamma_1, \Gamma_2$. Suppose $S_1$ and $S_2$ intersect transversely along a Legendrian curve $C$. (Hence $\Gamma_1$ and $\Gamma_2$ are both transverse to $C$.) As described in \cite[sec.\ 3.3.2]{Hon00I}, $S_1$ and $S_2$, together with their transverse contact vector fields, can be arranged so that the dividing sets $\Gamma_1$ and $\Gamma_2$ do not intersect but \emph{interleave} along $C$. (Note that if $C$ is closed, then this implies $|C \cap \Gamma_1| = |C \cap \Gamma_2|$. Since $S_1$ and $S_2$ induce the same framing on $C$, this also follows from the analysis in section \ref{sec:convex_surfaces}.)

If $C$ is part of the boundary of both $S_1$ and $S_2$, then we can consider $S_1 \cup S_2$ as a connected surface with a \emph{corner} along $C$. We may \emph{round the corner} at $C$ to obtain a smooth convex embedded surface in $M$, which is identical to $S_1 \cup S_2$ outside a small neighbourhood of $C$, and which in this neighbourhood has a standard form, including a standard transverse contact vector field. In particular, the dividing curves of $\Gamma_1$ and $\Gamma_2$ are rounded as shown in figure \ref{fig:a1}. The procedure works just as well in reverse: we may ``crease'' a convex surface, to obtain a surface with a corner with equivalent contact topology.

\begin{figure}
\begin{center}
\begin{tikzpicture}
\draw (0,0) node {\includegraphics[scale=0.4]{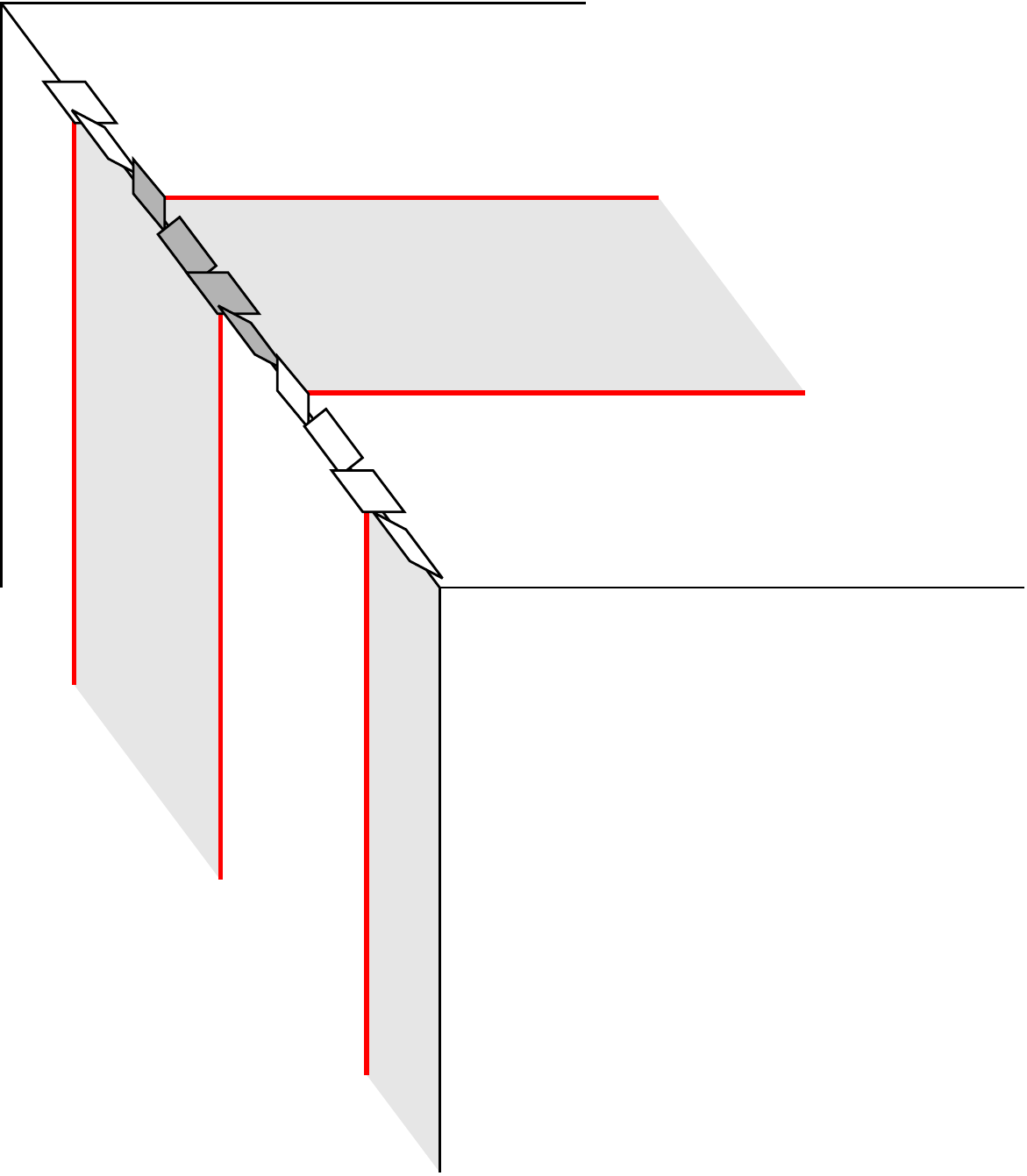}};
\draw [<->] (3,0) -- (4,0);
\draw (7,0) node {\includegraphics[scale=0.4]{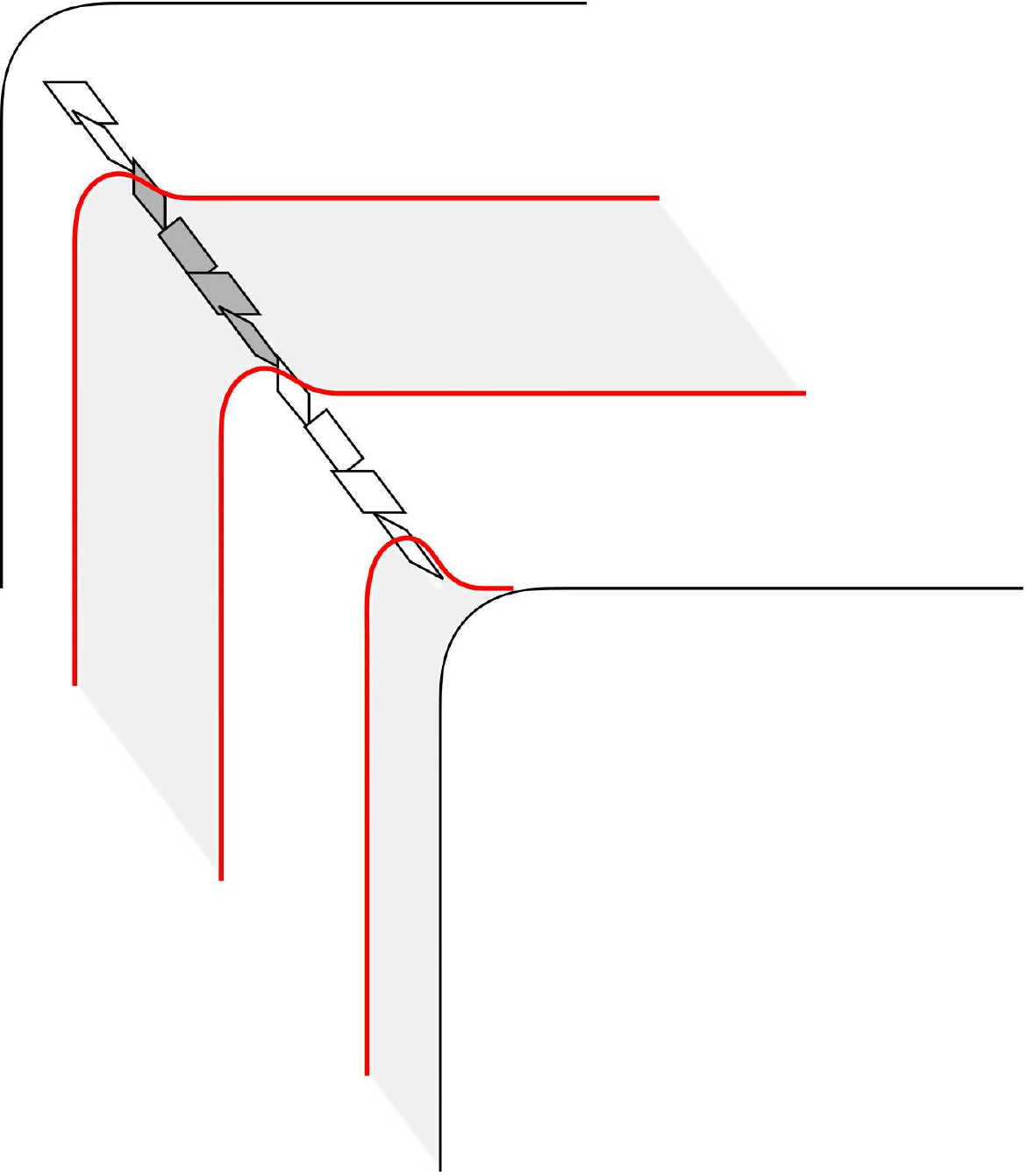}};
\end{tikzpicture}
\end{center}
\caption{Edge-rounding and creasing.}
\label{fig:a1}
\end{figure}

In a similar fashion, if $(M, \xi)$ is a contact $3$-manifold with convex boundary which is smooth, except for corners on $\partial M$ along closed Legendrian curves, we can apply a 
corner-rounding procedure and obtain a $3$-manifold $M'$ with 
boundary and a contact structure $\xi' = \xi|_{M'}$; as $\partial M'$ and $\partial M$ can be made arbitrarily close, there is a natural bijective correspondence between isotopy classes (relative to boundary) of contact structures on $M$ and on $M'$. Thus $M$ and $M'$ have equivalent contact topology.  

Motivated by these edge-rounding and creasing procedures, we can regard a surface with corners along closed curves, interleaving sutures along the creases, 
and appropriately defined positive and negative subsurfaces,
as a generalised sutured surface; we refer to this notion as a \emph{sutured surface with corners}. When $(\partial M, \Gamma)$ is a sutured surface with corners, we refer to $(M, \Gamma)$ as a \emph{sutured manifold with corners}. Upon performing edge-rounding, a sutured surface with corners becomes a bona fide sutured surface, and a sutured manifold with corners becomes a bona fide sutured manifold.

\begin{figure}[h]
\begin{center}
\includegraphics[scale=0.35]{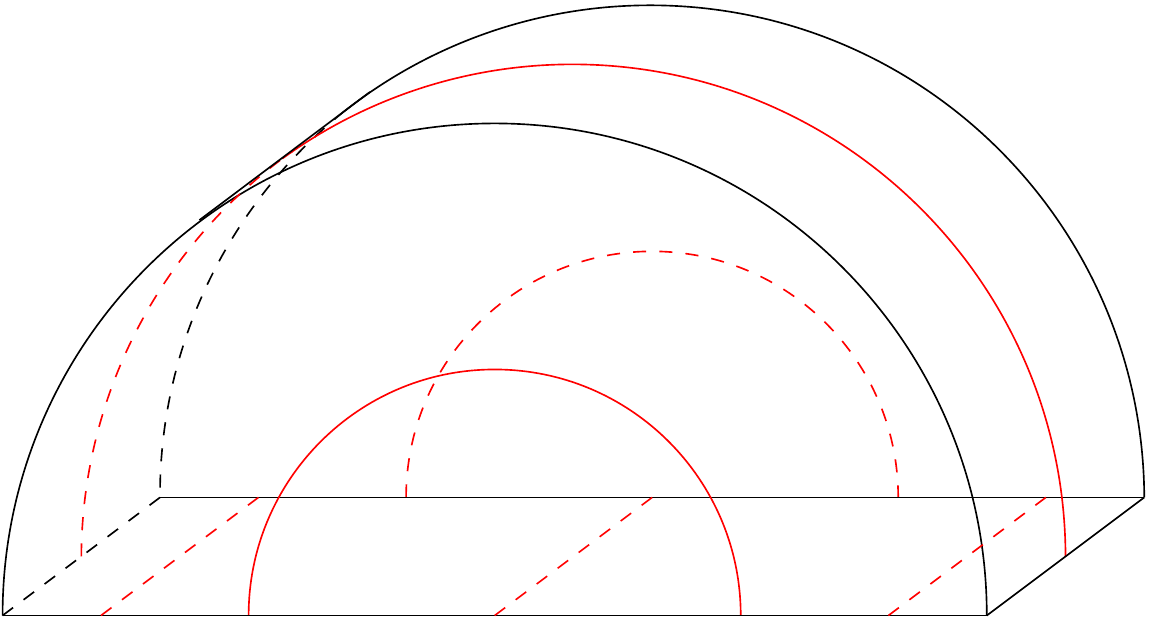}
\end{center}
\caption{A bypass}
\label{fig:a2}
\end{figure}

Another type of adjustment to contact structures and dividing sets comes from \emph{bypass addition}. A \emph{bypass} is a particular contact 3-manifold with boundary: it is half of a thickened overtwisted disc. More precisely, consider a convex overtwisted disc $D$ with dividing set consisting of a single closed loop $\gamma$ and Legendrian boundary with Thurston--Bennequin number $0$. Then $D \times [0,1]$ is a 3-manifold, with boundary $D \times \{0,1\} \cup \partial D \times [0,1]$, corners along $\partial D \times \{0,1\}$, and with a contact structure so that the dividing set is $\gamma \times \{0,1\} \cup \partial D \times \{1/2\}$. A bypass is obtained by slicing this object through a diameter times $[0,1]$. This ``sliced'' rectangle is the ``base'' of the bypass. On the base, the dividing set consists of three arcs of the form $\{\cdot\} \times [0,1]$. See figure \ref{fig:a2}.

\begin{figure}
\begin{center}
\begin{tikzpicture}
\draw (0,0) node {\includegraphics[scale=0.25]{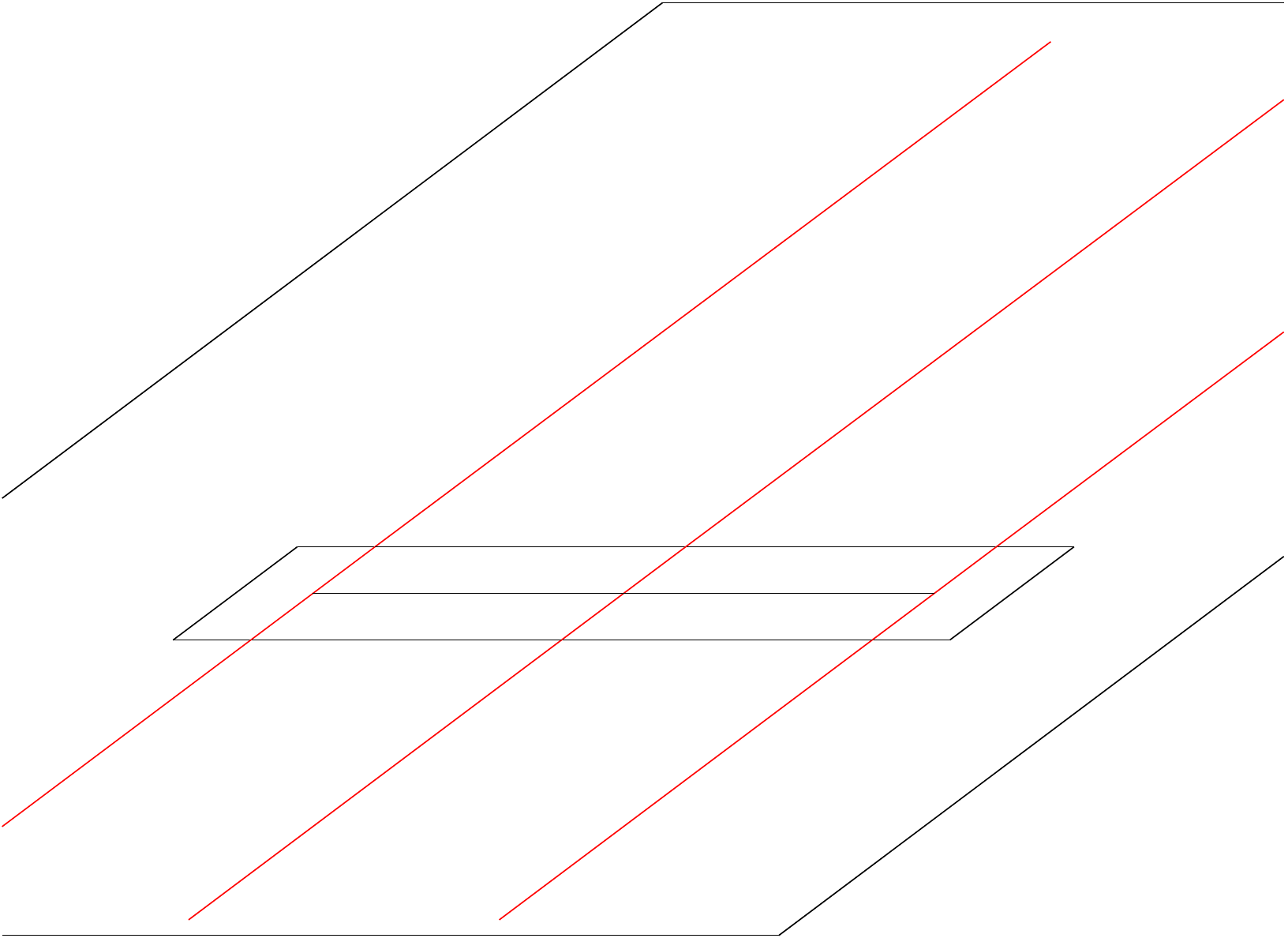}};
\draw [->] (3,0) -- (4,0);
\draw (7,0) node {\includegraphics[scale=0.25]{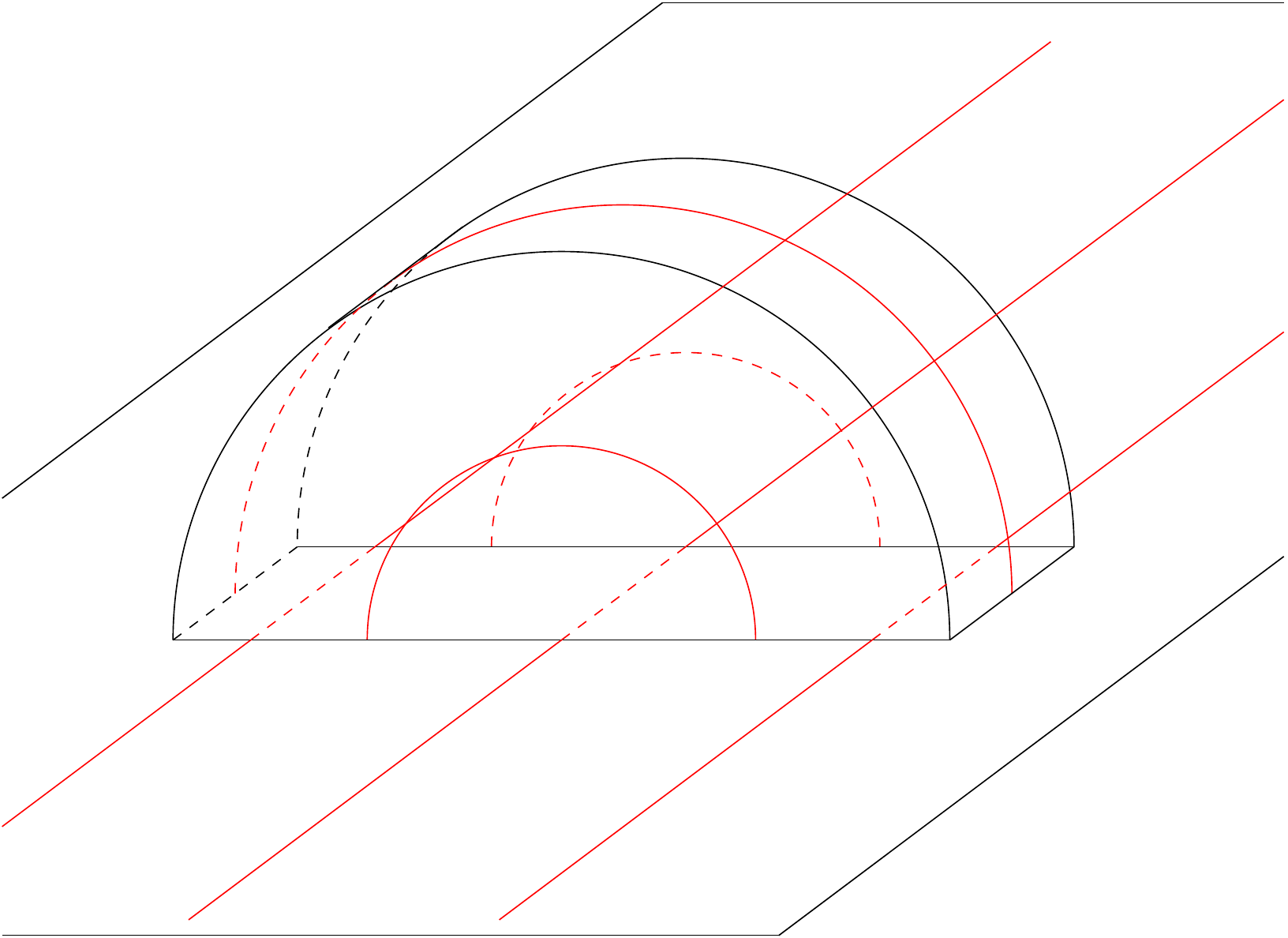}};
\draw [->] (7,-2) -- (7,-3);
\draw (7,-5) node {\includegraphics[scale=0.25]{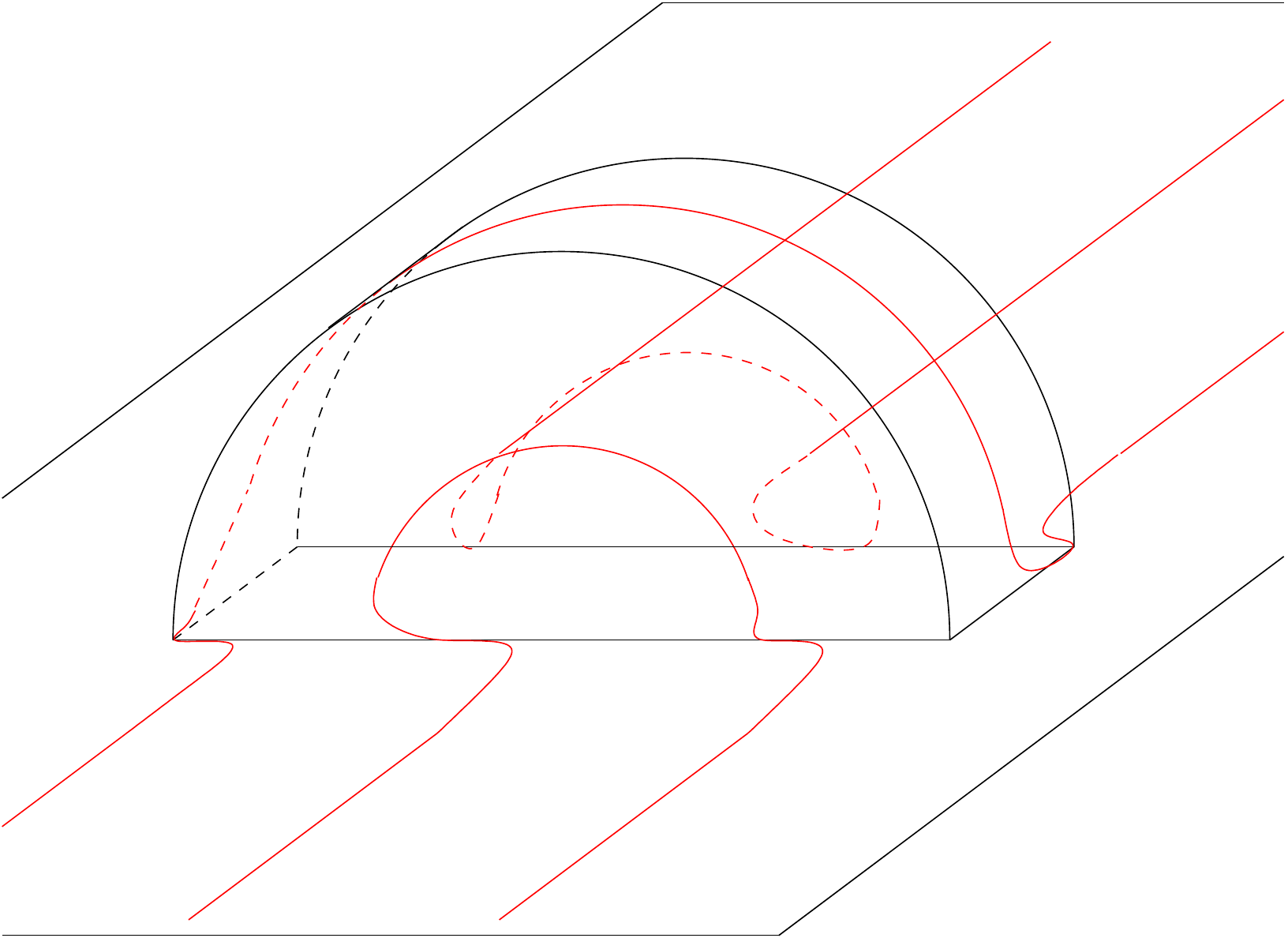}};
\draw [->] (4,-5) -- (3,-5);
\draw (0,-5) node {\includegraphics[scale=0.25]{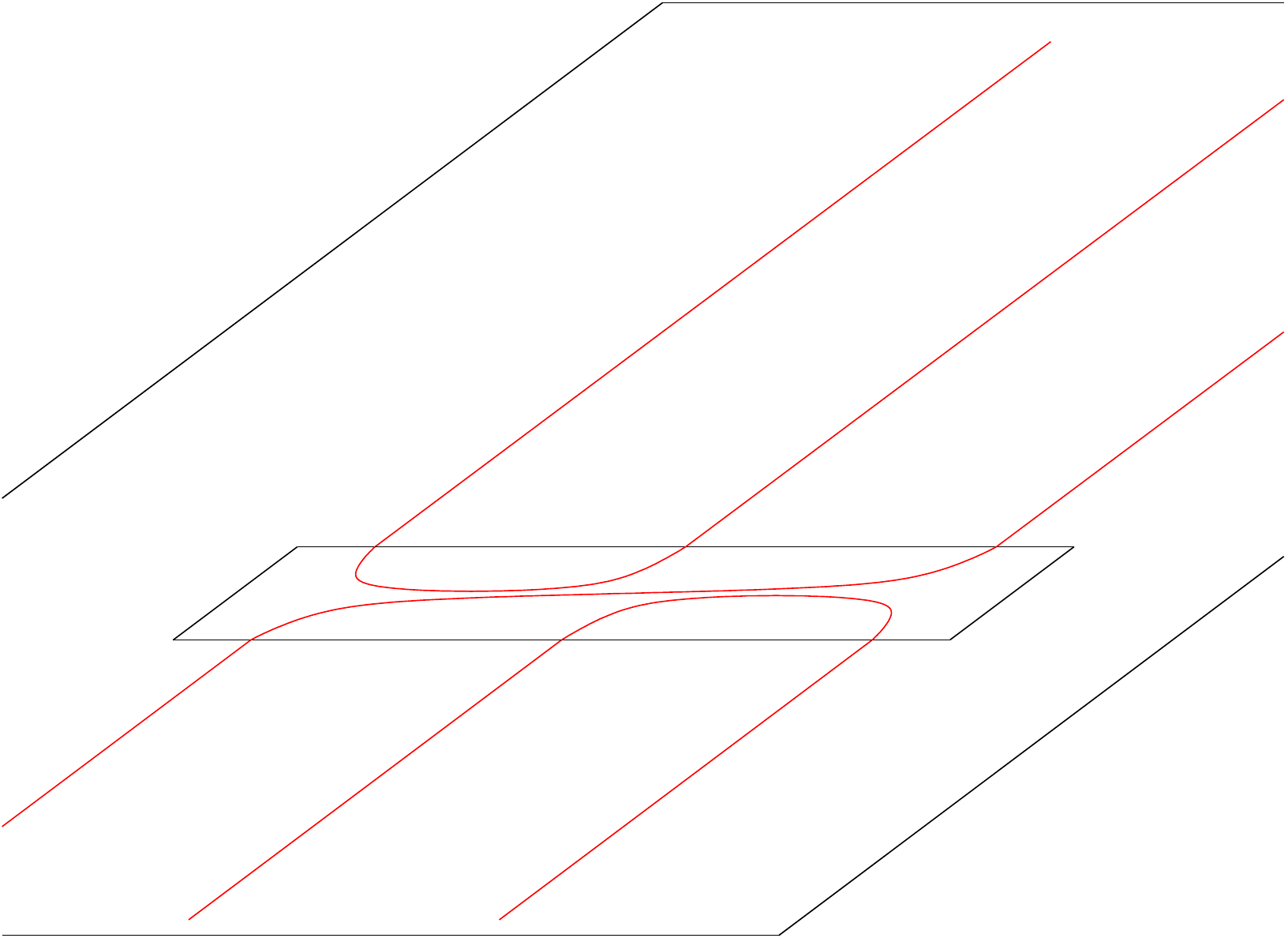}};
\end{tikzpicture}
\end{center}
\caption{Attaching a bypass.}
\label{fig:a2b}
\end{figure}
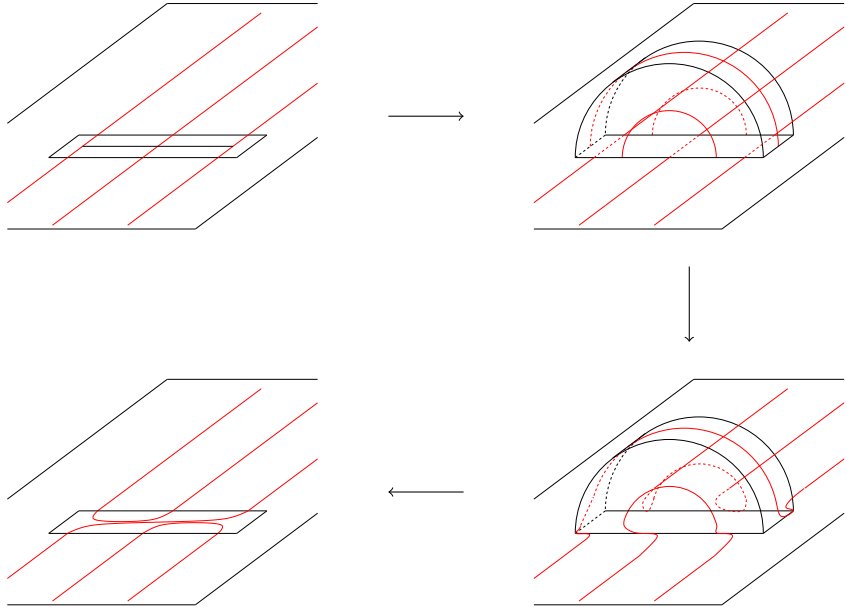

An \emph{attaching arc} on a convex surface $S$ is an embedded arc on $S$ which intersects the dividing set $\Gamma$ at precisely three points, namely its two endpoints and one interior point. An attaching arc has a neighbourhood in $S$ which contains precisely three arcs of $\Gamma$, and hence we can identify this neighbourhood with the base of a bypass. 

Thus, given a contact 3-manifold $(M, \xi)$ with convex boundary, we can attach a bypass to the boundary along an attaching arc $c$ on $\partial M$. The result is a slightly larger contact 3-manifold $(M', \xi')$. It has corners on its boundary, but these can be rounded so as to obtain a convex boundary: see figure \ref{fig:a2b}. As $M'$ is ``$M$ with a bump on the boundary'', clearly $M'$ is homeomorphic to $M$. However in general the two contact manifolds are not contactomorphic. The effect of a bypass attachment on the dividing set is shown in figures \ref{fig:a2b} and \ref{fig:a3} and we refer to it as \emph{outwards bypass surgery} along $c$. 

If one side of the surface (such as `upward' or `downward') and the attaching arc are specified, then that determines the effect of bypass surgery. In such cases we call the surgery \emph{upwards or downwards bypass surgery}, respectively.

Note how `attachment' is an operation on contact manifolds and `surgery' is one on sutured surfaces, and the latter does not determine the former. In other words, it is not enough to keep track only of dividing sets. For instance, if we perform the three operations of figure \ref{fig:a3} in sequence, the dividing set returns to its original configuration. However after two steps, the bypasses join to form an overtwisted disk. So the resulting contact structure is always overtwisted; in fact, in general it is not even homotopic to the original contact structure \cite{Huang14_bypass}. 

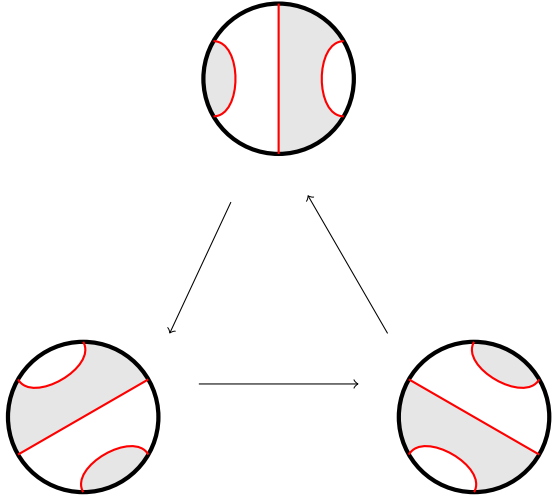
\begin{figure}
\begin{center}
\input{bypasstriple.tikz.tex}
\end{center}
\caption{Effect of bypass addition on the dividing set.}
\label{fig:a3}
\end{figure}

Also note that while bypass surgery certainly changes the dividing set $\Gamma$ into a new dividing set $\Gamma'$, the resulting $\Gamma'$ may be isotopic to $\Gamma$ after just one step. In this case the resulting $(M', \xi')$ is in fact contactomorphic to $(M, \xi)$ \cite[lem.\ 2.10]{Hon02} and we say the bypass is a \emph{trivial bypass}.

Bypass addition is, in a certain sense, the smallest nontrivial change that can be made to a contact 3-manifold. It does not change the topology of the manifold, 
but it can change the isotopy class of the contact structure.
 Indeed, the operation of bypass surgery on a dividing set is the simplest local change one can make to a dividing set which maintains the Euler class evaluation $\chi(R_+) - \chi(R_-)$ on the surface. 

Apart from adding bypasses to the exterior of a contact 3-manifold, we can also \emph{subtract} them. Let $c$ be an attaching arc on $\partial M$. We say that a bypass \emph{exists inside $M$ along $c$} if there exists a submanifold $(N, \xi|_N) \subset (M, \xi)$ which is contactomorphic to a bypass, via a contactomorphism which identifies the base of the bypass with $N \cap \partial M$, where $N \cap \partial M$ is a neighbourhood of $c$ on $\partial M$. That is, a bypass exists inside $M$ along $c$ if we may ``dig out'' a bypass, digging down from (a neighbourhood of) $c$ as the base of the bypass. 

If a bypass exists inside $M$ along $c$, it is unique in an appropriate sense, 
as we prove below. (It is certainly not new, but we have not seen an explicit proof in the literature.)

In general, given a contact 3-manifold $(M, \xi)$ with convex boundary, dividing set $\Gamma$ on $\partial M$, and attaching arc $c$, it may be difficult to tell whether a bypass exists inside $M$ along $c$. However, some things may be said: see \cite[sec.\ 3.2]{Me09Paper} 
for further discussion. We can easily describe what the dividing set would be after removing the bypass: it is the result of performing the inverse operation of outwards bypass surgery, which we call \emph{inwards bypass surgery} along $c$. If a bypass exists in $M$ along $c$, and then we attach (outside $M$) a bypass along $c$, then the two bypasses join to give an overtwisted disc. If $\xi$ is tight, and the result of inwards bypass surgery along $c$ is a dividing set which fails Giroux's criterion, hence with an overtwisted neighbourhood, then no bypass exists inside $M$ along $c$.

On the other hand, if performing inwards bypass surgery along $c$ results in a dividing set $\Gamma'$ isotopic to the original $\Gamma$, then a bypass exists and removing the bypass gives a contact 3-manifold $(M', \xi')$ contactomorphic to the original $(M, \xi)$. Indeed, in this case $(M, \xi)$ is obtained from $(M', \xi')$ by a trivial bypass attachment. Thus, ``trivial bypasses always exist''. Honda--Kazez--Mati\'{c} flippantly refer to this idea as the ``right to life'' principle (see \cite[lem.\ 2.9]{Hon02}, \cite[prop.\ 2.2]{HKM01}).

We now prove the uniqueness of bypasses existing in $M$ along an attaching arc.
\begin{lem}
\label{lem:unique_bypass}
Let $c,c'$ be isotopic attaching arcs (through other attaching arcs) on the boundary of a contact manifold $(M, \xi)$. Then a bypass $B$ exists inside $M$ along $c$ if and only if a bypass $B'$ exists along $c'$. Moreover, $M \setminus B$ and $M \setminus B'$ are contactomorphic.
\end{lem}

\begin{proof}
First suppose $c,c'$ are disjoint, and the bypass $B$ exists along $c$.  After removing $B$, it is easy to check that the attaching arc $c'$ becomes trivial, so by the right-to-life principle a trivial bypass $B'$ exists along $c'$ in $M \setminus B$. In particular, $B'$ is disjoint from $B$. Similarly, after removing $B'$ from $M$, the bypass $B$ becomes trivial. Thus $M \setminus B$, $M \setminus B'$ and $M \setminus (B \cup B')$ are contactomorphic, and the result holds in this case.

If $c$ and $c'$ are not disjoint, we may take a third attaching arc $c''$ disjoint from both, and apply the previous paragraph to the pairs $(c,c'')$ and $(c',c'')$.
\end{proof}

One useful application of these ideas is when our contact 3-manifold is a tight 3-ball $B$, with connected dividing set $\Gamma$ on $\partial B$. Any attaching arc $c$ on $\partial B$ must be in one of the two configurations shown in figure \ref{fig:a4}. In the first case, a bypass added along $c$ produces a disconnected dividing set, hence an overtwisted contact structure, but a bypass exists inside $B$ along $c$ by the right-to-life principle. In the second case, we obtain precisely opposite results: a bypass added along $c$ is a trivial bypass, so results in a tight contact structure, but no bypass exists inside $B$ along $c$, as inwards bypass surgery produces a disconnected dividing set. We call the attaching arc \emph{inner} or \emph{outer} accordingly.

\begin{figure}
\begin{center}
\includegraphics[scale=0.3]{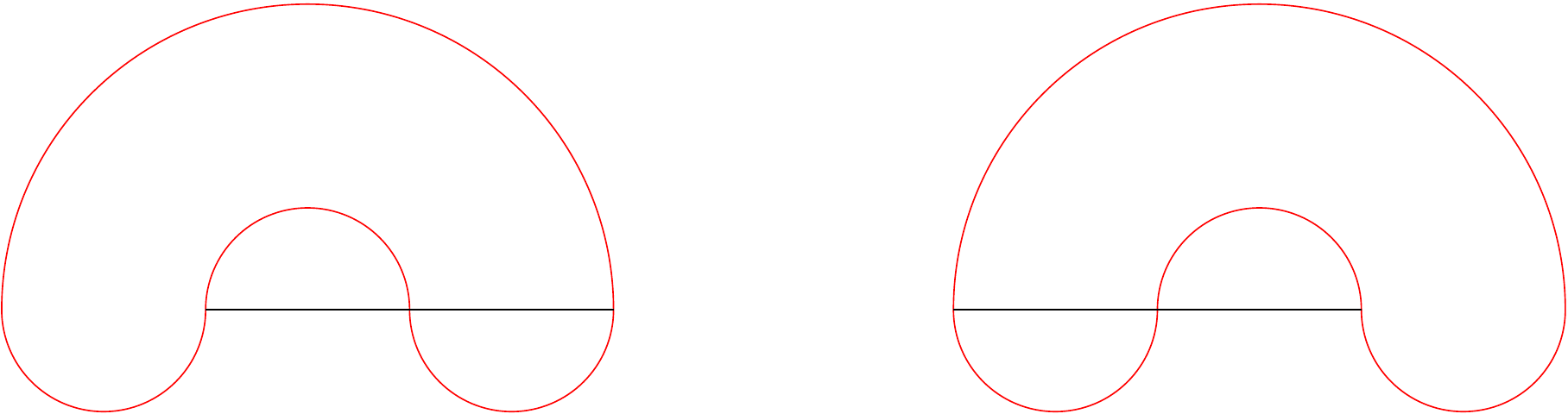}
\end{center}
\caption{Inner (left) and outer (right) attaching arcs on a tight $\partial B^3$.}
\label{fig:a4}
\end{figure}

\subsection{Honda's theorem} 
\label{sec:gluing_tight_contact_structures}

From the above discussion follows a general principle for classifying the tight contact structures on a sutured 3-manifold $(M, \Gamma)$, via decomposition along convex surfaces. This is Honda's theorem of \cite{Hon02}, which built on work of Kanda \cite{Kan97}, Torisu \cite{Torisu00} and Colin \cite{Colin97_chirurgies,Colin99_recollement}.

Roughly, the idea is as follows. Suppose we can cut $M$ along a properly embedded surface $S$ to obtain another, simpler 3-manifold $M'$. We assume that we have a complete understanding of the tight contact structures on $M'$, and use this understanding to give a classification of all tight contact structures on $M$. 

When we say we assume a ``complete" understanding of contact structures on $M'$, we really mean \emph{complete}! We need to know, for every possible set of sutures $\Gamma'$ on $\partial M'$: a complete classification of all tight contact structures $\xi'$ on $(M', \Gamma')$; for each such contact structure $\xi'$, the locations of all attaching arcs on $\partial M'$ where bypasses exist inside $M'$; and also for each such contact structure $\xi'$, the effect of adding a bypass along any attaching arc. The idea is that if we know all this information about $M'$, then this is enough to give us similar information for $M$. In particular, any tight contact structure $\xi$ on $(M, \Gamma)$ will restrict to a tight contact structure $\xi'$ on $(M', \Gamma')$ for some $\Gamma'$, and hence $\xi$ can be constructed from the fully-understood $\xi'$ on $(M', \Gamma')$. There may be several $\xi'$ on $(M', \Gamma')$, with various $\Gamma'$, which glue up to give contact structures isotopic to $\xi$ on $(M, \Gamma)$. But the idea is that if $\xi'_1, \xi'_2$ on $M'$ glue up to give contact isotopic contact structures $\xi_1, \xi_2$ on $M$, then, starting from $\xi_1$, we can move bypasses through $S$ in $M$, until we obtain $\xi_2$. And the effect of moving these bypasses is equivalent to bypass addition and subtraction on $M'$, which is fully understood.

Let us now be a bit more precise. Consider a sutured 3-manifold $(M, \Gamma)$. Then $\Gamma$ defines a contact structure $\xi_\partial$ near $\partial M$. We consider a compact, oriented, properly embedded surface $S$ in $M$, along which we wish to cut in the manner of convex decomposition theory as in section \ref{sec:sutured_manifolds}. Hence we require that $\partial S \subset \partial M$ satisfy the condition in the Legendrian realisation principle. In fact, we assume the slightly stronger condition that 
each component of
$\partial S$ intersects $\Gamma$ nontrivially.

Then, given any contact structure $\xi$ on $(M, \Gamma)$, we can make $S$ convex as described in section \ref{sec:convex_surfaces}; 
note that, after making $\partial S$ Legendrian, the non-positive (in fact, negative) twisting of $\xi$ along $\partial S$ follows from the previous assumption.

We consider the possible dividing sets $\Gamma_S$ which can arise on $S$. Any such $\Gamma_S$ must interleave with the dividing set $\Gamma$ along $\partial S$, as discussed in section \ref{sec:edge_rounding_bypasses}. The set of isotopy classes of such $\Gamma_S$ is countable. Let us fix one such isotopy class.

Cutting $M$ along $S$ produces a manifold with boundary and corners; the boundary consists of $\partial M$ and two copies of $S$; each copy of $S$ meets $\partial M$ along the corners. Considering sutures $\Gamma$ on $\partial M$ and $\Gamma_S$ on each copy of $S$, we obtain a sutured manifold with corners. Rounding the corners yields a sutured 3-manifold which we denote with $(M', \Gamma')$; so $\Gamma'$ is obtained from $\Gamma$ and the two copies of $\Gamma_S$ by edge-rounding.

By assumption, we can then enumerate the isotopy classes of tight contact structures on $(M', \Gamma')$. A contact structure $\xi'$ on $(M', \Gamma')$ determines a contact structure $\xi$ on $(M, \Gamma)$ by gluing the two copies of $S$ on $\partial M'$ back together as discussed in section \ref{sec:sutured_manifolds}. Moreover, every (isotopy class of) contact structure $\xi$ on $(M, \Gamma)$ arises by gluing up some $\xi'$ on some $(M', \Gamma')$ in this way.

\begin{defn}
\label{def:configuration}
Let $(M, \Gamma)$ be a sutured 3-manifold, so that $\Gamma$ defines a contact structure near $\partial M$. Let $S$ be an embedded 
surface in $M$ all of whose boundary components are Legendrian in $\partial M$. Then a \emph{configuration} on $(M, \Gamma, S)$ is a pair $(\Gamma_S, \xi')$, where
\begin{itemize}
\item 
$\Gamma_S$ is a set of curves on $S$, up to isotopy, whose endpoints interleave with $\Gamma \cap \partial S$ along $\partial S$, so that cutting $M$ along $S$, drawing sutures $\Gamma_S$ on both copies of $S$, and edge-rounding yields a sutured manifold $(M', \Gamma')$, and
\item $\xi'$ is an (isotopy class of) contact structure on $(M', \Gamma')$. 
\end{itemize}
We denote the set of all configurations on $(M, \Gamma, S)$ by $\Conf(M,\Gamma,S)$.
\end{defn}

Note that in a configuration $(\Gamma_S, \xi')$, neither $\Gamma_S$ nor the resulting $\Gamma'$ need to satisfy Giroux's criterion, and even if they do, $\xi'$ can be overtwisted.
On the other hand, if $(M, \Gamma)$ is balanced, then automatically so is $(M', \Gamma')$. This is because if a connected component of $S\setminus\Gamma_S$ belongs to $R_+$ along one copy of $S$, then the identical set along the other copy of $S$ belongs to $R_-$.
In order to discuss the classification of contact structures, we introduce the following notation.

\begin{defn}
Let $(M, \Gamma)$ be a sutured 3-manifold.
\begin{enumerate}
\item Denote by $\T(M, \Gamma)$ the set of isotopy classes of tight contact structures on $(M, \Gamma)$.
\item Define $\T_* (M, \Gamma) = \T(M, \Gamma) \cup \{*\}$, where $*$ is an extra element which we use to denote overtwisted contact structures.
\end{enumerate}
\end{defn}

A tight contact structure on $(M, \Gamma)$ gives an element of $\T(M, \Gamma)$, and any contact structure whatsoever on $(M, \Gamma)$  gives an element of $\T_* (M, \Gamma)$.

The above discussion constructs a map $F\colon \Conf(M, \Gamma, S) \To \T_* (M, \Gamma)$, which takes a configuration $(\Gamma_S, \xi')$ to the element of $\T_* (M, \Gamma)$ corresponding to the contact structure obtained by gluing together the two copies of $S$ on the boundary of $(M',\xi')$. The argument above shows that $F$ is surjective. But $F$ may not be injective: many configurations may lead to isotopic contact structures on $(M, \Gamma)$. 

By studying bypasses, we can find which configurations correspond to tight contact structures, and which configurations correspond to isotopic tight contact structures. For instance, suppose that $F(\Gamma_S, \xi') = \xi$, and a bypass exists in $(M', \xi')$ along an attaching arc $c$ on one of the two copies of $S$. Then if we remove this bypass from $M'$, and attach a bypass to the other copy of $c$ (on the other copy of $S$), we obtain another configuration $(\Gamma'_S, \eta')$. Here $\Gamma'_S$ differs from $\Gamma_S$ by a bypass surgery along $c$. This configuration gives a contact structure $\eta$ on $M$ isotopic to the original $\xi$: indeed, $\xi$ and $\eta$ differ only in that a bypass has been passed through $S$. Thus $F(\Gamma_S, \xi') = F(\Gamma'_S, \eta')$. We refer to this ``passing a bypass" as a ``state transition" between configurations, as defined below. 

\begin{defn}
\label{def:transition}
Let $(\Gamma_S, \eta$) and $(\Gamma'_S, \eta')$ be configurations on $(M, \Gamma, S)$. There is a \emph{state transition} $(\Gamma_S, \eta) \To (\Gamma'_S, \eta')$ if
\begin{enumerate}
\item $\eta'$ can be obtained from $\eta$ by removing a bypass from $\eta$ along an attaching arc $c$ on one copy of $S$, and attaching a bypass along the other copy of $c$ on the other copy of $S$, and
\item the effect of the bypass removal and attachment on the dividing set yields $\Gamma'_S$ on each copy of $S$.
\end{enumerate}
\end{defn}

A configuration $C = (\Gamma_S, \xi')$, where $\xi'$ is overtwisted, will definitely glue up to give an overtwisted contact structure on $M$; but if $\xi'$ is tight, gluing it up may give a tight or overtwisted contact structure on $M$. Hence we make the following definition.

\begin{defn}
A configuration $C = (\Gamma_S, \xi') \in \Conf(M, \Gamma, S)$ is \emph{potentially tight} if $\xi'$ is tight. We call $C$ \emph{overtwisted} if $\xi'$ is overtwisted.
\end{defn}

The essence of Honda's theorem is that (under certain natural conditions) state transitions effected by bypass operations are enough to connect any configurations representing the same contact structure on $(M, \Gamma)$, giving a precise meaning to the idea that bypass addition is the smallest nontrivial operation on a contact manifold. Formally, we make the following definitions.

\begin{defn}
\label{def:configuration_graph}
The \emph{configuration graph} $\GG(M, \Gamma, S)$ is a directed graph with vertices given by the configurations $C \in \Conf(M, \Gamma, S)$, and directed edges given by the state transitions.
\end{defn}

The graph $\GG(M, \Gamma, S)$ is in fact bidirected: a bypass that has been pushed across $S$ can be pushed right back. In other words, we may think of $\GG(M, \Gamma, S)$ as an undirected graph. Some vertices of $\GG(M, \Gamma, S)$ are potentially tight; we now define a notion of tightness that will correspond to the contact-topological notion.

\begin{defn} \
\label{def:tight_configs}
\begin{enumerate}
\item
A connected component of $\GG(M, \Gamma, S)$ is \emph{tight} if every configuration in the component is potentially tight. 
\item A configuration $C \in \Conf(M, \Gamma, S)$ is \emph{tight} if it lies in a tight component of $\GG(M, \Gamma, S)$.
\item Let $\GG_0(M, \Gamma, S)$ be the subgraph of $\GG(M, \Gamma, S)$ consisting of its tight components.
\end{enumerate}
\end{defn}

The statement of Honda's theorem is as follows.

\begin{thm}[Honda \cite{Hon02}]
\label{thm:Honda_gluing}
Let $(M, \Gamma)$ be an irreducible sutured 3-manifold, so that $\Gamma$ defines a contact structure near $\partial M$. Let $S$ be a properly embedded, incompressible surface in $M$, all of whose boundary components (possibly there are none) are Legendrian and intersect $\Gamma$ nontrivially. Then there is a bijection $\pi_0 (\GG_0(M, \Gamma, S)) \To \T(M, \Gamma)$, induced by the gluing map $F$ above.
\qed
\end{thm}

Here $\pi_0(\GG_0)$ refers to the set of connected components of $\GG_0(M, \Gamma,S)$, i.e., to the set of tight components of $\GG(M, \Gamma,S)$. 

The bijection in Honda's theorem is essentially a quotient of the map $F\colon \Conf(M, \Gamma, S) \To \T_* (M, \Gamma)$ defined above. Since two configurations related by a state transition give isotopic contact structures, $F$ descends to a map $\pi_0 (\GG(M,\Gamma_S)) \To \T_* (M, \Gamma)$. Overtwisted configurations map to $*$, as do any configurations connected to them by state transitions. A potentially tight configuration from which state transitions can only reach other potentially tight configurations yields a tight contact structure. Equivalently, if there is an overtwisted disc in a contact structure, we can eventually push it away from $S$ via bypass operations, reaching a configuration on $M'$ that is overtwisted. Any two configurations which correspond to isotopic contact structures on $(M, \Gamma)$ are related by bypass operations and an isotopy in $M'$. 

If we understand the contact topology of $M'$, then we can, at least in principle, construct the graph $\GG(M,\Gamma,S)$, and then the theorem provides an understanding of the contact topology of $M$.

An arbitrary sutured 3-manifold $(M, \Gamma)$ can be successively decomposed as above, until we arrive at a collection of balls. Thus, in principle at least, our understanding of the contact topology of $B^3$ can (eventually) give an understanding of the contact topology of $(M, \Gamma)$.

For us $M$ will always be a handlebody, $S$ a disjoint union of discs and $M'$ the disjoint union of two balls, $M' = M^+ \sqcup M^-$. As we have discussed, the contact topology of a ball $B$ is quite simple. Namely, if the dividing set $\Gamma$ on $\partial B$ is connected then there is a unique (isotopy class of) tight contact structure on $(B, \Gamma)$, and every attaching arc is inner or outer.

Thus, in this case, if $C = (\Gamma_S, \xi')$ is a potentially tight configuration, then $\Gamma_S$ must consist of a \emph{chord diagram} on each disc of $S$ (any closed curves would produce an overtwisted contact structure), such that, when corners are rounded, we obtain a connected dividing set on each of $\partial M^+$ and $\partial M^-$; and $\xi'$ must belong to the unique isotopy class of tight contact structures on $M^+$ and $M^-$. Since there is a unique such $\xi'$, up to isotopy, we can specify our configurations simply by $\Gamma_S$.

\subsection{Dividing sets in contact cylinders}
\label{sec:dividing_sets_cylinders}

It will be important later to consider some contact structures on a specific and simple family of sutured 3-manifolds: solid cylinders $D^2 \times I$, with ``vertical" sutures on the ``side" $\partial D^2 \times I$. It turns out that fixing sutures on the side, and considering possible dividing sets on the ``top" and ``bottom" discs $D^2 \times \{0,1\}$, leads to an interesting structure, studied by the second author in \cite{Me09Paper, Me10_Sutured_TQFT}.

More precisely, take the closed disc $D^2$, so $D^2 \times [0,1]$ is topologically a ball, with corners along $\partial D^2 \times \{0,1\}$. Consider sutures on this manifold, interleaving along the corners, as follows. Let $n$ be a positive integer and let $F$ be a set of $2n$ points on $\partial D^2$. Take sets of sutures $\Gamma_0, \Gamma_1$ respectively on $D^2 \times \{0\}, D^2 \times \{1\}$ with boundaries $F \times \{0\}, F \times \{1\}$. Take sutures on $\partial D^2 \times I$ consisting of $2n$ vertical curves $\{\cdot\} \times [0,1]$, interleaving with the curves of the $\Gamma_i$.

Following \cite{Me09Paper}, we denote this cylinder, a sutured manifold with corners, by $\M(\Gamma_0, \Gamma_1)$. As usual, the sutures define a contact structure $\xi_\partial$ near the boundary. 

If $\Gamma_0$ or $\Gamma_1$ contains a closed curve then $\xi_\partial$ is overtwisted; if $\Gamma_0, \Gamma_1$ contain no closed curves then they are chord diagrams, and $\xi_\partial$ may be tight or overtwisted. If, after edge rounding, the dividing set consists of a single closed curve, then $\xi_\partial$ is tight and there is a unique (isotopy class of) tight contact structure on the ball extending $\xi_\partial$. On the other hand, if after edge rounding the dividing set is disconnected, then $\xi_\partial$ is overtwisted, and there is no tight contact structure on $\M(\Gamma_0, \Gamma_1)$. We refer to $\M(\Gamma_0, \Gamma_1)$ as \emph{tight} or \emph{overtwisted} respectively, as in \cite[defn.\ 3.5]{Me09Paper}.
It is not hard to show that $\M(\Gamma, \Gamma)$ is tight for any chord diagram $\Gamma$.

In \cite[lem.\ 3.1]{Me09Paper}, the second author proved the following. It is a version of the \emph{isotopy discretisation} principle: see also \cite{Colin97_chirurgies} and \cite{Hon02}.

\begin{lem}[\cite{Me09Paper}]
\label{lem:cylinders_from_bypasses}
If $\M(\Gamma, \Gamma')$ is tight, then the tight contact structure so obtained on the solid cylinder is contactomorphic to 
$\M(\Gamma, \Gamma)$
with a finite set of bypass attachments.
\qed
\end{lem}

In particular, if $\M(\Gamma, \Gamma')$ is tight, then attaching the bypasses of lemma \ref{lem:cylinders_from_bypasses} in sequence, we obtain a sequence of dividing sets
\[
\Gamma = \Gamma_0, \Gamma_1, \ldots, \Gamma_m = \Gamma'
\]
where for each $i \geq 0$, the set $\Gamma_{i+1}$ is obtained from $\Gamma_i$ by an upwards bypass surgery. Each $\M(\Gamma, \Gamma_i)$ obtains a contact structure by attaching the first $i$ of these bypasses to 
$\M(\Gamma, \Gamma)$.
Being contactomorphic to a sub-cylinder of the tight $\M(\Gamma, \Gamma')$, each $\M(\Gamma, \Gamma_i)$ is tight.

We shall need the following lemma later.

\begin{lem}
\label{lem:chord_diagrams_connected}
If $\M(\Gamma_0, \Gamma)$ and $\M(\Gamma_0, \Gamma')$ are both tight, then there is a sequence of chord diagrams
\[
\Gamma = \Gamma_1, \Gamma_2, \ldots, \Gamma_m = \Gamma'
\]
where each $\Gamma_{i+1}$ is obtained from $\Gamma_i$ by a bypass surgery, and each $\M(\Gamma_0, \Gamma_i)$ is tight.
\end{lem}

\begin{proof}
The previous paragraph, applied to $\M(\Gamma_0, \Gamma)$, provides a sequence of dividing sets $\Gamma_0, \ldots, \Gamma_m = \Gamma$ such that for each $i \geq 0$, the set $\Gamma_{i+1}$ is obtained from $\Gamma_i$ by an upwards bypass surgery, and each $\M(\Gamma_0, \Gamma_i)$ is tight. The same argument on $\M(\Gamma_0, \Gamma')$ yields a sequence of dividing sets $\Gamma_0 = \Gamma'_0, \ldots, \Gamma'_l = \Gamma'$ such that for each $i \geq 0$, the set $\Gamma'_{i+1}$ is obtained from $\Gamma'_i$ by an upwards bypass surgery, and each $\M(\Gamma_0, \Gamma'_i)$ is tight. Putting these two sequences together gives the sequence
\[
\Gamma = \Gamma_m, \Gamma_{m-1}, \ldots, \Gamma_1, \Gamma_0 = \Gamma'_0, \Gamma'_1, \ldots, \Gamma'_l = \Gamma'.
\]
which has the desired properties. Note the first $m$ bypass surgeries are downwards, and the last $l$ surgeries are upwards.
\end{proof}

\subsection{Sutured Floer homology and contact invariants}
\label{sec:SFHbackground}

Sutured Floer homology is an invariant of balanced sutured 3-manifolds, introduced by Juh\'{a}sz in \cite{Ju06}, extending the (hat version of) Heegaard Floer homology of Ozsv\'{a}th--Szab\'{o} \cite{OS04Prop, OS04Closed, OSContact, OS06} for closed 3-manifolds.

We refer to \cite{Ju06} for a full definition of $SFH$. It suffices here to mention a few details. To define $SFH(M, \Gamma)$, we start by taking a Heegaard diagram for $(M, \Gamma)$ consisting of a surface $\Sigma$ and curves $\alpha_i$, $\beta_j$. This means that thickening $\Sigma$ to $\Sigma \times [0,1]$ and performing surgery along each $\alpha_i \times \{0\}$ and $\beta_j \times \{1\}$ yields $M$, and, furthermore, $\Gamma = \partial \Sigma \times \{1/2\}$. The signed region $R_+$ consists of the surgered $\Sigma \times \{1\}$ and $\partial \Sigma \times (1/2, 1)$; similarly $R_-$ consists of the surgered 
$\Sigma \times \{0\}$ and $\partial \Sigma \times (0, 1/2)$.

The $\alpha_i$ and $\beta_j$ form asymptotic conditions for holomorphic curves in $\Sigma \times I \times \R$ (using the cylindrical reformulation of Lipshitz \cite{Lip}). Sutured Floer homology is the homology of a chain complex generated by these asymptotic conditions (precisely, complete intersections of $\alpha_i \cap \beta_j$; note that the balanced condition implies that the $\alpha$- and $\beta$-curves form equinumerous sets), with a differential defined by counts of rigid holomorphic curves.

Originally defined over $\Z_2$ coefficients, $SFH(M, \Gamma)$ can be extended to work with $\Z$ coefficients or twisted coefficients \cite{Ghiggini_Honda08, Massot09, Me14_twisty_itsy_bitsy}. In this paper we always consider $\Z$ coefficients.

Whatever coefficient ring is used, $SFH(M, \Gamma)$ is a bigraded module over this ring. The first grading is by spin-c structures on $(M, \Gamma)$. That is,
\[
SFH(M, \Gamma) = \bigoplus_{\s \in \Spin^c (M, \Gamma)} SFH(M, \Gamma, \s).
\]
Spin-c structures on $M$ are in bijective correspondence with $H^2 (M, \partial M)$. 
In section \ref{sec:spin-c_structures} below we discuss in detail spin-c structures 
and their relationship to the Euler class.
The second (Maslov) grading is a relative $\Z_2$ homological grading and is determined by a homology orientation on $H_*(M, R_-)$; we do not need it in this paper and refer to \cite[sec.\ 2.4]{FJR11} for details.

Most importantly for our purposes, a contact structure $\xi$ on $(M, \Gamma)$ gives rise to a \emph{contact invariant} (or \emph{contact element} or \emph{contact class}) $c(\xi)$ in $SFH$.  See \cite{HKM09, HKMContClass, OSContact}. This generalises the situation in Heegaard Floer homology of closed manifolds. Precisely, $c(\xi)$ lies in $SFH(-M, -\Gamma)$, where the minus signs refer to reversed orientation. The contact invariant has a $(\pm 1)$ ambiguity: it is a well defined element when $\Z_2$ coefficients are used, but with integer or twisted coefficients it is given by a pair of elements $c(\xi) = \{\pm a\}$, for some $a \in SFH(-M, -\Gamma)$ \cite{HKM08}. Of course if $0 \in c(\xi)$, then $c(\xi)$ has a single element, as $\pm 0 = 0$. In that case we will write $c(\xi)=0$. In general we write $c(\xi) = \pm a$.

The contact invariant $c(\xi)$ can be constructed using the Giroux correspondence between open book decompositions and contact structures (\cite{Et06, Gi02}; see also \cite{Etgu_Ozbagci11, HKM09} for the sutured case), building a Heegaard decomposition from an open book supporting $\xi$, and taking a specific associated element of Floer homology. We refer to \cite{HKMContClass} for the closed case and \cite{HKM09} for the sutured case.

We mention some properties of contact elements. When $\xi$ is overtwisted, $c(\xi) = 0$ \cite{HKM09}. The following TQFT-like property was proved by Honda--Kazez--Mati\'{c} \cite{HKM08} over $\Z$ coefficients, further discussed in \cite[sec.\ 8.2]{Me12_itsy_bitsy} and extended to twisted coefficients in \cite[thm.\ 4.14]{Me14_twisty_itsy_bitsy}. Suppose we have a balanced sutured manifold $(M', \Gamma')$ lying in the interior of another balanced sutured manifold $(M, \Gamma)$. The ``intermediate'' region between $M'$ and $M$ is naturally a sutured manifold $(M - \Int M', \Gamma \cup -\Gamma')$. Let $\xi$ be a contact structure on this intermediate sutured manifold. Then there is a natural map
\[
\Phi_\xi\colon SFH(-M', -\Gamma') \To SFH(-M, -\Gamma),
\]
well defined up to an overall sign. A contact structure $\xi'$ on $(M', \Gamma')$ naturally extends to a contact structure $\xi' \cup \xi$ on $(M, \Gamma)$, and $\Phi_\xi$ takes $c(\xi')$ to $c(\xi' \cup \xi)$. Thus inclusions of sutured manifolds, with a contact structure between them, naturally give linear maps on sutured Floer homology, which are natural with respect to contact elements.

At the end of the next section 
we discuss how the Euler class of a contact structure relates to the spin-c grading of its contact invariant in sutured Floer homology. But first we must discuss spin-c structures themselves.

\subsection{Spin-c structures and the Euler class}
\label{sec:spin-c_structures}

We briefly review spin-c structures here, and refer to \cite{OS04Closed} and \cite{Turaev97_torsion} for details.

A \emph{spin-c structure} on a closed, connected, oriented 3-manifold $M$ is a homology class of nonvanishing vector fields on $M$. Two (nonvanishing) vector fields are \emph{homologous} if they are homotopic through nonvanishing vector fields in the complement of a 3-ball in $M$. Equivalently, they are homologous if they are homotopic in the complement of finitely many disjoint 3-balls in $M$. We denote the set of all spin-c structures on $M$ by $\Spin^c (M)$.

There are several equivalent definitions useful for our purposes. If $M$ has a Riemannian metric then we may restrict to unit vector fields.
Further, unit vector fields $v$ are naturally identified with oriented 2-plane fields $\zeta_v$ by taking the orthogonal complement. If we fix a trivialisation $\tau$ of the tangent bundle $TM$, a unit vector field $v$ gives a map $f_v \colon M \To S^2$. Each of the unit vector field $v$, oriented 2-plane bundle $\zeta_v$ or map $f_v$ determines all the others, and a homotopy of one is equivalent to a homotopy of the others. So a spin-c structure can also be defined as any one of these, up to homotopy in the complement of one (or finitely many) balls.

Obstruction-theoretically, if we give $M$ the structure of a CW complex, then for $v_1, v_2$ (or $f_{v_1}, f_{v_2}$ or $\zeta_{v_1}, \zeta_{v_2}$) to be homologous means that they are homotopic on the 2-skeleton of $M$. The obstruction to such a homotopy lies in $H^2(M; \Z)$, and vanishes if and only if $f_{v_1}, f_{v_2}$ induce the same maps on cohomology, i.e., $f_{v_1}^* = f_{v_2}^* \colon H^2 (S^2; \Z) \To H^2 (M; \Z)$. Since $H^2(S^2; \Z) \cong \Z$, letting $\mu$ be an arbitrarily chosen generator of $H^2(S^2; \Z)$, the spin-c structure of $f_v$ is determined by the cohomology class $f_v^* (\mu)$; furthermore all cohomology classes in $H^2 (M; \Z)$ can arise from spin-c structures in this way. So the assignment $\delta^\tau \colon \Spin^c (M) \To H^2(M; \Z)$ which sends $v \mapsto f_v^* (\mu)$ is a bijection. This assignment depends on the trivialisation $\tau$, but the \emph{difference} $\delta^\tau (v_1) - \delta^\tau (v_2) \in H^2 (M; \Z)$ does \emph{not}; it only depends on $v_1$ and $v_2$ \cite{OS04Closed}. Switching to multiplicative notation, we write $[v_1] / [v_2] \in H^2 (M;\Z) \cong H_1 (M;\Z)$ for this (co)homology class.

Thus, $\Spin^c (M)$ is naturally an affine space over $H^2 (M; \Z)$, and different trivialisations of $TM$ give different identifications between $\Spin^c (M)$ and $H^2 (M; \Z)$. The action of $H^2 (M; \Z) \cong H_1 (M; \Z)$ on a spin-c structure represented by a nonvanishing vector field $v$ can be given explicitly: one can ``add'' a homology class $h \in H_1 (M; \Z)$ to the spin-c structure of $v$ by performing \emph{Reeb turbularisation}, which is a certain modification of $v$ in the neighbourhood of an embedded curve representing $h$ \cite{Turaev89_Euler}.

Juh\'{a}sz in \cite[sec.\ 4]{Ju06} considered spin-c structures on sutured 3-manifolds. Letting $(M, \Gamma)$ denote a connected balanced sutured 3-manifold, we define a vector field $\nu_0$ on $\partial M$. This vector field points out of $M$ along $R_+$, into $M$ along $R_-$, and is tangent to $\partial M$ along $\Gamma$, transverse to $\Gamma$ and pointing from $R_-$ to $R_+$. 
The set of such vector fields forms a contractible space and up to homotopy among such vector fields $\nu_0$ is well defined.

We consider nonvanishing vector fields on $M$ which extend $\nu_0$. Two such vector fields on $M$ are again homologous if they are homotopic, through nonvanishing extensions of $\nu_0$, in the complement of an open ball (or finitely many open balls) in the interior of $M$. A \emph{spin-c structure} on $(M, \Gamma)$ is a homology class of nonzero vector fields on $M$ which restrict to $\nu_0$ on $\partial M$. The set of spin-c structures on $(M, \Gamma)$ is denoted by $\Spin^c (M, \Gamma)$.

The set $\Spin^c (M, \Gamma)$ is an affine space over $H^2 (M, \partial M) \cong H_1 (M)$. Again we may speak equally of a vector field $v$ extending $\nu_0$, an orthogonal 2-plane bundle $\zeta_v$ (using a Riemannian metric) extending $\nu_0^\perp$ (which in turn can be taken to be $T(\partial M)$ along $\partial M$, outside of a neighbourhood of $\Gamma$), or a function $f_v\colon M \To S^2$ extending $f_{\nu_0} \colon \partial M \To S^2$ (using a trivialisation of $TM$). To say that $v_1, v_2$ (or $f_{v_1}, f_{v_2}$ or $\zeta_{v_1}, \zeta_{v_2}$) are homologous means that they are homotopic on the 2-skeleton of $M$, relative to the boundary where they agree. Thus the obstruction to such a homotopy lies in $H^2 (M, \partial M)$. The obstruction vanishes if $f_{v_1}, f_{v_2}$ induce the same maps on cohomology, $f_{v_1}^* = f_{v_2}^* \colon H^2 (S^2) \To H^2 (M)$. Since $f_{v_1}, f_{v_2}$ agree on $\partial M$, the map $f_{v_1}^* - f_{v_2}^*$ has image in $H^2 (M, \partial M)$. Again this difference does not depend on the choice of trivialisation, and the action of $H_1 (M)$ on $\Spin^c (M, \Gamma)$ can be achieved explicitly by Reeb turbularisation.

It follows from the above that an oriented 2-plane bundle $\zeta$ on a closed oriented 3-manifold $M$ (such as a contact structure) belongs to a spin-c structure, and $\Spin^c (M)$ is affine over $H^2 (M)$. But $\zeta$ also has an \emph{Euler class} in $H^2(M)$. It will be important in the sequel to understand the relationship between these two invariants.

It is well known that the Euler class $e(\zeta)$ of an oriented 2-plane bundle $\zeta$ on a CW complex $X$ only depends on the homotopy class of $\zeta$. Moreover, $e(\zeta)$ is the obstruction to finding a nonvanishing section of the bundle over the 2-skeleton of $X$ (see, e.g., \cite[thm.\ 12.5]{Milnor_Stasheff}).

Hence two oriented 2-plane fields on $M$ in the same spin-c class have the same Euler class. Since nonvanishing vector fields and oriented 2-plane fields, and their homotopies, correspond via orthogonal complements, we may abuse notation and write $e(v) = e(\zeta) \in H^2 (M)$ for a nonvanishing vector field $v$, where $\zeta = v^\perp$. Vector fields in the same spin-c class have the same Euler class, so we may then write $e(\s) = e(v)$, where $\s$ is the spin-c class of $v$.

Consider a 2-dimensional vector bundle $\zeta$, orthogonal to the unit vector field $u$. Then the Euler class $e(\zeta)$ is the obstruction $[u]/[-u]$ to a homotopy from $-u$ to the opposite vector field $u$ over the 2-skeleton. 
Indeed, a homotopy exists if and only if ``it knows where to cross $\zeta$.'' Let us also recall that the spin-c structures $[u]+h$ and $[-u]-h$, where $h\in H_1 (M)$, can always be represented by a non-vanishing vector field $v$ and its opposite $-v$, respectively.
See \cite{Turaev89_Euler} for details.

This means that if we ``add'' a homology class $h$ to a spin-c class $\s$, then we add $2h$ to its Euler class (more precisely, its Poincar\'{e} dual):
\begin{equation}
\label{eqn:spin-c_and_euler}
e(\s+h) = 2 PD(h) + e(\s)
\end{equation}
See \cite[thm.\ 5.3.1]{Turaev89_Euler} for full details. 
(Note we write the action of $H_1 \cong H^2$ on spin-c structures additively, whereas Turaev writes it multiplicatively.)

Turaev also effectively treats the sutured case, considering nonvanishing vector fields on a 3-manifold $M$ with boundary, which point into $M$ on a subset $R_- \subset \partial M$ and out on $R_+ \subset \partial M$; these are precisely the balanced sutured manifolds with vector fields extending $\nu_0$, 
whose homology classes form $\Spin^c (M, \Gamma)$. For such vector fields $u$ and $v$, the obstruction $[u]/[v]$ to a homotopy over the 2-skeleton lies in $H^2 (M, \partial M) \cong H_1(M)$. Nonvanishing vector fields extending $\nu_0$ (or equivalently, oriented 2-plane fields extending $\nu_0^\perp$) also have Euler class lying in $H^2 (M, \partial M)$, and again we may speak of the Euler class of a spin-c structure on $(M, \Gamma)$. We again obtain equation (\ref{eqn:spin-c_and_euler}), where now $\s \in \Spin^c (M, \Gamma)$, $h \in H_1 (M)$, and $e(\s), e(\s+h) \in H^2(M, \partial M)$.

As mentioned earlier, $SFH(M, \Gamma)$ splits as a direct sum $\bigoplus_\s SFH(M, \Gamma, \s)$ over spin-c structures $\s \in \Spin^c (M, \Gamma)$. We have also seen that if $\xi$ is a contact structure on $(M, \Gamma)$, then there is a contact invariant $c(\xi) \subset SFH(-M,-\Gamma)$. On the other hand, a 2-plane field such as $\xi$ determines a spin-c class $\s_\xi$ (i.e., the spin-c class of $\xi^\perp$). It would be natural for $c(\xi)$ to lie in the spin-c summand of $SFH(-M,-\Gamma)$ corresponding to the spin-c class of $\xi$.
This is in fact the case; it appears as proposition 9.4 of \cite{Ju16_cobordisms}.

\begin{prop}
\label{prop:contact_invariant_spin-c_summand}
Let $\xi$ be a contact structure on the balanced sutured manifold $(M, \Gamma)$, with spin-c structure $\s_\xi$. Then $c(\xi) \subset SFH(-M, -\Gamma, \s_\xi)$.
\qed
\end{prop}

\section{Background on formal knot theory}
\label{sec:FKT}

We now give a brief discussion of some aspects of Kauffman's \emph{formal knot theory}; for further details we refer to \cite{Kauffman_FKT83}.

Consider a connected, unoriented knot diagram where crossing data is forgotten. This yields a connected plane graph where each vertex has degree $4$. Using Euler's formula, one can show that the number of complementary regions of the graph exceeds the number of vertices by $2$. 

A \emph{universe} is a connected planar graph, where each vertex has degree $4$, and two adjacent complementary regions are labelled with stars. See figure \ref{fig:figure_8_knot_universe}. Thus in a universe, the number of vertices equals the number of unstarred regions. One can draw a universe in the plane so that the graph consists of an immersion of some circles into $\R^2$, with all intersections being transverse double points. 

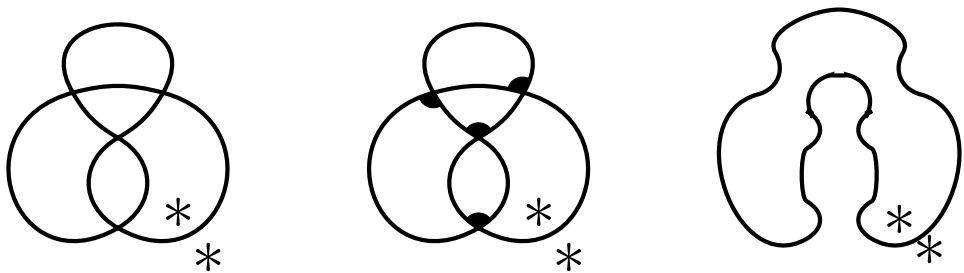
\begin{figure}
\begin{center}
\begin{minipage}{0.3\textwidth}
\input{fig-8_universe.tikz.tex}
\end{minipage}
\begin{minipage}{0.3\textwidth}
\input{fig-8_state.tikz.tex}
\end{minipage}
\begin{minipage}{0.3\textwidth}
\input{fig-8_trail.tikz.tex}
\end{minipage}
\end{center}
\caption{
Left: a universe $\U$ based on the figure 8 knot. Centre: a state of $\U$. Right: The cor\-res\-pond\-ing 
Euler--Jordan trail.
}
\label{fig:figure_8_knot_universe}
\end{figure}

At each vertex $v$ of a universe $\U$, four edges meet and between them lie four quadrants, i.e., corners of regions. (The four regions involved need not be distinct.) A \emph{marker} at $v$ is a choice of one of the four quadrants at $v$.

A \emph{state} of a universe $\U$ is a choice of marker at each vertex of $\U$, so that each unstarred region contains a marker. Thus a state of $\U$ provides a bijection between vertices of $\U$ and (adjacent) unstarred regions of $\U$. See figure \ref{fig:figure_8_knot_universe} (centre).

It is sometimes useful to \emph{split} the crossings in a universe. A \emph{splitting} of a universe $\U$ at a vertex $v$ replaces a neighbourhood of $v$ with two non-intersecting arcs in one of two possible ways, as shown in figure \ref{fig:vertex_splitting}. If $v$ has a state marker, the marker specifies a splitting as shown in figure \ref{fig:state_marker_splitting}. 

\begin{figure}
\begin{center}
\input{vertex_splitting.tikz.tex}
\end{center}
\caption{A vertex of a universe can be split in two possible ways.}
\label{fig:vertex_splitting}
\end{figure}
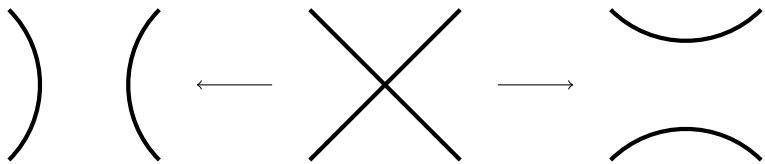

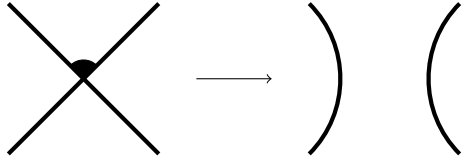
\begin{figure}
\begin{center}
\input{state_marker_splitting.tikz.tex}
\end{center}
\caption{Splitting a vertex according to a state marker.}
\label{fig:state_marker_splitting}
\end{figure}

If every vertex of $\U$ is split in some way, then the result will be a collection of 
loops, i.e., non-intersecting embedded circles, in the plane. If there is just one such loop, then we call it an \emph{Euler--Jordan trail} or just a \emph{trail}. That is, a trail 
is a single loop obtained by splitting each vertex. 

The splitting at each vertex provided by the markers of a state produces a trail. See figure \ref{fig:figure_8_knot_universe} (right). Conversely, a splitting of each vertex of $\U$ that produces a trail arises from a state of $\U$. Thus there is a bijective correspondence between states of $\U$ and trails on $\U$, known as the \emph{state-trail correspondence}.

A \emph{transposition} is a transition between states of a universe $\U$, which involves switching two state markers under certain circumstances. Suppose $v$ and $w$ are distinct vertices of $\U$. 
Suppose further that two regions $R_1, R_2$
near $v$ are also regions
near $w$, so that the situation is as shown in figure \ref{fig:transposition}. Now suppose $s_1$ is a state of $\U$ where the vertex $v$ has marker in $R_1$ and $w$ has marker in $R_2$, as shown. Then the assignment $s_2$ obtained by switching the markers at $v$ and $w$ so that $v$ has marker in $R_2$ and $w$ has marker in $R_1$ is also a state, and we say $s_2$ is obtained from $s_1$ by a \emph{transposition}. We say the transposition is \emph{clockwise} since, in switching from $s_1$ to $s_2$, markers have moved $90^\circ$ clockwise around $v$ and $w$. Conversely, from $s_2$ there is a \emph{counterclockwise transposition} to $s_1$.

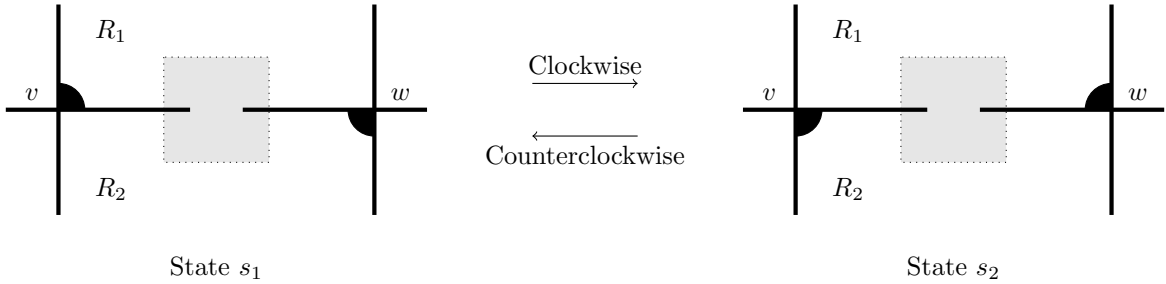
\begin{figure}
\begin{center}
\input{transposition.tikz.tex}
\end{center}
\caption{A state transposition. The two vertices $v,w$ are both adjacent to regions $R_1, R_2$. 
There may be other vertices and edges of the universe in the shaded box.}
\label{fig:transposition}
\end{figure}

The set of states on a universe $\U$ has an interesting structure. We form a directed graph $\L_\U$ whose vertices are the states of $\U$, and which has a directed edge from $s_1$ to $s_2$ if the state $s_2$ can be obtained from $s_1$ by a clockwise transposition. More generally we write $s_1 \leq s_2$ if 
there is a directed path from $s_1$ to $s_2$ in $\L_\U$. \emph{Kauffman's clock theorem} \cite{Kauffman_FKT83} (in its slightly extended form due to Gilmer and Litherland \cite{Gilmer-Litherland86}) says that $\L_\U$ is in fact a \emph{lattice}: the relation $\leq$ is a partial order on states, and any two states have a least upper bound and greatest lower bound with respect to this order.

\section{Background on hypertrees and trinities}
\label{sec:hyper_background}

\subsection{Plane graphs and trinities}
\label{sec:planar_graphs_links_trinities}

Let $G$ be a finite plane graph, possibly with multiple edges. That is, $G$ has a fixed embedding in $\R^2 \subset \R^3$; we may also compactify and regard $G$ as embedded in $S^2 \subset S^3$. 
If $G$ is connected then all its complementary regions are homeomorphic to discs. By placing a new vertex in each region and connecting it to the surrounding vertices, we obtain a triangulation of $S^2$.

The same vertices are used in the construction of the dual graph $G^*$.
The spanning trees of 
$G$ and $G^*$ 
are closely related. Each spanning tree $T$ of $G$ yields a dual spanning tree $T^*$ of $G^*$; the edges of $T^*$ are precisely those edges of $G^*$ that correspond to edges of $G$ not contained in $T$. Conversely, each spanning tree of $G^*$ yields a dual spanning tree of $G$; the spanning trees of $G$ and $G^*$ are in bijective correspondence via planar duality.

In this paper we are mostly concerned with the case when $G$ is bipartite. Then, as discussed below, $G^*$ is naturally directed and the triangulation constructed from $G$ is properly 3-coloured.
Such a structure was studied by Tutte in \cite{Tutte48}. 
We present some background here and refer to \cite[sec.\ 9]{Kalman13_Tutte} for further details.

Denote the two vertex classes of $G$ by $V$ and $E$; call them violet and emerald. 
Place a red vertex in each region of $G$; call this set of vertices $R$. 
Then the triangulation above is such that its $1$-skeleton is a tripartite plane graph: there are three vertex classes/colours $V,E,R$, and each edge joins vertices of different colours. We colour each edge by the unique colour different from its endpoints.

Each triangle has its three vertices of the three distinct colours. (The same is true of its edges.) 
Such a structure --- a triangulation of $S^2$ with a three-colouring of the vertices --- is called a  \emph{trinity}. 
See figure \ref{fig:trinity_construction} for an example.
In each triangle of the trinity, the vertices are coloured violet, emerald and red in clockwise or anticlockwise order; thus triangles come in two types, and we colour them black or white accordingly. Further, two triangles sharing an edge are of opposite colours; in other words, the dual graph to the triangulation is bipartite.

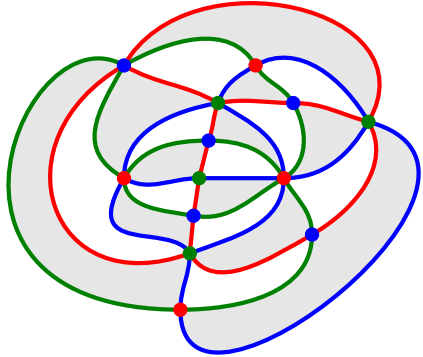
\begin{figure}
\begin{center}
\input{trinity.tikz.tex}
\end{center}
\caption{The trinity constructed from the bipartite planar graph of figure \ref{fig:median_construction}.}
\label{fig:trinity_construction}
\end{figure}

A trinity contains three connected bipartite plane graphs, consisting of the subgraphs given by the edges of each colour \cite{Juhasz-Kalman-Rasmussen12, Kalman13_Tutte}. The \emph{violet graph} $G_V$ has violet edges and vertex classes $E$, $R$; the \emph{emerald graph} $G_E$ has emerald edges and vertex classes $R$, $V$; the \emph{red graph} $G_R$ has red edges and vertex classes $V$, $E$. Any one of these three graphs yields an equivalent triangulation of $S^2$ and hence the same trinity. In other words, as is obvious from the definition, the roles of the three colours in a trinity are perfectly symmetric.
Conversely (and to further underscore the last point), given a triangulation of $S^2$ with bipartite dual, one can always find a proper 3-colouring of its vertices (i.e., with each edge joining vertices of different colours) endowing the triangulation with the structure of a trinity. This is an old result, proven, e.g., in \cite[prop.\ 9.4]{Kalman13_Tutte}.

\subsection{Arborescence number and Tutte's tree trinity theorem}
\label{sec:arborescence_number}

Let us fix a trinity with vertex sets $V,E,R$ as above, and let $n$ denote the number of its white triangles.
Each red edge lies on the boundary of precisely one white triangle, and each white triangle has precisely one red edge, so the number of red edges is $n$. By the same argument the number of black triangles, the number of violet edges, and the number of emerald edges are also $n$. So there are $3n$ edges and $2n$ faces in the triangulation, whence Euler's formula gives $|V|+|E|+|R| - 3n + 2n = 2$; thus the total number of vertices exceeds the number of white triangles by $2$.

\begin{figure}
\begin{center}
\input{trinity_with_dual.tikz.tex}
\end{center}
\caption{A trinity, together with the dual graph $G_V^*$ of $G_V$ drawn in dotted blue.}
\label{fig:trinity_with_dual}
\end{figure}
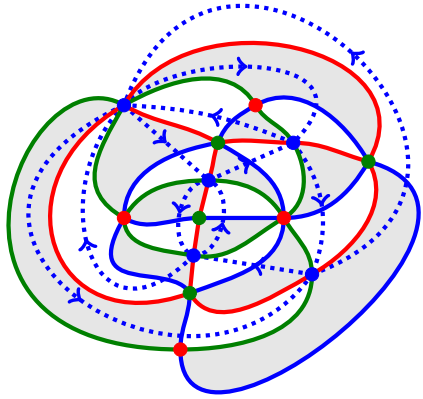

Consider the dual graph $G_V^*$ of $G_V$. The vertex set 
of $G_V^*$ is $V$, and each edge of $G_V^*$ runs 
through a black and a white triangle, crossing exactly one 
edge of $G_V$. See figure \ref{fig:trinity_with_dual}.
This allows us to orient 
the edges of $G_V^*$ 
to point from black to white, giving $G_V^*$ 
the structure of a directed plane graph (possibly with loops and multiple edges). An equivalent way of defining the direction is to say that each edge of $G_V^*$ runs between an emerald vertex to its left and a red vertex to its right.

Each violet vertex is surrounded by triangles which alternate in colour between black and white. Hence at each vertex of $G_V^*$, edges are alternately incoming and outgoing; in particular, the in-degree and out-degree are equal. A directed graph where each vertex has equal in-degree and out-degree is called 
\emph{Eulerian} or
\emph{balanced}.
Of course, $G_R^*$ and $G_E^*$ are also balanced directed plane graphs. 

We may fix a white triangle $t_0$ of the trinity as a \emph{root} or \emph{outer} triangle, and call the three adjacent vertices $v_0 \in V$, $e_0 \in E$, $r_0 \in R$ the \emph{root} violet, emerald, and red vertices, respectively. 
For any directed graph $D$ with a root vertex $r$, a \emph{(spanning) arborescence} of $D$ is a spanning tree $T$ of $D$ all of whose edges point away from $r$; that is, for each vertex $v$ of $D$ the unique path in $T$ between $r$ and $v$ has all edges oriented from $r$ to $v$.

When $D$ is a balanced finite directed graph, van Aardenne-Ehrenfest and de Bruijn 
in \cite{vA-E_dB51} 
showed 
that its number of spanning arborescences does not depend on the choice of root vertex. Moreover, this number is also equal to the number of spanning trees all of whose edges point \emph{towards} the root, wherever the root may be. We call this number the \emph{arborescence number} $\rho(D)$ of $D$.

Tutte's \emph{tree trinity theorem} \cite{Tutte48} 
generalises planar duality of spanning trees (reviewed in section \ref{sec:planar_graphs_links_trinities}) to the statement
that the arborescence numbers of 
$G_V^*, G_E^*, G_R^*$ agree:
\[
\rho(G_V^*) = \rho(G_E^*) = \rho(G_R^*).
\]
We will call this number the \emph{magic number} of the trinity.
Tutte 
gave a `trijective' proof of his result in \cite{Tutte75}. That is, he arranged all arborescences in triples, one from each graph, where each triple is described in the form of a bijection between non-outer white triangles and adjacent non-root vertices. 
Berman \cite{Berman80} formulated this as an expression for the magic number as the determinant of an $(n-1) \times (n-1)$ adjacency matrix.
The latter formula was refined by the first author \cite{Kalman13_Tutte} to enumerate the hypertrees in the corresponding hypergraphs.
We explain these notions in the next section.

\subsection{Hypergraphs and hypertrees}
\label{sec:hypergraphs_hypertrees}

A \emph{hypergraph} is a pair $\HH = (V, E)$, where $V$ is a set of \emph{vertices} and $E$ is a (multi-)set of \emph{hyperedges}. A hyperedge is a non-empty subset of $V$. 
A multiset is used to allow for hyperedges with multiplicity.
When each hyperedge contains precisely two vertices, a hypergraph reduces to a graph in the usual sense (with no loop edges but possibly with multiple edges).

A hypergraph $\HH$ naturally determines a bipartite graph $\Bip \HH$, which has vertex classes $V$ and $E$; an edge connects $v \in V$ to $e \in E$ in $\Bip \HH$ if and only if the hyperedge $e$ contains $v$.
Conversely, given a bipartite graph $G$ with vertex classes $V_0, V_1$, we may form a hypergraph $\G = (V_0, V_1)$ with vertex set $V_0$ and hyperedge set $V_1$: for $v_1 \in V_1$, the hyperedge $v_1$ contains all vertices in $V_0$ to which $v_1$ is connected in $G$. We may also reverse the roles of $V_0$ and $V_1$, forming a hypergraph $\G' = (V_1, V_0)$ with vertex set $V_1$ and hyperedge set $V_0$, constructed in similar fashion. Note that $\Bip \G = \Bip \G' = G$. We say that the hypergraphs $\G$ and $\G'$ are \emph{induced} by the bipartite graph $G$. The two hypergraphs $\G$ and $\G'$ are called \emph{abstract dual} or \emph{transpose} to one another: the abstract dual of a hypergraph is given by reversing the roles of vertices and hyperedges, and the incidence matrix of one is the transpose of the other.
These considerations are independent of planar embeddings.

There is another notion of duality related to hypergraphs, one that was implicit in the previous two subsections. When $\Bip\HH$ is given as a plane graph, one may replace the set $V$ of vertices with the set $R$ of complementary regions, while keeping the same set $E$ of hyperedges. The containment relation is defined in the obvious way, by adjacency in the violet graph $G_V$; in other words, $\Bip(R,E)=G_V$. Note how this generalizes the way one takes the dual of a usual plane graph. We will indeed call $(R,E)$ the \emph{planar dual} of $(V,E)$.

When we have a connected plane bipartite graph --- or equivalently, a trinity --- we can use both these notions of duality. Indeed, a trinity naturally contains \emph{six} hypergraphs. For if we take a trinity with vertex classes $V, E, R$, then we have the three plane graphs $G_V, G_E, G_R$ described above, 
and each of these determines two hypergraphs.
Denoting planar duality with a star, and abstract duality with a bar, letting $\HH = (V,E)$ gives the six hypergraphs as
\[
\HH = (V,E), \quad
\HH^* = (R,E), \quad
\overline{\HH^*} = (E,R), \quad
\overline{\HH^*}^* = \overline{\overline{\HH}^*} = (V,R), \quad
\overline{\HH}^* = (R,V), \quad
\overline{\HH} = (E,V).
\]

Returning to the not-necessarily-planar context, we define
a \emph{hypertree} in a hypergraph $\HH = (V,E)$ as a function $f\colon E \To \Z_{\ge0}$
such that there exists a spanning tree in the bipartite graph $\Bip \HH$ with degree $f(e) + 1$ at each node $e \in E$. Such a spanning tree in $\Bip \HH$ is said to \emph{realise} the hypertree $f$. 

Note that in any hypertree, for any $e \in E$ we have $0 \leq f(e) \leq |e| - 1$. When $\HH$ is just a graph, so every $e \in E$ satisfies $|e| = 2$, a hypertree reduces to a tree: each $f(e) = 0$ or $1$ and a tree is chosen by selecting those edges with $f(e) = 1$.

A function $E \To \Z_{\ge0}$ can be regarded as an element of $\Z^E \subset \R^E$. Thus the set of hypertrees of $\HH = (V,E)$ can be regarded as a subset of the $|E|$-dimensional integer lattice $\Z^E$, and it
turns out to  be the set of lattice points of a convex polytope $\QQ_\HH \subset \R^E$
\cite{Postnikov09}.
See for example \cite[Theorem 3.4]{Kalman13_Tutte} for a description of $\QQ_\HH$ by linear inequalities. 
Let us write $B_\HH = \QQ_\HH \cap \Z^E$ for the set of hypertrees in $\HH$. 
For an arbitrary pair of abstract-dual hypergraphs, Postnikov \cite{Postnikov09} showed that their hypertrees are equinumerous. In symbols, 
\begin{equation}\label{eq:postnikov}
|B_\HH|=|B_{\overline{\HH}}|.
\end{equation}

Given a bipartite plane graph $G$ with vertex classes $V_0, V_1$, we can form the two induced abstract-dual hypergraphs $\G_0 = (V_1, V_0)$ and $\G_1 = (V_0, V_1)$. We can also form the planar dual $G^*$ of $G$ (as a graph; this is not to be confused with planar dual hypergraphs). 
We saw that $G^*$ is a balanced directed graph. The first author showed \cite[Theorem 10.1]{Kalman13_Tutte} that the number of hypertrees in $\G_0$, and (either by the same proof, or as a consequence of (\ref{eq:postnikov})) the number of hypertrees in $\G_1$, is equal to the arborescence number of $G^*$. In symbols,
\[
\rho(G^*) = |B_{\G_0}| = |B_{\G_1}|.
\]
In a trinity, then, 
the arborescence numbers of all three planar duals and the number of hypertrees of all six hypergraphs are equal:
\begin{equation*}
\rho(G_V^*) 
= \rho(G_E^*) = \rho(G_R^*) 
= | B_{(V,E)} |
= | B_{(E,V)} |
= | B_{(E,R)} | 
= | B_{(R,E)} |
= | B_{(R,V)} |
= | B_{(V,R)} |.
\end{equation*}

\section{Special alternating links and sutured Floer homology}
\label{sec:trinities_SFH}

Work of the first author with Juh\'{a}sz and Rasmussen established a connection between hypertrees in trinities
and sutured Floer homology. 
We review a few relevant details here.

\subsection{Links and sutured manifolds from trinities}
\label{sec:sutured_manifolds_trinities}

Suppose we have an oriented link $L \subset S^3$, and a Seifert surface $R$ for $L$, so $\partial R = L$. Splitting $S^3$ open along $R$ produces a 3-manifold with boundary consisting of two homeomorphic copies of $R$, which we denote by $R_\pm$. Taking $L = \partial R_\pm$ as a set of sutures and $R_\pm$ as signed regions yields a balanced sutured 3-manifold $S^3(R)$. 

Now from a plane graph $G$, we may construct a surface $F_G$ bounding an alternating link $L_G$ via the \emph{median construction}. Take a regular neighbourhood $U$ of $G$ in $\R^2$, which can be considered as a union of discs around each vertex of $G$, and a band along each edge. Then insert a negative half twist in each band of $U$ to obtain $F_G$, and let $L_G = \partial F_G$. 
The link $L_G$ is a union of arcs of circles around vertices of $G$, and arcs twisting around the edges of $G$.  It may be drawn so that its crossings correspond bijectively to edges of $G$. Since $L_G$ twists negatively around each edge, the diagram obtained of $L_G$ is alternating. See figure \ref{fig:median_construction} for an example.

Drawn in this way, a side of $F_G$ is facing up at each vertex of $G$. If two vertices are connected by an edge, then opposite sides of $F_G$ are facing up at them. Hence $F_G$ is orientable if and only if $G$ is bipartite, in which case the vertex classes may be distinguished by which side of $F_G$ faces up.
In fact, if $G$ is bipartite then $L_G$ is naturally oriented as the boundary of 
$F_G$, and $F_G$ is a Seifert surface for $L_G$. The orientation of $L_G$ is such that the arcs of circles around vertices run anticlockwise or clockwise according to the vertex class.

The primary objects of study in this paper are sutured 3-manifolds $(M_G,L_G)$ obtained from the general construction of $S^3(R)$ performed on the Seifert surface $R=F_G$ and oriented link $L=L_G$, 
associated to a connected bipartite plane graph $G$.

Let the vertex classes of $G$ be $V$ and $E$ and colour them violet and emerald, respectively.
For definiteness, orient $L_G$ to run anticlockwise around violet vertices, and clockwise around emerald vertices. Then the orientation on $F_G$ agrees with that of $S^2$ near violet vertices, and disagrees near emerald vertices.

Since $F_G$ deformation retracts onto $G$, the manifold $M_G$ is homeomorphic to $S^3 - N(G)$, where $N(G)$ is a regular neighbourhood of the graph $G$ in $S^3$. In particular, the manifold is a handlebody. 

The median construction provides a diagram for $L_G$ which is alternating and \emph{special}: every Seifert circle is innermost in $S^2$. (Each Seifert circle runs around a single vertex of $G$ in $S^2$, so all are innermost.) Applying Seifert's algorithm to an alternating diagram always yields a minimal genus Seifert surface \cite{Crowell59, Gabai86_genera, Murasugi58}, and applying Seifert's algorithm to $L_G$ yields $F_G$.

In sum, the surface $F_G$ obtained from the median construction on the connected bipartite plane graph $G$ is a minimal genus Seifert surface for the non-split special alternating link $L_G$. A converse of this result is also true: any minimal genus Seifert surface of a non-split prime special alternating link arises as $F_G$ from the median construction on a connected bipartite plane graph $G$ \cite{Banks11, Hirasawa-Sakuma96}.

\subsection{Sutured $L$-spaces and $SFH$ support}
\label{sec:sutured_L-spaces}

There are certain sutured 3-manifolds $(M, \Gamma)$ where, for each spin-c structure $\mathfrak{s}$, sutured Floer homology is either zero or $\Z$, i.e., $SFH(M, \Gamma, \mathfrak{s}) \cong \Z$ 
or $0$. The sutured 3-manifolds $(M_G, L_G)$ considered in this paper are of this type.

Friedl--Juh\'{a}sz--Rasmussen in \cite{FJR11} define a \emph{sutured $L$-space} to be a balanced sutured 3-manifold $(M, \Gamma)$ such that $SFH(M, \Gamma)$ is torsion free and supported in a single $\Z_2$ homological grading. It follows immediately that for every spin-c structure $\s \in \Spin^c (M, \Gamma)$, the abelian group $SFH(M, \Gamma, \s)$ is trivial or free. Friedl--Juh\'{a}sz--Rasmussen in fact showed \cite[cor.\ 1.7]{FJR11} that each $SFH(M, \Gamma, \s)$ is either trivial or isomorphic to $\Z$.

Thus, for a sutured $L$-space, to understand $SFH(M, \Gamma)$, it is sufficient to know, for each spin-c structure $\s$, whether or not the group $SFH(M, \Gamma, \s)$ is trivial. That is, it is sufficient to know the \emph{support} of $SFH(M, \Gamma)$, which is defined as
\[
\Supp(M, \Gamma) = \left\{ \s \in \Spin^c (M, \Gamma) \mid SFH(M, \Gamma, \s) \neq 0 \right\}.
\]

Examples of sutured $L$-spaces come from Seifert surfaces for links. Given an oriented link $L \subset S^3$ and a Seifert surface $R$, consider the balanced sutured 3-manifold $S^3 (R)$ discussed in the previous subsection, obtained by splitting $S^3$ along $R$. In \cite[cor.\ 6.11]{FJR11} it is shown that if $L$ is a non-split alternating link and $R$ is a minimal genus Seifert surface, then $S^3(R)$ is a sutured $L$-space. 
Thus $S^3(F_G) = (M_G, L_G)$ is a sutured $L$-space, and hence $SFH(M_G, L_G)$ consists of a direct sum of $\Z$'s, one for each spin-c structure in $\Supp(M_G, L_G)$.

\subsection{Sutured Floer homology and hypergraphs}
\label{sec:SFH_and_hypergraphs}

In \cite{Juhasz-Kalman-Rasmussen12}, the first author, with Juh\'{a}sz and Rasmussen, showed that for manifolds $(M_G, L_G)$, the support of sutured Floer homology essentially coincides with the set of hypertrees in a hypergraph associated to $G$. For the precise statement, let $G$ be a connected bipartite plane graph, with vertex classes $V$ and $E$ and complementary regions $R$. Let us form a trinity coloured in violet, emerald and red as in section \ref{sec:planar_graphs_links_trinities}.

On the one hand, we have the hypergraphs associated to the trinity, and hypertrees in them, as discussed in section \ref{sec:hypergraphs_hypertrees}. For current purposes it is useful to consider the hypergraphs $(E,R)$ and $(V,R)$, which are planar duals. 

On the other hand, we may perform the median construction on $G$ and split $S^3$ along the resulting Seifert surface to obtain the sutured $L$-space $(M_G, L_G)$, whose sutured Floer homology is determined by its support. The main result of \cite{Juhasz-Kalman-Rasmussen12} is that the support $\Supp(M_G,L_G)$ and the sets of hypertrees $B_{(E,R)}, B_{(V,R)}$ are essentially the same:
\[
\Supp(M_G, L_G) \cong B_{(E,R)} \cong -B_{(V,R)}.
\]

To make these equivalences precise, first note that $B_{(E,R)} \cong -B_{(V,R)}$ means that these two sets in $\Z^R$ are translates of each other. It is in fact a general property of planar dual hypergraphs $\HH$ and $\HH^*$ that the convex polytopes $\QQ_\HH$ and $\QQ_{\HH^*}$ are reflections of each other in a certain point \cite{Kalman13_Tutte}, so that $B_\HH$ and $-B_{\HH^*}$ are translates.

As for $\Supp(M_G, L_G)$, it lies in $\Spin^c(M_G, L_G)$, which is affine over $H_1(M_G)$. As we discuss in detail below in section \ref{sec:graph_to_sutured_manifold}, cutting the handlebody $M_G$ along $|R|$ discs, one disc in each complementary region of $G$, breaks $M_G$ into two balls, one above and one below the plane of the diagram, so $M_G$ has genus $|R|-1$ and $H_1(M_G) \cong \Z^{|R|-1}$. In the equivalence $\Supp(M_G,L_G) \cong B_{(E,R)}$, the affine $\Z^{|R|-1}$ space $\Spin^c(M_G, L_G)$ is identified with an affine hyperplane of $\Z^R$. We will describe this explicitly when it is needed in section \ref{sec:euler_classes_of_contact_structures}.
In particular, we will see then that hypertrees lie along a hyperplane because the sum of their coordinates is a constant, just as the number of edges in a spanning tree of a given graph is constant.

\section{Contact structures on plane bipartite graph complements}
\label{sec:contact_structures}

We now come to the main objective of the paper, which is to consider tight contact structures on the sutured manifolds $(M_G, L_G)$ of section \ref{sec:sutured_manifolds_trinities}. As before, let $G$ be a connected bipartite plane graph, with colour classes $V$ and $E$, and complementary regions $R$.

\subsection{From graph to sutured manifold}
\label{sec:graph_to_sutured_manifold}

We first analyse $(M_G, L_G)$ in detail, and develop some notation which will be useful in the sequel.
The set $S^2 \setminus G$ consists of $|R|$ open discs; we denote the closures of these discs by $\tilde{D}_r$, for each $r \in R$. See figure \ref{fig:complementary_regions} (left).

Now, consider $G$ in 3 dimensions, with 
$G\subset S^2 \subset S^3$. Then $S^2$ splits $S^3$ into two open 3-balls, one lying above and one below the plane $\R^2\subset S^2$ of our diagrams. Let the closures of these two balls be $B^+$ and $B^-$ respectively. So $B^+, B^-$ are closed balls whose intersection is $S^2$.

\begin{figure}
\begin{center}
\begin{minipage}{0.3\textwidth}
\input{red_graph_with_disc.tikz.tex}
\end{minipage}
\hspace{1cm}
\begin{minipage}{0.6\textwidth}
\input{M_G.tikz.tex}
\end{minipage}
\end{center}
\caption{Left: Bipartite plane graph $G$, and a complementary disc region $\tilde{D}_r$. Right: The sutured manifold $(M_G, L_G)$. Violet and emerald vertices of $G$ are shown. The boundary $\partial M_G = \partial N(G)$ is indicated in brown. The sutures/link $L_G$ is drawn in black. A decomposing disc $D_r$ for $(M_G, L_G)$ is shaded in grey. Two half-tubes $H_\varepsilon^+$ are shown, shaded in brown, about two edges, which meet consecutively at a violet vertex $v$; the polygon $H_v^+$ is shaded in violet.
}
\label{fig:complementary_regions}
\label{fig:M_G}
\end{figure}
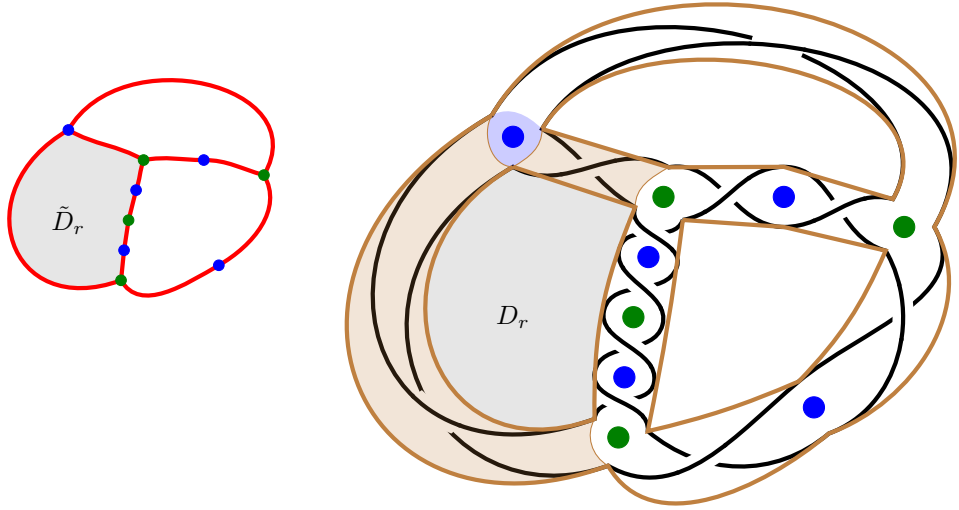

Recall that
$M_G$ is given by $S^3$ with a neighbourhood $N(G)$ of $G$ removed. The neighbourhood $N(G)$ can be regarded as a union of open balls about the vertices of $G$, and tubes about the edges of $G$. The sutures consist of the link $L_G$, which is drawn on the boundary of $N(G)$. See figure \ref{fig:M_G} (right).
It will be useful to decompose the boundary $\partial N(G) = \partial M_G$ into closed polygons as follows:
\begin{itemize}
\item
A cylinder $H_\varepsilon$ around each edge $\varepsilon$ of $G$, which is further split into two rectangles $H_\varepsilon^+ = H_\varepsilon \cap B^+$ and $H_\varepsilon^- = H_\varepsilon \cap B^-$.
\item
Two polygons $H_v^+, H_v^-$ around each vertex $v$ of $G$ of degree $\geq 2$. The polygons $H_v^+, H_v^-$ lie in $B^+, B^-$ respectively. Each polygon $H_v^\pm$ has its number of sides equal to the degree of the vertex $v$ in $G$. Each side of $H_v^\pm$ is shared with a rectangle $H_\varepsilon^\pm$, for an edge $\varepsilon$ incident to $v$.
\end{itemize}
Figure \ref{fig:M_G} (right) shows two rectangles $H_\varepsilon^+$ and a polygon $H_v^+$, which is a triangle since the vertex in the example has degree $3$.

As shown, we can take the cylinders $H_{\varepsilon}$ so that two cylinders $H_\varepsilon, H_{\varepsilon'}$ about two distinct edges $\varepsilon \neq \varepsilon'$ meet if and only if the edges $\varepsilon,\varepsilon'$ have a common endpoint $v$, and $\varepsilon,\varepsilon'$ are consecutive edges around $v$. In this case, $H_\varepsilon$ and $H_{\varepsilon'}$ intersect in a single point near $v$, which is also a vertex of $H_v^+$ and $H_v^-$. See figure \ref{fig:M_G} (right). 

When $G$ has a vertex $v$ of degree 1, the endpoint of an edge $\varepsilon$, we take $H_\varepsilon$ to be the boundary of a ``sock"-shaped neighbourhood; but by cutting along a semicircular arc (transverse to $S^2$) near $v$, at the closed end of the sock, we can regard $H_\varepsilon$ as a cylinder with one end glued up by identifications. Again we obtain two rectangles $H_\varepsilon^\pm$; the two vertices of $H_\varepsilon^+$ near $v$, and the two vertices of $H_\varepsilon^-$ near $v$, are all identified. 

In each complementary region $r \in R$ of $G$, we had a disc $\tilde{D}_r$. However not all of $\tilde{D}_r$ lies in $M_G$. We denote the intersection of $\tilde{D}_r$ with $M_G$ by $D_r$; so $D_r$ is a closed disc which is a slightly shrunken version (deformation retract) of $\tilde{D}_r$; again see figure \ref{fig:M_G} (right). Note that $D_r$ is a properly embedded, incompressible disc in $M_G$.

Now the sutures/link $L_G$ can be drawn on the boundary $\partial M_G = \partial N(G)$ so that they lie entirely in the tubes $H_\varepsilon$, and do not pass into the interior of any polygon $H_v^\pm$ about any vertex. In fact, they can be drawn so as to pass through every vertex of every rectangle $H_\varepsilon^\pm$, and intersect the interior of each $H_\varepsilon^\pm$ in a single arc, which runs diagonally across the rectangle, twisting negatively around the edge $\varepsilon$ of $G$.

If we cut $M_G$ along all the discs $D_r$, over all $r \in R$, then $M_G$ is cut into two balls $B^\pm \cap M_G$ which we denote by $M^\pm$. 
The components of the sutures $L_G$ lying in $M^+$ are the arcs of $L_G$ drawn on each $H_\varepsilon^+$, and the components of $L_G$ lying in $M^-$ are the arcs drawn on each $H_\varepsilon^-$.

Consider an $r \in R$, which corresponds to a complementary disc region $\tilde{D}_r$ of $G$, and a disc $D_r \subset M_G$. Recall that $D_r \subset \tilde{D}_r$. Proceeding around the boundary of $\tilde{D}_r$, there is a sequence of vertices and edges of $G$, with the vertices alternating in colour between violet and emerald; let there be $n$ vertices of each colour. Correspondingly, proceeding around the boundary of $D_r$, there is a sequence of arcs of rectangles $H_\varepsilon^\pm$. The sutures $L_G$ intersect $\partial D_r$ in precisely $2n$ points, which are all vertices of rectangles $H_\varepsilon^\pm$. The points of $L_G \cap \partial D_r$ are naturally in bijection with the vertices around the boundary of the complementary region $r$, i.e., with the vertices on $\partial \tilde{D}_r$.

Thus there is a natural homeomorphism $\tilde{D}_r \To D_r$ which takes each vertex $v$ (violet or emerald) on $\partial \tilde{D}_r$ to a vertex of the polygons $H_v^\pm$ on $\partial M_G$ about $v$; and which takes each edge $\varepsilon$ of $G$ on $\tilde{D}_r$ to a common edge of the rectangles $H_\varepsilon^\pm$.

\subsection{Applying the gluing theorem}
\label{sec:applying_gluing_thm}

To classify tight contact structures on $(M_G, L_G)$, we will decompose $M_G$ along the discs $D_r$ to obtain the two balls $M^+, M^-$. In other words, we apply theorem \ref{thm:Honda_gluing} to the sutured manifold $(M, \Gamma) = (M_G, L_G)$, cutting surface $S = \sqcup_{r \in R} D_r$, and cut-up manifold $M' = M^+ \cup M^-$. 

As a handlebody, $(M_G, L_G)$ is a compact, oriented, irreducible sutured 3-manifold. We may take a contact structure $\xi_{\partial}$ near $\partial M_G$ such that $\partial M_G$ is convex, with dividing set $L_G$. Recall (section \ref{sec:sutured_manifolds_trinities}) that we orient $F_G$ so that it carries the orientation of $S^2$ near violet vertices, and the opposite orientation near emerald vertices; accordingly $L_G$ runs anticlockwise around violet vertices, and clockwise around emerald vertices. 

The surface $S$ is properly embedded and incompressible in $M_G$. Every component of $\partial D_r$ intersects $L_G$ nontrivially; indeed $|\partial D_r \cap L_G|$ is the number of vertices around the boundary of the complementary region $r$. We will denote this number by $2n_r$; so the boundary of $r$ contains $n_r$ violet and $n_r$ emerald vertices. In particular, the Legendrian realisation principle applies to $\partial S$ and we may take the contact structure $\xi_\partial$ near $\partial M_G$ so that $\partial S$ is Legendrian.

The hypotheses of Honda's gluing theorem \ref{thm:Honda_gluing} are thus satisfied, and so there is a bijection between the tight components of the configuration graph $\GG(M_G, L_G, S)$, and the set $\T(M_G, L_G)$ of isotopy classes of tight contact structures on $(M_G, L_G)$.

Since cutting $M_G$ along the discs $S = \sqcup_{r \in R} D_r$ yields the very simple manifold $M' = M^+ \cup M^-$ consisting of two balls, we will be able to use the gluing theorem, together with our knowledge of the contact topology of 3-balls, to obtain a classification of the tight contact structures on $(M_G, L_G)$. 

We thus turn to an analysis of configurations on $(M_G, L_G, S)$.

\subsection{Analysing configurations}

From definition \ref{def:configuration}, a configuration on $(M_G, L_G, S)$ is a pair $(\Gamma_S, \xi')$, where $\Gamma_S$ is an (isotopy class of) set of sutures on $S$, and $\xi'$ is an (isotopy class of) contact structure on the sutured manifold $(M', \Gamma')$ obtained by cutting $M$ along $S$, drawing sutures $\Gamma_S$ on both copies of $S$, and edge-rounding.

A dividing set $\Gamma_S$ on $S$ consists of a dividing set $\Gamma_r$ on each disc $D_r$. Since $M'$ consists of the balls $M^+$ and $M^-$, each $\xi'$ is either overtwisted, or the unique tight contact structure on each of the balls $M^+$ and $M^-$ with the given boundary conditions.

We just saw that for a complementary region $r \in R$ of $G$, the disc $D_r$ has boundary $\partial D_r$ intersecting the sutures 
$L_G$ at $2n_r$ points near the $2n_r$ vertices ($n_r$ violet and $n_r$ emerald) around $\partial \tilde{D}_r$. The dividing set $\Gamma_r$ must therefore have endpoints interleaving with the $2n_r$ points of $L_G \cap \partial D_r$. In particular, $\partial \Gamma_r$ contains a point $f_{r,\varepsilon}$ for each edge $\varepsilon$ on the boundary of $\tilde{D}_r$, which lies on the cylinder $H_\varepsilon$. We can regard $f_{r,\varepsilon}$ as a point on $H_\varepsilon^+\cap H_\varepsilon^-$, midway along $\varepsilon$; each $f_{r,\varepsilon}$ thus lies near a crossing of $L_G$.

If any $\Gamma_r$ contains a closed curve, then the configuration is overtwisted. So in a potentially tight configuration, each $\Gamma_r$ consists of a collection of disjoint arcs, joining the points $f_{r,\varepsilon}$ associated to the boundary edges $\varepsilon$ of $r$. See figure \ref{fig:configuration} for an example.

\begin{figure}
\begin{center}
\input{configuration.tikz.tex}
\end{center}
\caption{A configuration on $(M_G, L_G, S)$. The dividing sets $L_G$ on $\partial M_G$, and $\Gamma_S$ on $S$, are both drawn in black. Intersection points of $L_G$ with $S$ are marked with black dots, which interleave with the points of $\partial\Gamma_S$.}
\label{fig:configuration}
\end{figure}
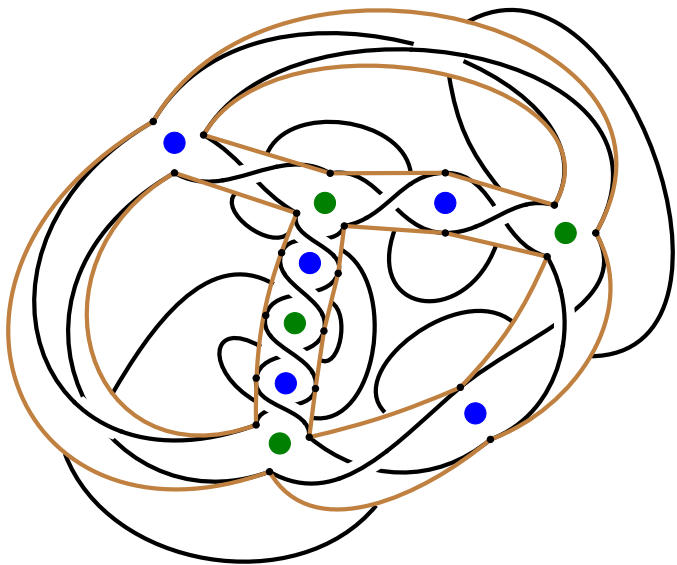

Given a dividing set $\Gamma_S$ on $S$, we obtain dividing sets on $\partial M^+$ and $\partial M^-$. The dividing set on $\partial M^+$ (resp.\ $\partial M^-$) consists of the arcs of $L_G$ in each rectangle $H_\varepsilon^+$ (resp.\ $H_\varepsilon^-$), together with the arcs of $\Gamma_S$. This gives $(M', \Gamma_S \cup L_G)$ the structure of a sutured manifold with corners along $\partial S = \sqcup_{r \in R} \partial D_r$. 

Upon rounding the corners, the dividing sets become a collection $\Gamma'$ of smooth curves. If we obtain a single connected curve on both $\partial M^+$ and $\partial M^-$, then there is a potentially tight configuration $(\Gamma_S, \xi')$, where $\xi'$ is the unique isotopy class of tight contact structures on $(M', \Gamma')$. If the sutures obtained on $\partial M^+$ or $\partial M^-$ are disconnected, then there are no potentially tight configurations. 

\begin{figure}
\begin{center} 
\begin{minipage}{0.4\textwidth}
\input{configuration_rounding_1.tikz.tex}
\end{minipage}
\hspace{1cm}
\begin{minipage}{0.4\textwidth}
\input{configuration_rounding_2.tikz.tex}
\end{minipage}
\end{center}
\caption{Dividing sets on $M^+$. Left: dividing set from $\Gamma_S$ and $L_G$. Right: Effect of rounding.}
\label{fig:configuration_rounding}
\end{figure}
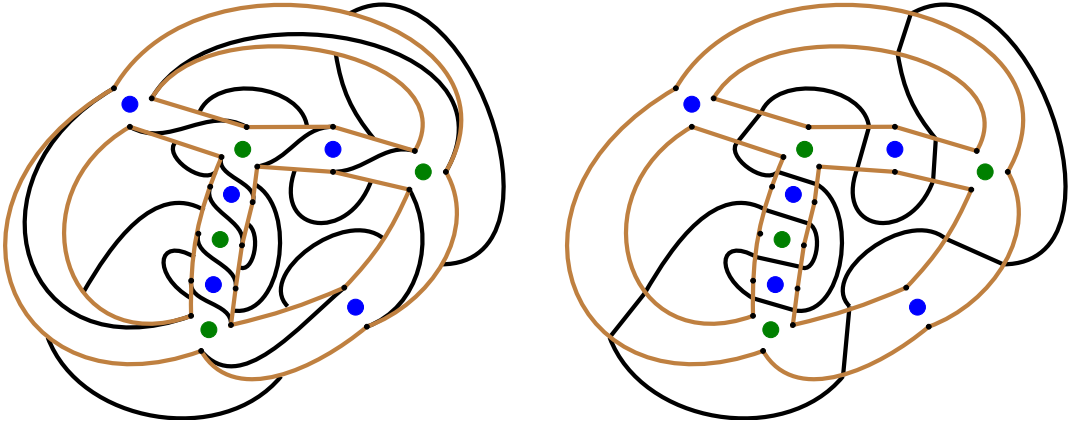

We consider the rounding process on $M^+$ in detail; see figure \ref{fig:configuration_rounding}. For each edge $\varepsilon$ of $G$, there is a diagonal arc $\gamma_\varepsilon^+$ of $L_G$ on the rectangle $H_\varepsilon^+$, making a negative half-twist around $\varepsilon$. There are two regions $r, r'$ on either side of $\varepsilon$ (possibly $r=r'$), with discs $D_r, D_{r'}$ and dividing sets $\Gamma_r, \Gamma_{r'}$. A dividing curve $\gamma_r$ of $\Gamma_r$ ends at a point $f_{r,\varepsilon}$ on $H_\varepsilon^+$, and a dividing curve $\gamma_{r'}$ of $\Gamma_{r'}$ ends at a point $f_{r',\varepsilon}$ on $H_\varepsilon^+$. Rounding the corners 
(cf.\ section \ref{sec:edge_rounding_bypasses}, especially figure \ref{fig:a1})
on either side of $H_\varepsilon^+$ results simply in straightening $\gamma_r, \gamma_\varepsilon^+, \gamma_{r'}$ into a single arc which proceeds directly across the rounded $H_\varepsilon^+$. 

Similarly, rounding the corners of $M^-$, the dividing curves $\gamma_r$ and $\gamma_{r'}$ are connected with an arc $\gamma_\varepsilon^-$ on $H_\varepsilon^-$ and the result is to straighten $\gamma_r, \gamma_\varepsilon^-, \gamma_{r'}$ into a single arc which proceeds directly across the rounded $H_\varepsilon^-$.

Thus rounding corners simply has the effect of connecting up the dividing curves $\Gamma_r$ on each disc $D_r$ directly across the rectangles $H_\varepsilon^\pm$. Hence, we can effectively ignore the dividing curves on $\partial M_G$. In fact, we can effectively ignore the entire neighbourhood $N(G)$ of $G$, and simply draw the dividing sets $\Gamma_r$ on the larger discs $\tilde{D}_r$, the complementary regions of $G$, and connect the dividing sets directly across the edges of $G$. The same diagram then describes the dividing curves on $M^+$ or $M^-$. We have now proved the following lemma.

\begin{lem}
\label{lem:potentially_tight_conditions}
Let $\Gamma_S = \sqcup_{r \in R} \Gamma_r$ be a dividing set on $S = \sqcup_{r \in R} D_r$. The following are equivalent:
\begin{enumerate}
\item 
After edge-rounding, $\Gamma_S \cup (L_G\cap M^+)$ yields a connected curve on $\partial M^+$.
\item
After edge-rounding, $\Gamma_S \cup (L_G\cap M^-)$ yields a connected curve on $\partial M^-$.
\item\label{condition_number_three}
Drawing each $\Gamma_r$ on the complementary regions $\tilde{D}_r$ of $G$ and connecting them across each edge of $G$ yields a connected curve in the plane.
\qed
\end{enumerate}
\end{lem}

If any (hence all) of the above equivalent conditions apply to $\Gamma_S$, then there is a unique potentially tight configuration $(\Gamma_S, \xi')$, where $\xi'$ is the unique isotopy class of tight contact structures on $(M', \Gamma')$, and we call $\Gamma_S$ \emph{potentially tight} accordingly. On the other hand, if the conditions do not apply, there are only overtwisted contact structures on $(M', \Gamma')$, and only overtwisted configurations with dividing set $\Gamma_S$. Thus, as discussed in section \ref{sec:gluing_tight_contact_structures}, potentially tight configurations on $(M_G, L_G, S)$ can be described simply by the dividing set $\Gamma_S$ and they correspond precisely to the potentially tight dividing sets $\Gamma_S$.

We now turn our attention to the state transitions of definition \ref{def:transition}. Perhaps the key contact-topological property of $(M_G, L_G)$ is that it turns out to be impossible to have a state transition from a potentially tight to an overtwisted configuration.

\begin{prop}
\label{prop:tight_to_tight}
If $C$ is potentially tight, and there is a state transition from $C$ to $C'$, then $C'$ is potentially tight. 
\end{prop}

\begin{proof}
Consider two configurations $C = (\Gamma_S, \xi)$ and $C' = (\Gamma'_S, \xi')$, where $\Gamma_S = \sqcup_{r \in R} \Gamma_r$ and $\Gamma'_S =  \sqcup_{r \in R} \Gamma'_r$, and let
$C$ be potentially tight. Then $\Gamma_S$ gives connected dividing sets on $\partial M^+$ and $\partial M^-$ after rounding edges. The existence of a transition $C \To C'$ means that $\xi'$ is obtained from $\xi$ by passing a bypass across some disc $D_r$; in particular, $\Gamma'_r$ is obtained from $\Gamma_r$ by bypass surgery on some attaching arc $a$ on $D_r$. 

Without loss of generality suppose that the bypass is removed from $M^+$ and added to $M^-$. Then, as discussed in section \ref{sec:edge_rounding_bypasses}, $\Gamma'_S$ is obtained from $\Gamma_S$ by inwards bypass surgery along $a$ in $\partial M^+$, and outwards bypass surgery along $a$ in $\partial M^-$. As a bypass exists in $M^+$ along $a$, it is an inner attaching arc for $M^+$.  Hence performing inwards bypass surgery on $\partial M^+$ along $a$ results in a connected dividing set on $M^+$ (after rounding edges); so by lemma \ref{lem:potentially_tight_conditions} then $\Gamma'_S$ is a potentially tight dividing set.

Now inwards bypass surgery on $\partial M^+$ is equivalent to outwards bypass surgery on $\partial M^-$. Applying lemma \ref{lem:potentially_tight_conditions} again, since inwards bypass surgery yields a connected dividing set on $M^+$, outwards bypass surgery on $\partial M^-$ along $a$ must result in a connected dividing set on $M^-$. Hence $a$ is outer for $M^-$.

Given a tight ball, removing a bypass along an inner attaching arc, or adding a bypass along an outer attaching arc, results in another tight ball. Thus $\xi'$ is tight, and $C'$ is potentially tight.
\end{proof}

Recall from section \ref{sec:gluing_tight_contact_structures} that the configuration graph $\GG(M_G, L_G, S)$ (definition \ref{def:configuration_graph}) has vertices given by configurations in $\Conf(M_G, L_G, S)$, and edges given by state transitions. Further, a component of $\GG(M_G, L_G, S)$ is tight if all its configurations are potentially tight; a configuration is tight if it lies in a tight component (definition \ref{def:tight_configs}). Proposition \ref{prop:tight_to_tight} says that there is no edge connecting a potentially tight configuration with an overtwisted configuration, so we immediately obtain the following.

\begin{prop}
\label{prop:potentially_tight_tight}
All potentially tight configurations on $(M_G, L_G, S)$ are tight.
\qed
\end{prop}

In light of this result, in our context we can simply refer to potentially tight configurations, or their dividing sets, as \emph{tight}.
We also obtain an additional nice fact about state transitions in $(M_G, L_G, S)$.

\begin{lem}
\label{lem:existence_of_state_transitions}
Let $\Gamma_S = \sqcup_{r \in R} \Gamma_r$ be a tight configuration and let $a$ be an attaching arc on some $D_r$. Then there exists precisely one state transition involving a bypass with attaching arc $a$.
\end{lem}

\begin{proof}
Consider the dividing set $\Gamma'$ obtained by edge-rounding on $\partial M'$. With respect to $\Gamma'$, the arc $a$ is inner on precisely one of $M^+$ or $M^-$; without loss of generality suppose $M^+$. Then, as in the proof of proposition \ref{prop:tight_to_tight}, $a$ is outer on $M^-$. So there is a state transition removing a bypass from $M^+$, and attaching it to $M^-$, along $a$, but not the other way. 
The statement now follows from Lemma \ref{lem:unique_bypass}.
\end{proof}

\subsection{Scheme of the proof}

We have now essentially reduced the proof of theorem \ref{thm:classification_of_tight_contact_structures} to a purely combinatorial problem. So we pause at this point to give an overview of the remaining portion of the proof.

To classify the tight contact structures on $(M_G, L_G)$, by the gluing theorem we must compute $\pi_0 (\G_0 (M_G, L_G, S))$, the connected components of the graph of tight configurations. By proposition \ref{prop:potentially_tight_tight}, there is no distinction between potentially tight and tight configurations. We only need to determine which tight configurations are connected to which others by state transitions. The tight configurations are given by dividing sets on each complementary disc $\tilde{D}_r$ without closed curves (i.e., chord diagrams or crossingless matchings) which join across edges of $G$ to give a connected curve. The transitions are bypass surgeries localised on a single $D_r$. 

Our goal is to show that there is a bijection between the components of $\GG_0 (M_G, L_G, S)$, and hypertrees in $(E,R)$. We prove this as follows:
\begin{itemize}
\item 
In section \ref{sec:hypertree_configurations} we introduce a special class of configurations associated to spanning trees of $G_V$, which we call ``tree-hugging". We show tree-hugging configurations are tight. 
\item
In section \ref{sec:hypertrees_Euler_class} we show that the Euler class of a tree-hugging configuration is determined by the associated hypertree of $(E,R)$. Hence two tree-hugging dividing sets with distinct associated hypertrees cannot be connected by state transitions.
\item
In section \ref{sec:tight_tree-hugging} we show that any tight configuration is connected to a tree-hugging configuration by state transitions.
\item
Finally, in section \ref{sec:hypertrees_same_hypergraph} we show that tree-hugging dividing sets with the same associated hypertree are connected by state transitions.
\end{itemize}

Thus, every component of $\G_0 (M_G, L_G, S)$ contains tree-hugging configurations, all representing the same hypertree; and every hypertree is represented by a configuration. This gives the desired bijection.

\subsection{Configurations from spanning trees}
\label{sec:hypertree_configurations}

Let, as usual, $G$ be a connected plane bipartite graph, with colour classes $V$ and $E$, and complementary regions $R$.
As we saw in section \ref{sec:hypergraphs_hypertrees}, $G$ is part of a trinity, where one of the six 
hypergraphs is $(E,R)$ with vertices $E$ (emerald vertices of $G$) and hyperedges $R$ (complementary regions, i.e.\ faces, of $G$). The associated bipartite graph is the violet graph $G_V$ of section \ref{sec:planar_graphs_links_trinities}, with vertex classes $E$ and $R$, the latter viewed as one red vertex in each face of $G$. 

At a red vertex $r \in R$ the trinity has $2n_r$ 
alternating violet and emerald edges, and so the degree of $r$ in $G_V$ is $n_r$.
Consider a hypertree $f\colon R \To \Z_{\ge0}$ of $(E,R)$ and an associated spanning tree $T$
of $G_V$ with degree $f(r)+1$ at each $r \in R$. (Note $1 \leq f(r) + 1 \leq n_r$.) We now derive a configuration $\Gamma_T$, via its dividing set, from the spanning tree $T$.

We construct $\Gamma_T$ as the boundary of a 
tubular neighbourhood $U$ of $T$ in $S^2$. The curve $\partial U\subset S^2$ intersects each complementary region $\tilde{D}_r$ in a collection of arcs. The endpoints of the arcs lie on $\partial \tilde{D}_r$ and interleave with the violet and emerald vertices. See figure \ref{fig:spanning_tree}(left) for an example.

As seen in section \ref{sec:graph_to_sutured_manifold}, the decomposing discs $D_r$ of $M_G$ are slightly shrunken deformation retracts of the $\tilde{D}_r$, with natural homeomorphisms $\tilde{D}_r \To D_r$ taking each vertex $v$ to a vertex of the polygons $H_v^\pm$ on $\partial M_G$ about $v$. Consider the images of $T$ and $U$ under this homeomorphism to regard them as lying on each $D_r$.

Thus, on each $D_r$, the tree $T$ consists of the red vertex $r$, together with $f(r)+1$ edges connecting $r$ to points of $\partial D_r$ which are vertices of the polygons $H_e^\pm$ near some emerald vertices $e$. 
The neighbourhood $U$ of $T$, in each $D_r$, consists of a ``central component", the boundary of a regular neighbourhood of the edges of $T$ just described, together with ``outer components", each being a small neighbourhood of a vertex of $D_r$ corresponding to an emerald vertex to which $r$ is not adjacent in $T$. The endpoints of the arcs of $\partial U$ 
interleave with the points of $L_G \cap S$, 
and hence can be taken to be the points $f_{r,\varepsilon}$ associated to the incidences between faces $r$ and edges $\varepsilon$ of $G$. See figure \ref{fig:spanning_tree} (right).

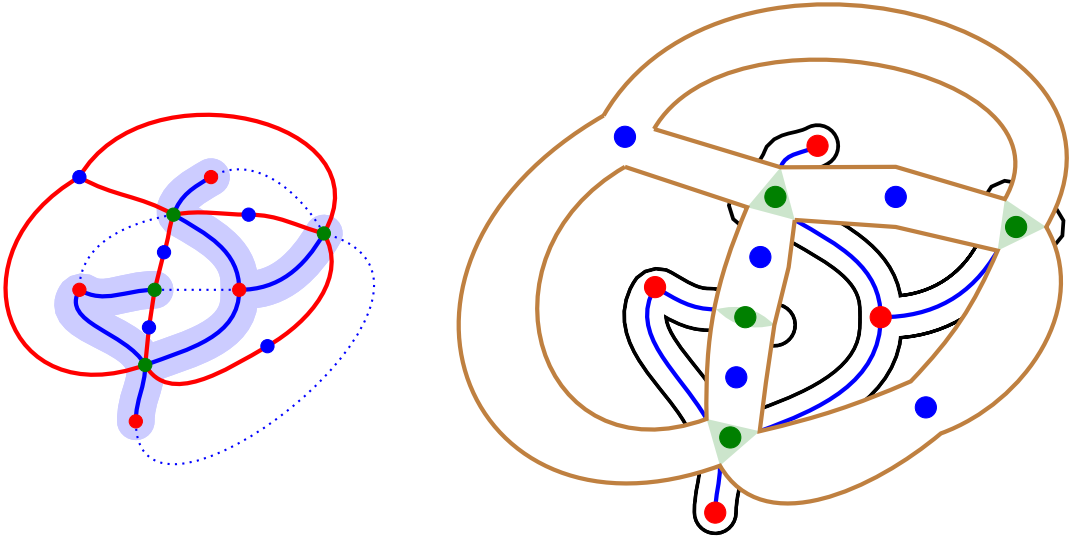
\begin{figure}
\begin{center}
\begin{minipage}{0.4\textwidth}
\input{spanning_tree_1.tikz.tex}
\end{minipage}
\begin{minipage}{0.55\textwidth}
\input{spanning_tree_2.tikz.tex}
\end{minipage}
\end{center}
\caption{Construction of the dividing set $\Gamma_T$ from a spanning tree $T$ of $G_V$. Left: the tree $T$ and the graph $G=G_R$; the neighbourhood $U$ is shaded. Right: $T$ and $\Gamma_T$ drawn on the cutting discs $S$ of $(M_G, L_G)$. The polygons $H_e^\pm$ are drawn about emerald vertices on $\partial M_G$. Note that the dividing set so obtained on each disc agrees with the dividing set in figure \ref{fig:configuration_rounding}.}
\label{fig:spanning_tree}
\end{figure}

We denote the curve so obtained on $S = \sqcup_{r \in R} D_r$ as $\Gamma_T$. 
Since $T$ is a spanning tree, it is clear that $\Gamma_T$ satisfies condition \ref{condition_number_three} of lemma \ref{lem:potentially_tight_conditions}. Therefore it also satisfies the other conditions and, as explained after the lemma, we may refer to $\Gamma_T$ as the dividing set of a (tight) configuration.
Because $\Gamma_T$ runs close to the tree $T$, 
we name these configurations as follows.

\begin{defn}
The configuration $\Gamma_T$ is called the \emph{$T$-hugging} configuration; we say $\Gamma_T$ \emph{hugs} $T$. A potentially tight
configuration $\Gamma_S$ is called \emph{tree-hugging} if it hugs some tree. A contact structure $\xi$ on $(M_G, L_G)$ is \emph{tree-hugging} if it arises from a tree-hugging configuration.
\end{defn}

Note that a contact structure $\xi$ on $(M_G, L_G)$ (more precisely, its isotopy class) in general corresponds to many configurations, some of which may be tree-hugging and some not; as long as one such configuration is tree-hugging, we say $\xi$ is tree-hugging. From the discussion above and proposition \ref{prop:potentially_tight_tight}, the following is clear.

\begin{prop}
\label{prop:tree-hugging_tight}
Tree-hugging contact structures are tight.
\qed
\end{prop}

\subsection{Hypertrees and Euler class}
\label{sec:hypertrees_Euler_class}

Recall from section \ref{sec:sutured_manifolds_trinities}  that $L_G$ and $F_G$ are assigned orientations so that, in our diagrams, $L_G$ runs anticlockwise around violet vertices, and clockwise around emerald vertices. 

Since $M^-$ lies below the projection plane, $\partial M^-$ has outward normal pointing upwards out of the projection plane, and induces the usual orientation on $\R^2$. Hence violet vertices (actually the polygons $H_v^-$ at such vertices $v$) correspond to positive regions on $\partial M^-$, and emerald vertices (actually their polygons $H_e^-$) correspond to negative regions. Conversely, $\partial M^+$ has violet vertices (polygons $H_v^+$) corresponding to negative regions, and emerald vertices ($H_e^+$) corresponding to positive regions.

We orient the discs $D_r$ according to the projection plane, hence as subsets of $\partial M^-$. Recall that the endpoints of any dividing set on $D_r$ are the points $f_{r,\varepsilon}$, one for each edge $\varepsilon$ around the boundary of the region $r$; the point $f_{r,\varepsilon}$ lies midway along $\varepsilon$, on the tube $H_\varepsilon$ around $\varepsilon$. The points $f_{r,\varepsilon}$ split $\partial D_r$ into arcs, and each such arc is associated to a vertex $v$ around the boundary of $r$. (Precisely, each arc intersects a single pair of polygons $H_v^\pm$, at one of their shared vertices.) The colour of $v$ alternates between violet and emerald. Since $D_r$ is oriented as $\partial M^-$, the arcs near violet vertices lie in $R_+$, and the arcs near  emerald vertices lie in $R_-$; here $R_\pm$ refer to the signed regions of an arbitrary dividing set on $D_r\subset\partial M^-$.

A configuration $(\Gamma_S, \xi')$ glues into a contact structure $\xi$ with Euler class $e(\xi) \in H^2(M_G, \partial M_G)$. Since the homology classes $[D_r]$ generate $H_2 (M_G, \partial M_G)$, it suffices to evaluate $e(\xi)$ on these discs. 
We compute each $e(\xi)[D_r]$ implicitly with respect to a section of $\xi$ over $\partial M_G$ which is positively tangent to the oriented Legendrian boundary $\partial D_r$, as discussed in section \ref{sec:convex_surfaces}. (Each $\partial D_r$ is non-separating on $\partial M_G$. Note we may use different boundary sections for different discs.)

If $(\Gamma_S, \xi')$, or just $\Gamma_S$, is a tight configuration, we can refer to the Euler class $e(\xi')$ as the \emph{Euler class of the configuration} and denote it with $e(\Gamma_S)$.

Now let $T$ be a spanning tree of $G_V$ with hypertree $f\colon R \To \Z_{\ge0}$ and let $\Gamma_T$ be the $T$-hugging configuration. Recall that the degree of $T$ at a red vertex $r$ is $f(r)+1$, and the degree of $r$ in $G_V$ is $n_r$. We now calculate the Euler class $e(\Gamma_T)$.

\begin{lem}
\label{lem:Euler_class_from_hypertree}
Let $T$ be a spanning tree of $G_V$ with hypertree $f$, and let $\Gamma_T$ be the $T$-hugging configuration. Then 
\[
e(\Gamma_T)[D_r] = 2f(r)- n_r +1.
\]
\end{lem}

\begin{proof}
According to (\ref{eq:eulerclass}) we have $e(\Gamma_T)[D_r] = \chi(R_+) - \chi(R_-)$, where $R_\pm$ are the signed complementary regions of $\Gamma_T$ on $D_r$. As $R_+$ and $R_-$ consist of discs, it suffices to count the number of components in $R_+$ and $R_-$. As discussed above, regions near violet (resp.\ emerald) vertices lie in $R_+$ (resp.\ $R_-$).

Using the homeomorphism $\tilde{D}_r \cong D_r$, we can regard $\Gamma_T$ as drawn on $\tilde{D}_r$ rather than $D_r$, which makes for simpler notation; the region $R_+$ (resp.\ $R_-$) then contains the violet (resp.\ emerald) vertices.

Now $T \cap \tilde{D}_r$ consists of the $f(r)+1$ edges connecting $r$ to emerald vertices, together with the remaining $n_r -f(r)-1$ emerald vertices not connected to $r$. As discussed in section \ref{sec:hypertree_configurations}, the regular neighbourhood $U$ of $T$ intersects $\tilde{D}_r$ in a central component (a regular neighbourhood of the $f(r)+1$ edges emanating from $r$) and $n_r - f(r) - 1$ outer components.

Thus, the regular neighbourhood $U$ of $T$ has $n_r - f(r)$ components in $D_r$, which are the components of $R_-$. Now $\Gamma_T$ consists of $n_r$ dividing curves drawn on $\tilde{D}_r$, cutting $\tilde{D}_r$ into $n_r + 1$ components. Hence $R_+$ has $f(r)+1$ components, and $e(\Gamma_T)[D_r] = 2f(r) - n_r + 1$.
\end{proof}

Therefore the Euler class of $\Gamma_T$ determines the hypertree of $T$, and vice versa. Hence, although several spanning trees might correspond to the same hypertree, all such spanning trees yield tree-hugging configurations with the same Euler class. Moreover, state transitions do not affect the Euler class of the contact structure induced by configurations (either on $M'$ or on $M_G$). From this we immediately obtain the following proposition.

\begin{prop}
\label{prop:distinct_hypertrees_distinct_contact_structures}
Let $T_1, T_2$ be spanning trees of $G_V$ with associated hypertrees $f_1, f_2\colon R \To \Z_{\ge0}$. If the tree-hugging configurations $\Gamma_{T_1}, \Gamma_{T_2}$ are related by state transitions, then $f_1 = f_2$.
\qed
\end{prop}

\subsection{Tight contact structures are tree-hugging}
\label{sec:tight_tree-hugging}

We have seen that tree-hugging contact structures are tight. This section is devoted to proving the converse.

\begin{prop}
\label{prop:tight_implies_tree-hugging}
Any tight contact structure on $(M_G, L_G)$ is tree-hugging.
\end{prop}

The proof amounts to showing that any tight configuration is related by state transitions to a tree-hugging configuration. We prove this directly.

We first give a characterisation of tree-hugging configurations. We saw in section \ref{sec:hypertrees_Euler_class} that in a tight configuration $\Gamma_S = \sqcup_{r \in R} \Gamma_r$, each $\Gamma_r$ consists of $n_r$ dividing curves, cutting $D_r$ into positive and negative regions $R_+$ and $R_-$. The $2 n_r$ endpoints of $\Gamma_r$ are the points $f_{r,\varepsilon}$ associated to the edges $\varepsilon$ of $G$ around the boundary of $r$. These points $f_{r,\varepsilon}$ cut $\partial D_r$ into $2 n_r$ arcs, which alternately lie in $R_+$ and $R_-$; we call these arcs the \emph{positive and negative signed arcs} of $\partial D_r$. Viewed on $\tilde{D}_r$ via the homeomorphism $\tilde{D}_r \cong D_r$, each positive (resp.\ negative) arc contains exactly one violet (resp.\ emerald) vertex.

Each component of $R_+$ (resp.\ $R_-$) is a disc containing some number of positive (resp.\ negative) signed arcs in its boundary.

\begin{defn}
The \emph{valence} $v(c)$ of a component $c$ of $R_\pm$ is the number of signed arcs in its boundary.
\end{defn}

Note that a component of $R_\pm$ has valence $1$ if and only if its boundary consists of a single signed arc, and a single dividing curve on $D_r$. Since each of the $n_r$ positively (resp.\ negatively) signed arcs counts precisely $1$ towards the valence of precisely one component of $R_+$ (resp.\ $R_-$), the sum of the valences over all components $c_-$ of $R_-$ is $n_r$; similarly for the components $c_+$ of $R_+$:
\begin{equation}
\label{eq:sum_of_valences}
\sum_{c_-} v(c_-) = \sum_{c_+} v(c_+) = n_r.
\end{equation}

\begin{lem}
\label{lem:tree-hugging_characterisation}
A tight configuration $\Gamma_S = \sqcup_{r \in R} \Gamma_r$ is tree-hugging if and only if in each disc $D_r$, all components of $R_-$, with at most one exception, have valence $1$.
\end{lem}

Thus tree-hugging configurations have minimal valence $1$ in as many components of $R_-$ as possible, and concentrate ``excess valence" above $1$ into a single component.

\begin{proof}
Throughout this proof we view each $\Gamma_r$ on $\tilde{D}_r$, via the homeomorphism $\tilde{D}_r \cong D_r$.

Suppose $\Gamma_S = \Gamma_T$ hugs $T$, a spanning tree of $G_V$ with hypertree $f\colon R \To \Z_{\ge0}$. 
Then the region $R_-$ of $\tilde{D}_r$ is a neighbourhood of $T\cap\tilde{D}_r$, and 
only its central component can have valence more than 1.

Conversely, suppose $\Gamma_S$ is a tight configuration such that on each $\tilde{D}_r$, all components of $R_-$ have valence $1$, with at most one exception. We will find a spanning tree $T$ of $G_V$ such that $\Gamma_S$ hugs $T$.

Consider an element $r \in R$ and the dividing set $\Gamma_r$ on $\tilde{D}_r$. Take a component of $R_-$ with maximal valence and call it the ``central component"; this component has some emerald vertices on its boundary which we call ``centrally connected". Every non-central component of $R_-$ has valence $1$; we call these ``outer components". An outer component must have boundary consisting of a single negative arc, and a single arc of $\Gamma_r$, and must therefore be a regular neighbourhood of an emerald vertex in $\tilde{D}_r$.

After an isotopy of $\Gamma_r$ relative to endpoints if necessary, we may assume the red vertex $r$ lies in the central component, and the edges of $G_V$ from $r$ to the centrally connected emerald vertices also lie in the central component.

Let $T_r$ be the graph consisting of the red vertex $r$, the edges from $r$ to the centrally connected emerald vertices, and all the emerald vertices on $\partial \tilde{D}_r$. Let $T$ be the subgraph of $G_V$ obtained by taking the union of all the $T_r$.

On each $\tilde{D}_r$, the region $R_-$ deformation retracts onto $T_r$.
The union of these deformation retractions, over all $r \in R$, shows that when $\Gamma_r$ is drawn in the complementary region $\tilde{D}_r$ of $G$ and connected across the edges of $G$, the negative region $R_-$ deformation retracts onto $T$. Indeed, $\Gamma_S$ is the boundary of a regular neighbourhood of $T$.

Since $\Gamma_S$ is a tight configuration, by lemma \ref{lem:potentially_tight_conditions}, when each $\Gamma_r$ is drawn on $\tilde{D}_r$ and connected across the edges of $G$, we obtain a single connected curve. Hence the $R_-$ so obtained on $S^2$ is a single disc, which deformation retracts onto $T$. So $T$ is a tree and, since it contains all emerald and red vertices, a spanning tree of $G_V$. As $\Gamma_S$ is the boundary of a regular neighbourhood of $T$, it hugs $T$.
\end{proof}

To prove proposition \ref{prop:tight_implies_tree-hugging}, we demonstrate a sequence of state transitions, starting from any tight configuration, and ending at a configuration satisfying the conditions of lemma \ref{lem:tree-hugging_characterisation}. The idea is to successively concentrate excess valence into fewer components of $R_-$, as in the following lemma.

\begin{lem}
\label{lem:increase_max_valence}
Let $\Gamma_S = \sqcup_{r \in R} \Gamma_r$ be a tight configuration. Suppose that, for some $r_0 \in R$, there are two components of $R_-$ on $D_{r_0}$ with valence $> 1$. Let 
$\V$
be the maximum valence of the components of $R_-$ on $D_{r_0}$. Then there is a state transition from $\Gamma_S$ to another tight configuration $\Gamma'_S = \sqcup_{r \in R} \Gamma'_r$, which is identical to $\Gamma_S$ outside of $D_{r_0}$, and such that on $D_{r_0}$, the maximum valence of a component of $R_-$ is greater than 
$\V$.
\end{lem}

\begin{proof}
Let $c$ be a component of $R_-$ on $D_{r_0}$ with maximal valence $\V$. Its boundary consists of $\V$ negative arcs on $\partial D_{r_0}$, and $\V$ arcs of $\Gamma_{r_0}$. On the other side of each of these $\V$ arcs of $\Gamma_{r_0}$ lie components $c^+_1, \ldots, c^+_\V$ of $R_+$ adjacent to $c$; and in turn these positive components are adjacent to negative components $c^-_1, c^-_2, \ldots, c^-_M$ (other than $c$) for some non-negative integer $M$. 

We claim that some $c^-_j$ has valence $v(c^-_j) > 1$. If not, then every $c^-_j$ has valence $1$, so no positive regions other than the $c^+_i$ can be adjacent to the $c^-_j$, and $c,c^-_1, \ldots, c^-_M$ must be a complete list of the components of $R_-$. This contradicts the existence of two components with valence $>1$. So there exist $i$ and $j$ such that the components $c, c^+_i, c^-_j$ are consecutively adjacent and $v(c^-_j) > 1$.

Now let $a$ be an attaching arc which begins on the common boundary of $c$ and $c^+_i$, runs through $c^+_i$ and $c^-_j$, and ends on a distinct (not adjacent to $c_i^+$) dividing curve on the boundary of $c^-_j$. 
(Note that a component of $R_-$ and a component of $R_+$ can share at most one boundary arc.)
See figure \ref{fig:max_valence_increase}.

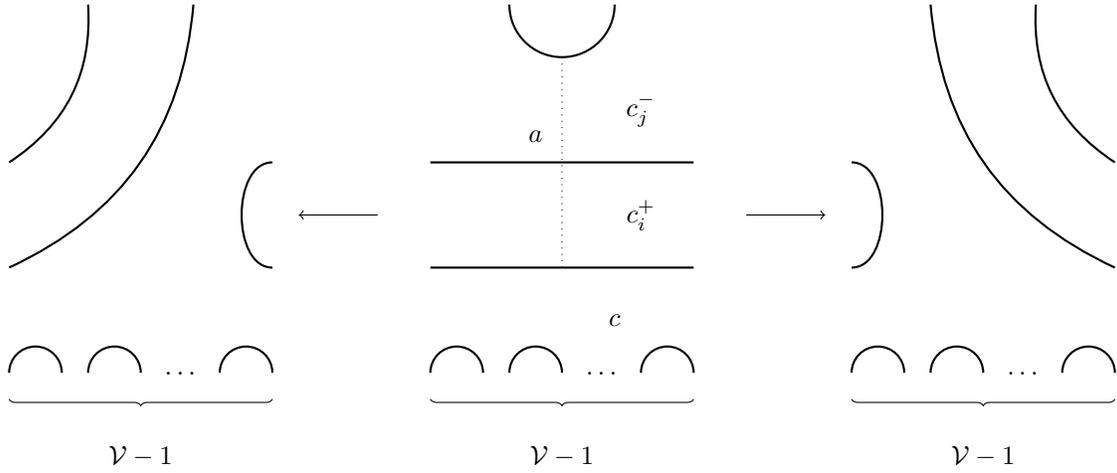
\begin{figure}
\begin{center}
\input{max_valence_increase.tikz.tex}
\end{center}
\caption{Increasing maximum valence. The region $c$ has maximum valence $\V$, so there are $\V$ dividing curves along its boundary, namely the common border with $c_i^+$, and $\V-1$ other dividing curves, shown at the bottom of the diagram. Performing either downwards (left) or upwards (right) bypass surgery yields a region with valence at least $\V+1$.}
\label{fig:max_valence_increase}
\end{figure}

By lemma \ref{lem:existence_of_state_transitions}, there is one (and only one) state transition from $\Gamma_S$ along $a$. This state transition goes from $\Gamma_S$ to another tight configuration $\Gamma'_S = \sqcup_{r \in R} \Gamma'_r$, which only differs from $\Gamma_S$ in effecting a bypass surgery along $a$. Now we observe that, whichever direction we perform bypass surgery along $a$, we obtain a component with valence strictly greater than $\V$.
\end{proof}

\begin{proof}[Proof of proposition \ref{prop:tight_implies_tree-hugging}]
Let $\xi$ be a tight contact structure on $(M_G, L_G)$ corresponding to a configuration $\Gamma_S = \sqcup_{r \in R} \Gamma_r$. Consider a disc $D_r$ with dividing set $\Gamma_r$ and negative region $R_-$. If there are two components of $R_-$ with valence $\geq 2$, then we repeatedly apply lemma \ref{lem:increase_max_valence}, successively performing state transitions to tight configurations, at each stage increasing the maximum valence of the components of $R_-$ on $D_r$.

By (\ref{eq:sum_of_valences}), the sum of the valences of the components of $R_-$ remains constant, so lemma \ref{lem:increase_max_valence} cannot be applied indefinitely. Thus, eventually we arrive at a tight configuration where all components of $R_-$ on $D_r$ have valence $1$, with at most one exception.

Doing the same for each disc $D_r$, we arrive at a configuration $\Gamma'_S$ which satisfies the condition of lemma \ref{lem:tree-hugging_characterisation}, hence is tree-hugging. As $\Gamma'_S$ is obtained from $\Gamma_S$ by a sequence of state transitions, $\Gamma'_S$ also yields the contact structure $\xi$, so $\xi$ is tree-hugging.
\end{proof}

\subsection{Spanning trees with the same hypertree yield equivalent structures}
\label{sec:hypertrees_same_hypergraph}

In this section we prove the following proposition.

\begin{prop}
\label{prop:same_degrees_same_contact_structure}
Let $T, T'$ be spanning trees of $G_V$ with the same associated hypertree $f\colon R \To \Z_{\ge0}$. Then the tree-hugging configurations $\Gamma_{T}, \Gamma_{T'}$ are related by state transitions.
\end{prop}

Thus, two tree-hugging configurations with the same hypertree yield the same (isotopy class of) contact structure. Proposition \ref{prop:same_degrees_same_contact_structure} is proved in the following two steps.

\begin{lem}
\label{lem:spanning_trees_by_local_moves}
Let $T,T'$ be spanning trees of $G_V$ with the same associated hypertree $f$. Then there is a sequence $T = T_0, T_1, \ldots, T_n = T'$ of spanning trees of $G_V$, where each $T_i$ has associated hypertree $f$, and each $T_{i+1}$ is obtained from $T_i$ by removing one edge and adding one edge.
\end{lem}

\begin{lem}
\label{lem:state_transitions_between_trees}
Let $T,T'$  be spanning trees of $G_V$ with associated hypertree $f$, where $T'$ is obtained from $T$ by removing one edge and adding one edge. Then the tree-hugging configurations $\Gamma_T, \Gamma_{T'}$ are related by a sequence of state transitions.
\end{lem}

In lemma \ref{lem:spanning_trees_by_local_moves}, as $T_{i+1}$ is obtained from $T_i$, the edge removed is adjacent to some red vertex $r$; the edge added must also be adjacent to $r$, in order to maintain degree $f(r)+1$ at $r$. Thus, the operation which yields $T_{i+1}$ from $T_i$ is localised at a single red vertex $r$. Lemma \ref{lem:spanning_trees_by_local_moves} follows from a theorem of the first author; its proof is essentially contained in the proof of theorem 10.1 of \cite{Kalman13_Tutte}, and we summarise the relevant arguments here.

Consider the planar dual $G_V^*$ of the violet graph $G_V$, as discussed in section \ref{sec:arborescence_number}. Its vertex set is $V$, and its edges are naturally directed from black to white triangles of 
the trinity. Spanning trees of $G_V$ and $G_V^*$ are in bijective correspondence via planar duality.

We choose a root vertex $v_0 \in V$ of $G_V^*$ and consider arborescences in $G_V^*$. Recall (section \ref{sec:arborescence_number}) that an arborescence is a spanning tree whose edges point away from the root; and that in a balanced directed graph such as $G_V^*$, the number of arborescences does not depend on the choice of root.

In theorem 10.1 of \cite{Kalman13_Tutte} (adapted to present notation), the first author constructed a bijection from arborescences of $G_V^*$ to hypertrees in $(E,R)$. Given an arborescence $A$ in $G_V^*$, take the dual $A^*$, which is a spanning tree of $G_V$, and consider its hypertree $f_A\colon R \To \Z_{\ge0}$ in $(E,R)$. It is shown that the map $A \mapsto f_A$ is a bijection.

For present purposes, we only need to consider the proof that $A \mapsto f_A$ is surjective. The task is to take a hypertree $f\colon R \To \Z_{\ge0}$ of $(E,R)$ and to find an arborescence $A$ of $G_V^*$ such that $f_A = f$. In order to find $A$, we start from an arbitrary spanning tree $T$ of $G_V$ with hypertree $f$; its dual $T^*$ is a spanning tree of $G_V^*$ but may not be an arborescence. A sequence of modifications is then made to $T^*$, and hence to $T$, yielding a sequence of spanning trees $T = T_0, T_1, \ldots, T_n = T'$ arriving at a tree $T'$ such that $T'^*$ is an arborescence. Each modification consists of removing a single edge of $T_i^*$ (namely, one of the closest edges to the root where $T_i^*$ fails to be an arborescence), and adding another edge; this corresponds to adding and removing an edge of $T_i$. This is done in such a way that all $T_i$ represent the same hypertree $f$. Thus $f = f_{T'^*}$ and we can take the arborescence $A = T'^*$.

We summarise what we need from the above argument in the following theorem.
\begin{thm}
\label{thm:tree_to_arborescence}
Let $T$ be a spanning tree of $G_V$ realising the hypertree $f\colon R \To \Z_{\ge0}$. Choose a root vertex $v_0 \in V$. Then there is a sequence $T = T_0, T_1, \ldots, T_n$ of spanning trees of $G_V$, such that
\begin{enumerate}
\item 
each $T_{i+1}$ is obtained from $T_i$ by adding one edge and removing one edge; 
\item
each $T_i$ realises the hypertree $f$; and
\item
$T_n^*$ is an arborescence.
\qed
\end{enumerate}
\end{thm}

\begin{proof}[Proof of lemma \ref{lem:spanning_trees_by_local_moves}]
Choose a root vertex $v_0 \in V$ arbitrarily, so that we may speak of arborescences in $G_V^*$. Given the spanning trees $T,T'$, theorem \ref{thm:tree_to_arborescence} provides sequences of spanning trees $T = T_0, T_1, \ldots, T_m$ and $T' = T'_0, T'_1, \ldots, T'_n$, where each $T_{i+1}$ (resp.\ $T'_{i+1}$) is obtained from $T_i$ (resp.\ $T'_i$) by adding one edge and removing one edge, each $T_i$ and $T'_i$ has hypertree $f$, and $T_m^*, T_n^{'*}$ are arborescences. 

Now $T_m^*$ and $T_n^{'*}$ are arborescences arising from the same hypertree $f$, i.e., $f = f_{T_m^*} = f_{T_n^{'*}}$. As discussed above, \cite[thm.\ 10.1]{Kalman13_Tutte} claims that the map $A \mapsto f_A$ is bijective, so $T_m^* = T_n^{'*}$ and hence $T_m = T'_n$. We conclude that the sequence
\[
T = T_0, T_1, \ldots, T_m = T'_n, T'_{n-1}, \ldots, T'_0 = T'
\]
has the desired properties.
\end{proof}

Now we turn to the proof of lemma \ref{lem:state_transitions_between_trees}. So suppose we have spanning trees $T,T'$ of $G_V$ with the same hypertree $f$, where $T'$ is obtained from $T$ by removing an edge and adding an edge. As discussed above, the removed edge and the added edge must both be incident to the same red vertex $r_0$. Hence the tree-hugging configurations $\Gamma_T, \Gamma_{T'}$ agree on each disc $D_r$ other than $D_{r_0}$. Denote the common restriction of $\Gamma_T$ and $\Gamma_{T'}$ to the (``external'') discs other than $D_{r_0}$ by $\Gamma_X$.

We must show that there is a sequence of state transitions from $\Gamma_T$ to $\Gamma_{T'}$. In fact, we will prove the following stronger result. 

\begin{prop}
\label{prop:transitions_on_disc}
Let $\Gamma, \Gamma'$ be dividing sets on the disc $D_{r_0}$ such that $\Gamma \cup \Gamma_X$ and $\Gamma' \cup \Gamma_X$ are tight configurations. Then there is a sequence of state transitions, which only involve bypass surgeries on $D_{r_0}$, from $\Gamma \cup \Gamma_X$ to $\Gamma' \cup \Gamma_X$.
\end{prop}

The key idea is that this statement is essentially identical to lemma \ref{lem:chord_diagrams_connected} about tight contact structures on cylinders. But first, we show how configurations such as $\Gamma \cup \Gamma_X$ correspond to contact structures on cylinders.

First we recall lemma \ref{lem:potentially_tight_conditions}, that is that a dividing set $\Gamma_S$ is tight if and only if, drawing each $\Gamma_r$ on the complementary region $\tilde{D}_r$ of $G$ and connecting them across the edges of $G$, we obtain a connected curve on the compactified plane $S^2$.

In such a diagram of the configurations $\Gamma \cup \Gamma_X$ and $\Gamma' \cup \Gamma_X$, the dividing sets differ on the disc $\tilde{D}_{r_0}$ but are identical on the other complementary regions of $G$. We can thus regard $D_{r_0}$ as the interior of a circle $C$ on $S^2$, and the other complementary regions as forming the exterior of this circle. Then $\Gamma, \Gamma'$ are dividing sets on the interior of $C$, and $\Gamma_X$ is a dividing set on the exterior of $C$.

Creasing the sphere $S^2$ along two parallel copies of $C$, we can obtain a 
sutured manifold with corners
of the form $\M(\cdot, \cdot)$, as discussed in section \ref{sec:dividing_sets_cylinders}. We can do this in such a way that the dividing sets on the interior and exterior of the circle $C$ become the dividing sets on the top and bottom of the cylinder respectively.

Thus, we can crease the sphere so that the configuration $\Gamma \cup \Gamma_X$ yields the cylinder $\M(\Gamma_X^*, \Gamma)$, and the configuration $\Gamma' \cup \Gamma_X$ yields the cylinder $\M(\Gamma_X^*, \Gamma')$. Here $\Gamma_X^*$ is the common dividing set on the bottom of the cylinders, obtained from the dividing set $\Gamma_X$ on the exterior of $C$ after creasing the sphere into a cylinder.

\begin{proof}[Proof of proposition \ref{prop:transitions_on_disc}]
As $\Gamma \cup \Gamma_X$ is a tight configuration, drawing $\Gamma \cup \Gamma_X$ on $S^2$ in the complementary regions of $G$ yields a single connected curve; and then creasing into a cylinder as discussed yields a tight cylinder $\M(\Gamma_X^*, \Gamma)$. Similarly, $\M(\Gamma_X^*, \Gamma')$ is tight.

Applying lemma \ref{lem:chord_diagrams_connected} then yields a sequence of chord diagrams $\Gamma = \Gamma_1, \ldots, \Gamma_m = \Gamma'$, where each $\Gamma_{i+1}$ is obtained from $\Gamma_i$ by a bypass surgery, and each $\M(\Gamma_X^*, \Gamma_i)$ is tight.

Rounding the cylinder back into the original compactified plane $S^2$, each $\Gamma_i \cup \Gamma_X$ yields a connected dividing set on the plane, consisting of $\Gamma_i$ drawn in $D_{r_0}$ and the same $\Gamma_X$ drawn on the other complementary regions. Thus each $\Gamma_i \cup \Gamma_X$ is a tight configuration.

Since each $\Gamma_{i+1} \cup \Gamma_X$ is obtained from $\Gamma_i \cup \Gamma_X$ by a bypass surgery on $D_{r_0}$, and is tight, these bypass surgeries correspond to transitions where an actual bypass is removed from one of the two balls of $M'$ along an inner attaching arc on $D_{r_0}$, and a bypass is added to the other ball of $M'$ along the same (now outer) attaching arc. (See the discussion at the end of section \ref{sec:edge_rounding_bypasses}.) This provides the desired sequence of state transitions.
\end{proof}

Lemma \ref{lem:state_transitions_between_trees} and proposition \ref{prop:same_degrees_same_contact_structure} now follow straightforwardly.

\begin{proof}[Proof of lemma \ref{lem:state_transitions_between_trees}]
The tree-hugging dividing sets $\Gamma_T, \Gamma_{T'}$ have common restriction $\Gamma_X$ outside $D_{r_0}$; let their restrictions to $D_{r_0}$ be $\Gamma_{0}, \Gamma'_{0}$ respectively. Applying proposition 
\ref{prop:transitions_on_disc} to $\Gamma_0, \Gamma'_0$ and $\Gamma_X$ yields the desired sequence of state transitions from $\Gamma_0 \cup \Gamma_X = \Gamma_T$ to $\Gamma'_0 \cup \Gamma_X = \Gamma_{T'}$.
\end{proof}

\begin{proof}[Proof of proposition \ref{prop:same_degrees_same_contact_structure}]
Given spanning trees $T,T'$ of $G_V$ with the same hypertree $f$, lemma \ref{lem:spanning_trees_by_local_moves} provides a sequence of spanning trees $T=T_0, \ldots, T_n = T'$ of $G_V$, all with hypertree $f$, where each $T_{i+1}$ is obtained from $T_i$ by removing one edge and adding another. Lemma \ref{lem:state_transitions_between_trees} then provides a sequence of state transitions from each $\Gamma_{T_i}$ to $\Gamma_{T_{i+1}}$, which together yield the desired sequence of state transitions from $\Gamma_T$ to $\Gamma_{T'}$.
\end{proof}

\subsection{Concluding the proof}

We now put the pieces together and complete the proof of theorem \ref{thm:classification_of_tight_contact_structures}.

\begin{proof}[Proof of theorem \ref{thm:classification_of_tight_contact_structures}]
The gluing theorem \ref{thm:Honda_gluing}, applied to the decomposition of $(M_G, L_G)$ along $S = \sqcup_{r \in R} D_r$ into $M' = M^+ \cup M^-$ as discussed in section \ref{sec:applying_gluing_thm}, says that isotopy classes of tight contact structures on $(M_G, L_G)$ are in bijection with $\pi_0 (\GG_0 (M_G, L_G, S))$, the tight connected components of the configuration graph. By proposition \ref{prop:potentially_tight_tight}, potentially tight configurations are tight. By lemma \ref{lem:potentially_tight_conditions} then tight configurations are dividing sets which, when drawn in the complementary regions of $G$ and connected across the edges of $G$, yield a connected curve.

Each spanning tree $T$ of $G_V$ yields a tree-hugging configuration $\Gamma_T$ (section \ref{sec:hypertree_configurations}), which is tight (proposition \ref{prop:tree-hugging_tight}).  Conversely, every tight contact structure is tree-hugging (proposition \ref{prop:tight_implies_tree-hugging}). In other words, all tight configurations are connected via state transitions to tree-hugging configurations.

Each spanning tree of $G_V$ has a corresponding hypertree $f\colon R \To \Z_{\ge0}$ in $(E,R)$. Spanning trees with distinct hypertrees yield configurations which have distinct Euler class, hence are not related by state transitions (proposition \ref{prop:distinct_hypertrees_distinct_contact_structures}), but spanning trees with the same hypertree yield configurations which are related by state transitions (proposition \ref{prop:same_degrees_same_contact_structure}).

Thus, the connected components of $\GG_0 (M_G, L_G, S)$ are in bijective correspondence with hypertrees in $(E,R)$, and by \cite{Kalman13_Tutte}, discussed in section \ref{sec:hypergraphs_hypertrees}, the number of hypertrees is the same in any of the six hypergraphs induced from the trinity of $G$.
\end{proof}

Note that the proof is explicit: given a hypertree $f$, we take a corresponding spanning tree $T$, and the tree-hugging configuration $\Gamma_T$ gives the corresponding contact structure. Let us denote the isotopy class of contact structures corresponding to $f$ by $\xi_f$.

\section{Properties of the contact structures}
\label{sec:properties_of_contact_structures}

\subsection{Inclusion into the 3-sphere}

All along, we have implicitly considered $(M_G, L_G)$ as a submanifold of $S^3$. We now consider the inclusion $(M_G, L_G) \hookrightarrow S^3$ explicitly.

\begin{prop}
\label{prop:inclusion_into_S3}
For each hypertree $f\colon R\To\Z_{\ge0}$, the contact structure $\xi_f$ extends from $(M_G, L_G)$ to a tight contact structure on $S^3$.
\end{prop}

\begin{proof}
Take a spanning tree $T$ of $G_V$ representing $f$, and a $T$-hugging configuration $\Gamma_T = \sqcup_{r \in R} \Gamma_r$. Using the homeomorphism $\tilde{D}_r \cong D_r$ of section \ref{sec:graph_to_sutured_manifold}, draw each $\Gamma_r$ on $\tilde{D}_r$. As $\Gamma_r$ is tree-hugging, connecting up the dividing sets on each $\tilde{D}_r$ across the edges of $G$ yields the usual single connected dividing curve $\tilde{\Gamma}$ on $S^2$.

Take a contact structure $\xi_{st}$ on $S^3$ so that $S^2$ is convex with connected dividing set $\tilde{\Gamma}$, and the two balls $B^\pm$ on either side have their unique (up to isotopy) tight contact structure. Thus (see section \ref{sec:sutured_manifolds}) $\xi_{st}$ is the unique (up to isotopy) tight contact structure on $S^3$.

The graph $G \subset S^2$ is transverse to $\tilde{\Gamma}$ and is nonisolating: the components of $S^2 \setminus (\tilde{\Gamma} \cup G)$ are precisely the components into which each $\tilde D_r$ is cut by the dividing set $\Gamma_T$, and these all intersect $\tilde{\Gamma}$. Hence the Legendrian realisation principle (for graphs, as discussed in section \ref{sec:convex_surfaces}) 
applies, and after a small isotopy of $S^2$ we may assume that $G$ is Legendrian embedded in $S^2$. 

Consider a regular neighbourhood $N_G$ of $G$ in $S^3$. We can take $N_G$ so that $\partial N_G$ is convex, $N_G \cap S^2$ is a regular neighbourhood of $G$ in $S^2$, and $\partial N_G \cap S^2$ is a set of Legendrian curves bounding this neighbourhood. For instance, take a regular neighbourhood $N_G$, 
use the Legendrian realisation principle along $S^2$ to
perturb it so $\partial N_G \cap S^2$ is Legendrian, then perturb $\partial N_G$ holding $\partial N_G \cap S^2$ fixed to be convex.
For this last step, note that since every face of $G$ in $S^2$ has at least two sides, and every component of $\partial N_G \cap S^2$ gets the same framing from $S^2$ as from $N_G$, the framing condition of section \ref{sec:convex_surfaces} is satisfied.

Since the contact planes make a negative half-turn along each edge of $G$, the dividing set on $\partial N_G$ is isotopic to $L_G$.

Now consider removing $N_G$ from $S^3$. This yields $M_G$, and the tight contact structure $\xi_{st}$ on $S^3$ restricts to a contact structure $\xi_{st}|_{M_G}$ on $(M_G, L_G)$. As the restriction of a tight contact structure, $\xi_{st}|_{M_G}$ is also tight.
Moreover, the dividing set obtained on each $D_r$ is isotopic to the original $\Gamma_r$. So $\xi_{st}|_{M_G}$ is given by the tree-hugging configuration $\Gamma_T$, hence isotopic to $\xi_f$.
Thus $\xi_f$ is a restriction of the tight contact structure $\xi_{st}$ on $S^3$, which is another way of saying that $\xi_f$ extends to $\xi_{st}$.
\end{proof}

\subsection{Euler classes and spin-c structures} 
\label{sec:euler_classes_of_contact_structures}

We now consider the Euler classes $e(\xi_f) \in H^2(M_G, \partial M_G)$ of our tight contact structures. With the conventions discussed in section \ref{sec:hypertrees_Euler_class}, the following is immediate.

\begin{prop}
\label{prop:Euler_class_xif}
For each $f\in B_{(E,R)}$, the Euler class $e(\xi_f)$ is given on the generators $[D_r]$ by
\[
e(\xi_f)[D_r] = 2f(r) - n_r + 1. 
\]
\end{prop}

\begin{proof}
Indeed, by proposition \ref{prop:same_degrees_same_contact_structure} and the definition of $\xi_f$, each spanning tree $T$ representing $f$ is such that $\xi_f$ is isotopic to (or rather, the class $\xi_f$ contains) a contact structure with a $T$-hugging dividing set. Then we may apply lemma \ref{lem:Euler_class_from_hypertree}.
\end{proof}

Now $H^2 (M_G, \partial M_G) \cong H_1(M_G) \cong \Z^{|R|-1}$, as $M_G$ is a handlebody of genus $|R|-1$. So we expect a relation between the $|R|$ evaluations $e(\xi_f)[D_r]$. In a similar vein, the $|R|$ values $f(r)$ of the hypertree $f\colon R \To \Z_{\ge0}$ also obey a relation. These are given in the next proposition.

\begin{prop}
\label{prop:affine_planes}
For any hypertree $f$ of $(E,R)$,
\begin{align}
\label{eqn:sum_of_hypertree}
\sum_{r \in R} f(r) &= |E| - 1 \\
\label{eqn:sum_of_Euler_class}
\sum_{r \in R} e(\xi_f)[D_r] &= |E| - |V|.
\end{align}
\end{prop}

\begin{proof}
For equation (\ref{eqn:sum_of_hypertree}), we recall the following short proof from \cite{Kalman13_Tutte}. A spanning tree $T$ of $G_V = \Bip(E,R)$ representing $f$ has degree $f(r)+1$ at each $r \in R$. Each edge is incident with precisely one red vertex, so summing $f(r)+1$ over all $r \in R$ gives the total number of edges in $T$. It has $|E|+|R|$ vertices and hence $|E|+|R|-1$ edges. We obtain
\(
\sum_{r \in R} (f(r) + 1) = |E| + |R| - 1,
\)
which yields (\ref{eqn:sum_of_hypertree}). 

To prove equation (\ref{eqn:sum_of_Euler_class}), we sum $e(\xi_f)[D_r] = 2f(r) - n_r + 1$ over all $r \in R$. Consider the triangulation that is the trinity of $G$. At each red vertex there are $2n_r$ triangles, $n_r$ white and $n_r$ black. Thus $\sum_{r \in R} n_r = n$, where $n$ is the number of white or black triangles. As discussed in section \ref{sec:arborescence_number}, $|V| + |E| + |R| - n - 2 = 0$. 
Hence our sum is $2(|E|-1)-n+|R|=|E|-|V|$.
\end{proof}

Let $(x_r)_{r \in R}$ denote coordinates on $\Z^R$. Define the rank $(|R|-1)$ affine hyperplane $\mathfrak{H}$ in $\Z^R$ by
\[
\sum_{r \in R} x_r = |E|-1.
\]
Each hypertree $f\colon R \To \Z_{\ge0}$ in $(E,R)$ can be regarded as a point of $\Z^R$, and equation (\ref{eqn:sum_of_hypertree}) says that in fact $f$ lies on $\mathfrak{H}$. (Similarly, there is an embedding $H^2(M_G, \partial M_G) \hookrightarrow \Z^R$ given by $h \mapsto (h[D_r])_{r \in R}$ and under this embedding $H^2(M_G, \partial M_G)$ lies on the affine hyperplane with equation $\sum_{r \in R} x_r = |E|-|V|$.)

The proof of theorem \ref{cor:hypertree_Euler_bijection} is now straightforward.

\begin{proof}[Proof of theorem \ref{cor:hypertree_Euler_bijection}]
Propositions \ref{prop:Euler_class_xif} and \ref{prop:affine_planes}
imply that the bijection which sends a hypertree $f$ to the Euler class $e(\xi_f)$ extends to an affine isomorphism $\mathfrak{H} \To H^2 (M_G, \partial M_G)$.
\end{proof}

As discussed in section \ref{sec:hypergraphs_hypertrees}, the points $B_{(E,R)}$of $\Z^R$ corresponding to hypertrees of $(E,R)$ generate a lattice polytope $\QQ_{(E,R)}$. We have just shown that in fact $B_{(E,R)}\subset\QQ_{(E,R)} \subset \mathfrak{H}$.

Turning to spin-c structures, as discussed in section \ref{sec:spin-c_structures}, $\Spin^c(M_G, L_G)$ is an affine space over $H^2 (M_G, \partial M_G) \cong H_1 (M_G) \cong \Z^{|R|-1}$.
We can now elaborate on our explanation in section \ref{sec:SFH_and_hypergraphs} of the main result of the first author, with Juh\'{a}sz and Rasmussen, in \cite{Juhasz-Kalman-Rasmussen12}. That result says that
\[
\Supp(M_G, L_G) \cong B_{(E,R)},
\]
where $\Supp(M_G, L_G) \subset \Spin^c(M_G, L_G)$ is the support of $SFH(M_G, L_G)$.
The equivalence here refers to an identification of the rank $(|R|-1)$ affine space $\Spin^c(M_G, L_G)$ with the affine hyperplane $\mathfrak{H} \subset \Z^R$.

Thus we see that both (Euler classes of) tight contact structures and the support of $SFH$ are combinatorially equivalent to hypertrees. In the next section we make this equivalence more geometric.

\subsection{Contact invariants}
\label{sec:contact_invariants}

We now consider the contact invariants $c(\xi_f)$ of each of the tight contact structures $\xi_f$ on $(M_G, L_G)$.
We combine the TQFT property of $SFH$ (reviewed in section \ref{sec:SFHbackground}) with proposition \ref{prop:inclusion_into_S3}.
Letting $\xi_G$ be a contact structure on a neighbourhood of $G$ such that $\xi_{st} = \xi_f \cup \xi_G$, we obtain a homomorphism of graded abelian groups
\[
\Phi_{\xi_G} \colon SFH(-M_G, -L_G) \To \widehat{HF}(-S^3).
\]
Since $\widehat{HF}(-S^3) \cong \Z$ and $c(\xi_{st}) = \pm 1$, we have $\Phi_{\xi_G} c(\xi_f) = \pm 1$. Hence $c(\xi_f) = \{ \pm x\}$ for some nonzero primitive element $x$ of $SFH(-M_G, -L_G)$.

Now recall from \cite{FJR11}, as discussed in section \ref{sec:sutured_L-spaces}, that $(M_G, L_G)$ is a sutured $L$-space, so for each $\s \in \Supp(M_G, L_G)$, we have $SFH(M_G, L_G, \s) \cong \Z$. Similarly, for each $\s \in \Supp(-M_G, -L_G)$ we have $SFH(-M_G, -L_G, \s) \cong \Z$. If $\xi_f$ has spin-c structure $\s$, then by proposition \ref{prop:contact_invariant_spin-c_summand}, $c(\xi_f) \subset SFH(-M_G, -L_G, \s)$. 
As $c(\xi_f)\ne\{0\}$, we then have $\s\in\Supp(-M_G, -L_G)$. Since
the only nonzero primitive elements in $\Z$ are $\pm 1$, we immediately obtain the following.

\begin{lem}
\label{lem:contact_invariants_are_ones}
If $\xi$ is a tight contact structure on $(M_G, L_G)$ with spin-c structure $\s$, then $c(\xi) = \{ \pm 1 \} \subset \Z \cong SFH(-M_G, -L_G, \s)$.
\qed
\end{lem}

We have seen that if $f,f'\in B_{(E,R)}$ are distinct hypertrees, then $\xi_f, \xi_{f'}$ have distinct Euler classes, by proposition \ref{prop:Euler_class_xif}. We now consider their spin-c structures.

\begin{lem}
\label{lem:distinct_spin-c_structures}
If $f,f'$ are distinct hypertrees, then the spin-c structures $\s_f, \s_{f'}$ of $\xi_f, \xi_{f'}$ are distinct.
\end{lem}

\begin{proof}
As $\Spin^c (-M_G, -L_G)$ is affine over $H^2 (-M_G, \partial(-M_G))=H^2(M_G,\partial M_G)$, we have $\s_{f'} - \s_{f} \in H^2(M_G, \partial M_G)$. Since $e(\xi_f) \neq e(\xi_{f'})$, by (\ref{eqn:spin-c_and_euler}) applied to $u = \s_f$ and $h = PD(\s_{f'} - \s_f)$ we have
\[
0 \neq e(\xi_{f'}) - e(\xi_f) = e(\s_{f'}) - e(\s_f) = 2(\s_{f'} - \s_f).
\]
As $H^2(M_G, \partial M_G) \cong \Z^{|R|-1}$ is torsion-free, $\s_{f'} \neq \s_f$ follows.
\end{proof}

\begin{proof}[Proof of theorem \ref{thm:classification_of_contact_invariants}]
Two distinct (non-isotopic) contact structures $\xi_f, \xi_{f'}$ on $(M_G, L_G)$ have distinct hypertrees $f,f'$, hence by lemma \ref{lem:distinct_spin-c_structures} they have distinct spin-c structures.

The number of (isotopy classes of) tight contact structures on $(M_G, L_G)$ is given by $|B_{(E,R)}|$. All these contact structures $\xi_f$, over $f \in B_{(E,R)}$, have distinct spin-c structures, and their contact invariants $c(\xi_f)$ are all nonzero and primitive
in the sutured Floer hoology groups that correspond to those spin-c structures, cf.\ proposition \ref{prop:contact_invariant_spin-c_summand}.
Hence the $c(\xi_f)$ lie in $|B_{(E,R)}|$ distinct spin-c summands of $SFH(-M_G,-L_G)$. But since $|\Supp(-M_G, -L_G)| = |\Supp(M_G, L_G)| = |B_{(E,R)}|$, the spin-c structures of the $\xi_f$ must be precisely $\Supp(-M_G, -L_G)$. That is, a tight contact structure $\xi$ with spin-c structure $\s$ exists if and only if $SFH(-M_G, -L_G, \s)$ is nontrivial. In this case, $SFH(-M_G, -L_G, \s) \cong \Z$ and lemma \ref{lem:contact_invariants_are_ones} says $c(\xi) = \{\pm x\}$, where $x$ is a generator.
\end{proof}

\section{Contact structures and knot theory}
\label{sec:alexander}

In this final section we consider applications of our results to knot theory, starting with the Alexander polynomial. Let $G$ be a connected plane bipartite graph, so that the median construction gives a special alternating link $L_G$. Recall that the planar dual graph $G^*$ is naturally oriented. We denote the Alexander polynomial of a link $L$ by $\Delta_L(t)$. The following observation is due to Murasugi and Stoimenow \cite{Murasugi-Stoimenow03}.

\begin{prop}
\label{prop:alexander_magic}
The leading coefficient of $\Delta_{L_G}$ is equal to the arborescence number $\rho(G^*)$ of the planar dual $G^*$. (That is, by definition, to the magic number of the trinity of $G$.)
\qed
\end{prop}

The proof is based on Kauffman's state summation formula for $\Delta$. The states may be described as a dual pair of spanning trees, one in $G$ and one in $G^*$. The contribution of the former turns out to be neutral (the same for all states), whereas the latter tree maximizes the exponent of $t$ if and only if it is an arborescence. As always for alternating diagrams, terms do not cancel.

Now we are in a position to prove one of the last two claims made in the Introduction.

\begin{proof}[Proof of corollary \ref{cor:Alexander_magic}]
We have seen that the arborescence number is equal for all three dual graphs of a trinity, given by the magic number. Proposition \ref{prop:alexander_magic} says that the leading coefficient of each Alexander polynomial is also the magic number, and theorem \ref{thm:classification_of_tight_contact_structures} says this is also the number of isotopy classes tight contact structures on the corresponding sutured manifolds.
\end{proof}

Finally we prove theorem \ref{thm:Kauffman_states}, by applying the above results to a universe in Kauffman's formal knot theory, as discussed in section \ref{sec:FKT}.

Let $\U$ be a (connected) universe. Our first observation is that there is naturally a trinity associated to $\U$. Place red vertices $R$ at the vertices of $\U$. Consider the planar dual graph $G$ of $\U$. Like any knot diagram, the complementary regions of $\U$ have a checkerboard colouring, so $G$ is naturally a bipartite planar graph. Let its colour classes be violet $V$ and emerald $E$. See figure \ref{fig:figure_8_red_graph}.

Each complementary region of $G$ corresponds to a vertex of $\U$; as each vertex of $\U$ has degree $4$, each region of $G$ has four sides. Joining each red vertex to the two violet and two emerald vertices on the boundary of its region yields the trinity, of which $G$ is the red graph. 

Next, consider the sutured manifold $(M_G, L_G)$, using the notation of section \ref{sec:contact_structures}. A potentially tight configuration is given by a dividing set $\Gamma_S = \sqcup_{r \in R} \Gamma_r$, where $S=\sqcup_{r\in R}D_r$ and $\Gamma_r$ is a chord diagram on the disc $D_r$, with endpoints midway between violet and emerald vertices on the boundary of $D_r$. By lemma \ref{lem:potentially_tight_conditions}, a collection of chord diagrams on each $D_r$ is potentially tight if and only if, drawing them in the complementary regions $\tilde{D}_r$ of $G$ and connecting them across each edge of $G$ results in a single connected curve. By proposition \ref{prop:potentially_tight_tight}, potentially tight and tight are equivalent.

Now in the graph $G$ arising from a universe $\U$, each complementary region has precisely four sides. So each chord diagram $\Gamma_r$ has precisely two chords. Hence, on each disc $D_r$, there are two possible chord diagrams (up to isotopy).
We may also observe that the two possible chord diagrams on $D_r$ are precisely given by the two splittings of $\U$ at the vertex $r$.

\begin{prop}
\label{prop:state_configuration_bijection}
Splitting vertices of $\U$ according to markers provides a bijection between states of $\U$ and tight configurations of $(M_G, L_G,S)$.
\end{prop}

\begin{proof}
The state-trail correspondence discussed in section \ref{sec:FKT} says that splitting vertices according to markers provides a bijection from states to trails on $\U$, i.e., choices of splittings at each vertex of $\U$ so as to obtain a single loop.

A tight configuration, on the other hand, consists of a choice of chord diagram $\Gamma_r$ in each disc $\tilde{D}_r$ so that the $\Gamma_r$ connect across each edge of $G$ to give a single connected curve. This corresponds precisely to a choice of splitting at each vertex of $\U$ so as to obtain a single loop, i.e., a trail. 
\end{proof}

Figure \ref{fig:figure_8_trail_configuration} shows a trail on $\U$, which can be regarded as a tight configuration of $(M_G, L_G)$.

\begin{figure}
\begin{center}
\input{fig-8_trail_configuration.tikz.tex}
\end{center}
\caption{The trail on the figure-8 universe of figure \ref{fig:figure_8_knot_universe} can also be regarded as a configuration.}
\label{fig:figure_8_trail_configuration}
\end{figure}

Now we observe that, when each $\Gamma_r$ consists of only two chords, there are no nontrivial bypass surgeries on $\Gamma_r$. Hence there are no state transitions between tight configurations. So $\GG_0(M_G, L_G, S)$ has vertices given by the tight configurations, and no edges.

Moreover, the two possible dividing sets on each $D_r$ have Euler classes $1$ and $-1$. So each tight configuration on $(M_G, L_G, S)$ has a distinct Euler class. We can now prove theorem \ref{thm:Kauffman_states}.

\begin{proof}[Proof of theorem \ref{thm:Kauffman_states}]
Let $\U$ be a universe, $G$ the bipartite planar graph obtained as the planar dual of $\U$, and $(M_G, L_G)$ the corresponding sutured 3-manifold. By proposition \ref{prop:state_configuration_bijection}, tight configurations on $(M_G, L_G)$ are in bijective correspondence with states of $\U$. As $\GG_0(M_G, L_G, S)$ has no edges, each tight configuration yields a distinct isotopy class of tight contact structures, giving a bijection between states of $\U$ and isotopy classes of tight contact structures on $(M_G, L_G)$.
\end{proof}

\addcontentsline{toc}{section}{References}

\small

\bibliography{superbib}
\bibliographystyle{amsplain}

\end{document}

%% file: red_graph.tikz.tex
\begin{tikzpicture}[scale = 0.15]
\coordinate (r3) at (1,6);
\coordinate (e1) at (-1,4); 
\coordinate (v1) at (-1.5, 2);
\coordinate (e2) at (-2,0); 
\coordinate (v2) at (-2.3,-2); 
\coordinate (e3) at (-2.5,-4); 
\coordinate (r0) at (-3,-7); 
\coordinate (v4) at (3,4);
\coordinate (r2) at (2.5,0);
\coordinate (v3) at (4,-3);
\coordinate (e0) at (7,3);
\coordinate (v0) at (-6,6); 
\coordinate (r1) at (-6, 0); 

\draw [ultra thick, draw=none] (r3) to[out=210,in=70] (e1); 
\draw [ultra thick, red] (e1) to [out = 250, in = 70] (v1); 
\draw [ultra thick, red] (v1) to (e2); 
\draw [ultra thick, red] (e2) to (v2); 
\draw [ultra thick, red] (v2) to (e3); 
\draw [ultra thick, draw=none] (e3) to [out=270, in=90] (r0); 
\draw [ultra thick, draw=none] (r0)
.. controls ($ (r0) + (270:8) $) and ($ (e0) + (345:10) $) .. (e0); 
\draw [ultra thick, red] (e0) to [out=165, in=0] (v4); 
\draw [ultra thick, red] (v4) to [out=180, in=10] (e1); 
\draw [ultra thick, draw=none] (e1) to [out=190, in=90] (r1); 
\draw [ultra thick, draw=none] (r1) to [out=270, in=170] (v2); 
\draw [ultra thick, draw=none] (v2) to [out=350, in=210] (r2); 
\draw [ultra thick, draw=none] (r2) to [out=30, in=300] (v4); 
\draw [ultra thick, draw=none] (v4) to [out=120, in=300] (r3); 
\draw [ultra thick, draw=none] (r3) to [out=120, in=30] (v0); 
\draw [ultra thick, red] (v0)
.. controls ($ (v0) + (210:8) $) and ($ (e3) + (200:8) $) .. (e3); 
\draw [ultra thick, draw=none] (e3) to [out=20, in=270] (r2); 
\draw [ultra thick, draw=none] (r2) to [out=90, in=330] (e1); 
\draw [ultra thick, red] (e1) to [out=150, in=330] (v0); 
\draw [ultra thick, draw=none] (v0)
.. controls ($ (v0) + (150:6) $) and ($ (r0) + (180:15) $) .. (r0); 
\draw [ultra thick, draw=none] (r0) to [out=0, in=270] (v3); 
\draw [ultra thick, draw=none] (v3) to [out=90, in=300] (r2); 
\draw [ultra thick, draw=none] (r2) to [out=120, in=0] (v1); 
\draw [ultra thick, draw=none] (v1) to [out=180, in=60] (r1); 
\draw [ultra thick, draw=none] (r1)
.. controls ($ (r1) + (240:5) $) and ($ (e3) + (120:2) $) .. (e3); 
\draw [ultra thick, red] (e3) to [out=300, in=210] (v3); 
\draw [ultra thick, red] (v3) to [out=30, in=300] (e0); 
\draw [ultra thick, draw=none] (e0) to [out=120, in=30] (r3); 

\draw [ultra thick, red] (e0) 
.. controls ($ (e0) + (60:7) $) and ($ (v0) + (60:7) $) .. (v0); 
\draw [ultra thick, draw=none] (v0)
.. controls ($ (v0) + (240:3) $) and ($ (r1) + (150:3) $) .. (r1); 
\draw [ultra thick, draw=none] (r1) to [out=330, in=180] (e2); 
\draw [ultra thick, draw=none] (e2) to [out=0, in=180] (r2); 
\draw [ultra thick, draw=none] (r2) to [out=0, in=240] (e0); 


\foreach \x/\word in {(e0)/e0, (e1)/e1, (e2)/e2, (e3)/e3}
{
\draw [green!50!black, fill=green!50!black] \x circle  (10pt);
}

\foreach \x/\word in {(v0)/v0, (v1)/v1, (v2)/v2, (v3)/v3, (v4)/v4}
{
\draw [blue, fill=blue] \x circle (10pt);
}

\end{tikzpicture}

%% file: median_construction.tikz.tex
\begin{tikzpicture}[scale = 0.4]
\coordinate (r3) at (1,6);
\coordinate (e1) at (-1,4); 
\coordinate (v1) at (-1.5, 2);
\coordinate (e2) at (-2,0); 
\coordinate (v2) at (-2.3,-2); 
\coordinate (e3) at (-2.5,-4); 
\coordinate (r0) at (-3,-7); 
\coordinate (v4) at (3,4);
\coordinate (r2) at (2.5,0);
\coordinate (v3) at (4,-3);
\coordinate (e0) at (7,3);
\coordinate (v0) at (-6,6); 
\coordinate (r1) at (-6, 0); 

\coordinate (v0a) at ($ (v0) + (15:1) $);
\coordinate (v0b) at ($ (v0) + (135:1) $);
\coordinate (v0c) at ($ (v0) + (270:1) $);

\coordinate (v1a) at ($ (v1) + (340:1) $);
\coordinate (v1b) at ($ (v1) + (160:1) $);

\coordinate (v2a) at ($ (v2) + (350:1) $);
\coordinate (v2b) at ($ (v2) + (170:1) $);

\coordinate (v3a) at ($ (v3) + (120:1) $);
\coordinate (v3b) at ($ (v3) + (300:1) $);

\coordinate (v4a) at ($ (v4) + (90:1) $);
\coordinate (v4b) at ($ (v4) + (270:1) $);

\coordinate (e0a) at ($ (e0) + (0:1) $);
\coordinate (e0b) at ($ (e0) + (112:1) $);
\coordinate (e0c) at ($ (e0) + (232:1) $);

\coordinate (e1a) at ($ (e1) + (80:1) $);
\coordinate (e1b) at ($ (e1) + (200:1) $);
\coordinate (e1c) at ($ (e1) + (310:1) $);

\coordinate (e2a) at ($ (e2) + (345:1) $);
\coordinate (e2b) at ($ (e2) + (165:1) $);

\coordinate (e3a) at ($ (e3) + (12:1) $);
\coordinate (e3b) at ($ (e3) + (142:1) $);
\coordinate (e3c) at ($ (e3) + (250:1) $);

\draw [draw=none, fill=black!10!white]
(e1b) 
to [out = 250, in = 70] (v1a)
to (v1b)
to [out=70, in=250] (e1c)
to (e1b);

\draw [draw=none, fill=black!10!white]
(v1a)
to [out=250, in=75] (e2b)
to (e2a)
to [out=75, in=250] (v1b)
to (v1a);

\draw [draw=none, fill=black!10!white]
(e2b)
to [out=255, in=80] (v2a)
to (v2b)
to [out=80, in=255] (e2a)
to (e2b);

\draw [draw=none, fill=black!10!white]
(v2a)
to [out=260, in=85] (e3b) 
to (e3a)
to [out=85, in=260] (v2b) 
to (v2a);

\draw [draw=none, fill=black!10!white]
(e3b)
.. controls ($ (e3b) + (200:8) $) and ($ (v0b) + (210:8) $) .. (v0b)
to (v0c)
.. controls ($ (v0c) + (210:8) $) and ($ (e3c) + (200:6) $) .. (e3c)
to (e3b);

\draw [draw=none, fill=black!10!white]
(v0b)
.. controls ($ (v0b) + (60:7) $) and ($ (e0b) + (60:5) $) .. (e0b) 
to (e0a)
.. controls ($ (e0a) + (60:7) $) and ($ (v0a) + (60:6) $) .. (v0a)
to (v0b);

\draw [draw=none, fill=black!10!white]
(e0b)
to [out=165, in=0] (v4b)
to (v4a)
to [out=0, in=165] (e0c)
to (e0c);

\draw [draw=none, fill=black!10!white]
(v4b)
to [out=180, in=10] (e1a)
to (e1c)
to [out=10, in=180] (v4a)
to (v4b);

\draw [draw=none, fill=black!10!white]
(e1a)
to [out=150, in=330] (v0c)
to (v0a)
to [out=330, in=150] (e1b)
to (e1a);

\draw [draw=none, fill=black!10!white]
(e3c)
to [out=330, in=220] (v3a)
to (v3b)
to [out=220, in=320] (e3a)
to (e3c);

\draw [draw=none, fill=black!10!white]
(v3a)
to [out=40, in=300] (e0a)
to (e0c)
to [out=300, in=20] (v3b)
to (v3a);

\draw [draw=none, fill=black!10!white]
(v0a) to (v0b) to (v0c) to (v0a);

\draw [draw=none, fill=black!10!white]
(e0a) to (e0b) to (e0c) to (e0a);

\draw [draw=none, fill=black!10!white]
(e1a) to (e1b) to (e1c) to (e1a);

\draw [draw=none, fill=black!10!white]
(e3a) to (e3b) to (e3c) to (e3a);

\draw [ultra thick, red] (e1) to [out = 250, in = 70] (v1); 
\draw [ultra thick, red] (v1) to (e2); 
\draw [ultra thick, red] (e2) to (v2); 
\draw [ultra thick, red] (v2) to (e3); 
\draw [ultra thick, red] (e0) to [out=165, in=0] (v4); 
\draw [ultra thick, red] (v4) to [out=180, in=10] (e1); 
\draw [ultra thick, red] (v0)
.. controls ($ (v0) + (210:8) $) and ($ (e3) + (200:8) $) .. (e3); 
\draw [ultra thick, red] (e1) to [out=150, in=330] (v0); 
\draw [ultra thick, red] (e3) to [out=300, in=210] (v3); 
\draw [ultra thick, red] (v3) to [out=30, in=300] (e0); 
\draw [ultra thick, red] (e0) 
.. controls ($ (e0) + (60:7) $) and ($ (v0) + (60:7) $) .. (v0); 

\begin{knot}[consider self intersections, 
clip width=2,
flip crossing=2,
flip crossing=4,
flip crossing=6,
flip crossing=8,
flip crossing=11
]
\strand [thick]
(e1b)
to [out = 250, in = 70] (v1a)
to [out=250, in=75] (e2b) 
to [out=255, in=80] (v2a) 
to [out=260, in=85] (e3b) 
.. controls ($ (e3b) + (200:8) $) and ($ (v0b) + (210:8) $) .. (v0b)
.. controls ($ (v0b) + (60:7) $) and ($ (e0b) + (60:5) $) .. (e0b) 
to [out=165, in=0] (v4b)
to [out=180, in=10] (e1a)
to [out=150, in=330] (v0c)
.. controls ($ (v0c) + (210:8) $) and ($ (e3c) + (200:6) $) .. (e3c) 
to [out=330, in=220] (v3a) 
to [out=40, in=300] (e0a) 
.. controls ($ (e0a) + (60:7) $) and ($ (v0a) + (60:6) $) .. (v0a) 
to [out=330, in=150] (e1b);
\strand[thick]
(e1c)
to [out=10, in=180] (v4a)
to [out=0, in=165] (e0c)
to [out=300, in=20] (v3b) 
to [out=220, in=320] (e3a) 
to [out=85, in=260] (v2b) 
to [out=80, in=255] (e2a) 
to [out=75, in=250] (v1b) 
to [out=70, in=250] (e1c); 
\end{knot}


\foreach \x/\word in {(e0)/e0, (e1)/e1, (e2)/e2, (e3)/e3}
{
\draw [green!50!black, fill=green!50!black] \x circle  (10pt);
}

\foreach \x/\word in {(v0)/v0, (v1)/v1, (v2)/v2, (v3)/v3, (v4)/v4}
{
\draw [blue, fill=blue] \x circle (10pt);
}


\end{tikzpicture}

%% file: fig-8_universe_small.tikz.tex
\begin{tikzpicture}[scale = 0.15,
knot/.style={ultra thick},
marker/.style={fill=black},
violetedge/.style={ultra thick, blue},
emeraldedge/.style={ultra thick, green!50!black},
rededge/.style={ultra thick, red}
]
\coordinate (b) at (0,-9);
\coordinate (m) at (0,-3);
\coordinate (l) at (-3,0);
\coordinate (r) at (3,0);
\coordinate (vl) at (-5 , -5);
\coordinate (vr) at (5,-5);
\coordinate (vt) at (0, 2);
\coordinate (eb) at (0,-6);
\coordinate (em) at (0,-1);
\coordinate (et) at (0, 6);


\draw [knot] (b) 
.. controls ($ (b) + (-30:8) $) and ($ (r) + (-15:8) $) .. (r)
to [out=165, in=15] (l)
.. controls ($ (l) + (195:8) $) and ($ (b) + (210:8) $) .. (b)
.. controls ($ (b) + (30:3) $) and ($ (m) + (-30:3) $) .. (m)
to [out=150, in=-60] (l)
.. controls ($ (l) + (120:7) $) and ($ (r) + (60:7) $) .. (r)
to [out=240, in=30] (m)
.. controls ($ (m) + (210:3) $) and ($ (b) + (150:3) $) .. (b);

\draw (4,-8) node {\Huge $*$};
\draw (6,-11) node {\Huge $*$};

\end{tikzpicture}

%% file: fig-8_red_graph.tikz.tex
\begin{tikzpicture}[scale = 0.4,
knot/.style={ultra thick},
marker/.style={fill=black},
violetedge/.style={ultra thick, blue},
emeraldedge/.style={ultra thick, green!50!black},
rededge/.style={ultra thick, red}
]
\coordinate (b) at (0,-9);
\coordinate (m) at (0,-3);
\coordinate (l) at (-3,0);
\coordinate (r) at (3,0);
\coordinate (vl) at (-5 , -5);
\coordinate (vr) at (5,-5);
\coordinate (vt) at (0, 2);
\coordinate (eb) at (0,-6);
\coordinate (em) at (0,-1);
\coordinate (et) at (0, 6);


\draw [knot] (b) 
.. controls ($ (b) + (-30:8) $) and ($ (r) + (-15:8) $) .. (r)
to [out=165, in=15] (l)
.. controls ($ (l) + (195:8) $) and ($ (b) + (210:8) $) .. (b)
.. controls ($ (b) + (30:3) $) and ($ (m) + (-30:3) $) .. (m)
to [out=150, in=-60] (l)
.. controls ($ (l) + (120:7) $) and ($ (r) + (60:7) $) .. (r)
to [out=240, in=30] (m)
.. controls ($ (m) + (210:3) $) and ($ (b) + (150:3) $) .. (b);

\draw (4,-9) node {\Huge $*$};
\draw (6,-11) node {\Huge $*$};

\draw [rededge] (vl) -- (eb);
\draw [rededge] (vr) -- (eb);
\draw [rededge] (vl) -- (em);
\draw [rededge] (vr) -- (em);
\draw [rededge] (vt) -- (em);
\draw [rededge] (vt) -- (et);
\draw [rededge] (vl) .. controls ($ (vl) + (180:8) $) and ($ (et) + (180:8) $) .. (et);
\draw [rededge] (vr) .. controls ($ (vr) + (0:8) $) and ($ (et) + (0:8) $) .. (et);

\foreach \x/\word in {(vl)/vl, (vr)/vr, (vt)/vt}
{
\draw [blue, fill=blue] \x circle (10pt);
}

\foreach \x/\word in {(eb)/eb, (em)/em, (et)/et}
{
\draw [green!50!black, fill=green!50!black] \x circle  (10pt);
}

\foreach \x/\word in {(b)/b, (m)/m, (l)/l, (r)/r}
{
\draw [red, fill=red] \x circle  (10pt);
}

\end{tikzpicture}

%% file: bypasstriple.tikz.tex
\begin{tikzpicture}[
scale=1, 
suture/.style={thick, draw=red},
boundary/.style={ultra thick},
vertex/.style={draw=red, fill=red}]
\coordinate (0a) at ($ (90:1) + (-150:3) $);
\coordinate (1a) at ($ (30:1) + (-150:3) $);
\coordinate (2a) at ($ (-30:1) + (-150:3) $);
\coordinate (3a) at ($ (-90:1) + (-150:3) $);
\coordinate (4a) at ($ (-150:1) + (-150:3) $);
\coordinate (5a) at ($ (150:1) + (-150:3) $);
\filldraw[fill=black!10!white, draw=none] (4a) arc (210:150:1) to [bend right=90] (0a) arc (90:30:1) -- cycle;
\filldraw[fill=black!10!white, draw=none] (2a) to [bend right=90] (3a) arc (-90:-30:1);
\draw [boundary] (-150:3) circle (1 cm);
\draw [suture] (5a) to [bend right=90] (0a);
\draw [suture] (4a) -- (1a);
\draw [suture] (2a) to [bend right=90] (3a);
\coordinate (0) at ($ (90:1) + (90:3) $);
\coordinate (1) at ($ (30:1) + (90:3) $);
\coordinate (2) at ($ (-30:1)  + (90:3) $);
\coordinate (3) at ($ (-90:1) + (90:3) $);
\coordinate (4) at ($ (-150:1) + (90:3) $);
\coordinate (5) at ($ (150:1) + (90:3) $);
\filldraw[fill=black!10!white, draw=none] (0) arc (90:30:1) to [bend right=90] (2) arc (-30:-90:1) -- cycle;
\filldraw[fill=black!10!white, draw=none] (4) to [bend right=90] (5) arc (150:210:1);
\draw [boundary] (90:3) circle (1 cm);
\draw [suture] (1) to [bend right=90] (2);
\draw [suture] (0) -- (3);
\draw [suture] (4) to [bend right=90] (5);
\coordinate (0b) at ($ (90:1) + (-30:3) $);
\coordinate (1b) at ($ (30:1) + (-30:3) $);
\coordinate (2b) at ($ (-30:1) + (-30:3) $);
\coordinate (3b) at ($ (-90:1) + (-30:3) $);
\coordinate (4b) at ($ (-150:1) + (-30:3) $);
\coordinate (5b) at ($ (150:1) + (-30:3) $);
\filldraw[fill=black!10!white, draw=none] (2b) arc (-30:-90:1) to [bend right=90] (4b) arc (210:150:1) -- cycle;
\filldraw[fill=black!10!white, draw=none] (0b) to [bend right=90] (1b) arc (30:90:1);
\draw [boundary] (-30:3) circle (1 cm);
\draw [suture] (3b) to [bend right=90] (4b);
\draw [suture] (2b) -- (5b);
\draw [suture] (0b) to [bend right=90] (1b);
\draw [->] (115:1.5) -- (-165:1.5);
\draw [->] (-135:1.5) -- (-45:1.5);
\draw [->] (-15:1.5) -- (75:1.5);
\end{tikzpicture}

%% file: fig-8_universe.tikz.tex
\begin{tikzpicture}[scale = 0.2,
knot/.style={ultra thick},
marker/.style={fill=black},
violetedge/.style={ultra thick, blue},
emeraldedge/.style={ultra thick, green!50!black},
rededge/.style={ultra thick, red}
]
\coordinate (b) at (0,-9);
\coordinate (m) at (0,-3);
\coordinate (l) at (-3,0);
\coordinate (r) at (3,0);
\coordinate (vl) at (-5 , -5);
\coordinate (vr) at (5,-5);
\coordinate (vt) at (0, 2);
\coordinate (eb) at (0,-6);
\coordinate (em) at (0,-1);
\coordinate (et) at (0, 6);


\draw [knot] (b) 
.. controls ($ (b) + (-30:8) $) and ($ (r) + (-15:8) $) .. (r)
to [out=165, in=15] (l)
.. controls ($ (l) + (195:8) $) and ($ (b) + (210:8) $) .. (b)
.. controls ($ (b) + (30:3) $) and ($ (m) + (-30:3) $) .. (m)
to [out=150, in=-60] (l)
.. controls ($ (l) + (120:7) $) and ($ (r) + (60:7) $) .. (r)
to [out=240, in=30] (m)
.. controls ($ (m) + (210:3) $) and ($ (b) + (150:3) $) .. (b);

\draw (4,-8) node {\Huge $*$};
\draw (6,-11) node {\Huge $*$};

\end{tikzpicture}

%% file: fig-8_state.tikz.tex
\begin{tikzpicture}[scale = 0.2,
knot/.style={ultra thick},
marker/.style={fill=black},
violetedge/.style={ultra thick, blue},
emeraldedge/.style={ultra thick, green!50!black},
rededge/.style={ultra thick, red}
]
\coordinate (b) at (0,-9);
\coordinate (m) at (0,-3);
\coordinate (l) at (-3,0);
\coordinate (r) at (3,0);
\coordinate (vl) at (-5 , -5);
\coordinate (vr) at (5,-5);
\coordinate (vt) at (0, 2);
\coordinate (eb) at (0,-6);
\coordinate (em) at (0,-1);
\coordinate (et) at (0, 6);


\draw [knot] (b) 
.. controls ($ (b) + (-30:8) $) and ($ (r) + (-15:8) $) .. (r)
to [out=165, in=15] (l)
.. controls ($ (l) + (195:8) $) and ($ (b) + (210:8) $) .. (b)
.. controls ($ (b) + (30:3) $) and ($ (m) + (-30:3) $) .. (m)
to [out=150, in=-60] (l)
.. controls ($ (l) + (120:7) $) and ($ (r) + (60:7) $) .. (r)
to [out=240, in=30] (m)
.. controls ($ (m) + (210:3) $) and ($ (b) + (150:3) $) .. (b);

\draw (4,-8) node {\Huge $*$};
\draw (6,-11) node {\Huge $*$};

\draw [marker] (b) -- ($ (b) + (35:1) $) arc (35:145:1) -- cycle;
\draw [marker] (m) -- ($ (m) + (32:1) $) arc (32:148:1) -- cycle;
\draw [marker] (l) -- ($ (l) + (200:1) $) arc (200:300:1) -- cycle;
\draw [marker] (r) -- ($ (r) + (65:1) $) arc (60:165:1) -- cycle;

\end{tikzpicture}

%% file: fig-8_trail.tikz.tex
\begin{tikzpicture}[scale = 0.2,
knot/.style={ultra thick},
marker/.style={fill=black},
violetedge/.style={ultra thick, blue},
emeraldedge/.style={ultra thick, green!50!black},
rededge/.style={ultra thick, red}
]
\coordinate (b) at (0,-9);
\coordinate (m) at (0,-3);
\coordinate (l) at (-3,0);
\coordinate (r) at (3,0);
\coordinate (vl) at (-5 , -5);
\coordinate (vr) at (5,-5);
\coordinate (vt) at (0, 2);
\coordinate (eb) at (0,-6);
\coordinate (em) at (0,-1);
\coordinate (et) at (0, 6);


\draw [knot, rounded corners=0.5 cm] (0,5) 
.. controls ($ (0,5) + (180:2) $) and ($ (l) + (120:4) $)  .. (l)
.. controls ($ (l) + (195:8) $) and ($ (b) + (210:8) $) .. (b)
.. controls ($ (b) + (150:3) $) and ($ (m) + (210:3) $) .. (m)
to [out=150, in=-60] (l)
to [out=15, in=165] (r)
to [out=240, in=30] (m)
.. controls ($ (m) + (-30:3) $) and ($ (b) + (30:3) $) .. (b)
.. controls ($ (b) + (-30:8) $) and ($ (r) + (-15:8) $) .. (r)
.. controls ($ (r) + (60:4) $) and ($ (0,5) + (0:2) $) .. (0,5);
;

\draw (4,-9) node {\Huge $*$};
\draw (6,-11) node {\Huge $*$};

\end{tikzpicture}

%% file: vertex_splitting.tikz.tex
\begin{tikzpicture}[
scale=0.5]

\draw [ultra thick] (-2,-2) -- (2,2);
\draw [ultra thick] (-2,2) -- (2,-2);

\draw [->] (3,0) -- (5,0);
\draw [->] (-3,0) -- (-5,0);

\draw [xshift = 8 cm, ultra thick] (-2,-2) to [out=45, in=135] (2,-2);
\draw [xshift = 8 cm, ultra thick] (-2,2) to [out=-45, in=225] (2,2);

\draw [xshift = -8 cm, ultra thick] (-2,-2) to [out=45, in=-45] (-2,2);
\draw [xshift = -8 cm, ultra thick] (2,-2) to [out=135, in= 225] (2,2);

\end{tikzpicture}

%% file: state_marker_splitting.tikz.tex
\begin{tikzpicture}[
scale=0.5]

\draw [ultra thick] (-2,-2) -- (2,2);
\draw [ultra thick] (-2,2) -- (2,-2);
\draw [fill=black] (0,0) -- (45:0.5) arc (45:135:0.5) -- cycle;

\draw [->] (3,0) -- (5,0);

\draw [xshift = 8 cm, ultra thick] (-2,-2) to [out=45, in=-45] (-2,2);
\draw [xshift = 8 cm, ultra thick] (2,-2) to [out=135, in= 225] (2,2);

\end{tikzpicture}

%% file: transposition.tikz.tex
\begin{tikzpicture}[
scale=0.7]

\foreach \x in {-7, 7}
{
\draw [xshift = \x cm, dotted, fill=black!10!white] (-1,1) -- (-1,-1) -- (1,-1) -- (1,1) -- cycle;
\draw [xshift= \x cm, ultra thick] (-4,0) -- (-0.5,0);
\draw [xshift= \x cm, ultra thick] (4,0) -- (0.5,0);
\draw [xshift = \x cm, ultra thick] (-3,-2) -- (-3,2);
\draw [xshift = \x cm, ultra thick] (3,-2) -- (3,2);
\draw [xshift = \x cm] (-3.5,0.3) node {$v$};
\draw [xshift = \x cm] (3.5,0.3) node {$w$};
\draw [xshift = \x cm] (-2,1.5) node {$R_1$};
\draw [xshift =\x cm] (-2,-1.5) node {$R_2$};
}

\draw [->] (-1,0.5) -- node [above] {Clockwise} (1,0.5);
\draw [->] (1,-0.5) -- node [below] {Counterclockwise} (-1,-0.5);

\draw [xshift = -7 cm, fill=black] (-2.5,0) arc (0:90:0.5) -- (-3,0) -- cycle;
\draw [xshift = -7 cm, fill=black] (2.5,0) arc (180:270:0.5) -- (3,0) -- cycle;
\draw [xshift = -7 cm] (0,-3) node {State $s_1$};

\draw [xshift = 7 cm, fill=black] (-2.5,0) arc (0:-90:0.5) -- (-3,0) -- cycle;
\draw [xshift = 7 cm, fill=black] (2.5,0) arc (180:90:0.5) -- (3,0) -- cycle;
\draw [xshift = 7 cm] (0,-3) node {State $s_2$};

\end{tikzpicture}

%% file: trinity.tikz.tex
\begin{tikzpicture}[scale = 0.25]

\clip (-13,-10) rectangle (11,11);

\coordinate (r3) at (1,6);
\coordinate (e1) at (-1,4); 
\coordinate (v1) at (-1.5, 2);
\coordinate (e2) at (-2,0); 
\coordinate (v2) at (-2.3,-2); 
\coordinate (e3) at (-2.5,-4); 
\coordinate (r0) at (-3,-7); 
\coordinate (v4) at (3,4);
\coordinate (r2) at (2.5,0);
\coordinate (v3) at (4,-3);
\coordinate (e0) at (7,3);
\coordinate (v0) at (-6,6); 
\coordinate (r1) at (-6, 0); 

\fill [black!10!white] (e0)
.. controls ($ (e0) + (60:7) $) and ($ (v0) + (60:7) $) .. (v0)
to [out=30, in=120] (r3)
to [out=30, in=120] (e0);

\fill [black!10!white] (r3)
to[out=210,in=70] (e1)
to [out=10, in=180] (v4)
to [out=120, in=300] (r3);

\fill [black!10!white] (v0)
.. controls ($ (v0) + (240:3) $) and ($ (r1) + (150:3) $) .. (r1)
to [out=90, in=190] (e1)
to [out=150, in=330] (v0);

\fill [black!10!white] (e1)
to [out = 250, in = 70] (v1)
to [out=0, in=120] (r2)
to [out=90, in=330] (e1);

\fill [black!10!white] (v1)
to (e2)
to [out=180, in=330] (r1)
to [out=60, in=180] (v1);

\fill [black!10!white] (e2)
to (v2)
to [out=350, in=210] (r2)
to [out=180, in=0] (e2);

\fill [black!10!white] (v2)
to (e3)
.. controls ($ (e3) + (120:2) $)  and ($ (r1) + (240:5) $) .. (r1)
to [out=270, in=170] (v2);

\fill [black!10!white] (v0) 
.. controls ($ (v0) + (210:8) $) and ($ (e3) + (200:8) $) .. (e3)
to [out=270, in=90] (r0)
.. controls ($ (r0) + (180:15) $) and ($ (v0) + (150:6) $) .. (v0);

\fill [black!10!white] (v4)
to [out=0, in=165] (e0)
to [out=240, in=0] (r2)
to [out=30, in=300] (v4);

\fill [black!10!white] (v3) 
to [out=90, in=300] (r2)
to [out=270, in=20] (e3)
to [out=300, in=210] (v3);

\fill [black!10!white] (v3) 
to [out=30, in=300] (e0)
.. controls ($ (e0) + (345:10) $)  and ($ (r0) + (270:8) $) .. (r0)
to [out=0, in=270] (v3);

\draw [ultra thick, red] (e1) to [out = 250, in = 70] (v1); 
\draw [ultra thick, red] (v1) to (e2); 
\draw [ultra thick, red] (e2) to (v2); 
\draw [ultra thick, red] (v2) to (e3); 
\draw [ultra thick, red] (e0) to [out=165, in=0] (v4); 
\draw [ultra thick, red] (v4) to [out=180, in=10] (e1); 
\draw [ultra thick, red] (v0)
.. controls ($ (v0) + (210:8) $) and ($ (e3) + (200:8) $) .. (e3); 
\draw [ultra thick, red] (e1) to [out=150, in=330] (v0); 
\draw [ultra thick, red] (e3) to [out=300, in=210] (v3); 
\draw [ultra thick, red] (v3) to [out=30, in=300] (e0); 
\draw [ultra thick, red] (e0) 
.. controls ($ (e0) + (60:7) $) and ($ (v0) + (60:7) $) .. (v0); 

\draw [ultra thick, blue] (e1) to [out=190, in=90] (r1); 
\draw [ultra thick, green!50!black] (r1) to [out=270, in=170] (v2); 
\draw [ultra thick, green!50!black] (v1) to [out=180, in=60] (r1); 
\draw [ultra thick, blue] (r1)
.. controls ($ (r1) + (240:5) $) and ($ (e3) + (120:2) $) .. (e3); 
\draw [ultra thick, green!50!black] (v0)
.. controls ($ (v0) + (240:3) $) and ($ (r1) + (150:3) $) .. (r1); 
\draw [ultra thick, blue] (r1) to [out=330, in=180] (e2); 

\draw [ultra thick, green!50!black] (v2) to [out=350, in=210] (r2); 
\draw [ultra thick, green!50!black] (r2) to [out=30, in=300] (v4); 
\draw [ultra thick, blue] (e3) to [out=20, in=270] (r2); 
\draw [ultra thick, blue] (r2) to [out=90, in=330] (e1); 
\draw [ultra thick, green!50!black] (v3) to [out=90, in=300] (r2); 
\draw [ultra thick, green!50!black] (r2) to [out=120, in=0] (v1); 
\draw [ultra thick, blue] (e2) to [out=0, in=180] (r2); 
\draw [ultra thick, blue] (r2) to [out=0, in=240] (e0); 

\draw [ultra thick, blue] (r3) to[out=210,in=70] (e1); 
\draw [ultra thick, green!50!black] (v4) to [out=120, in=300] (r3); 
\draw [ultra thick, green!50!black] (r3) to [out=120, in=30] (v0); 
\draw [ultra thick, blue] (e0) to [out=120, in=30] (r3); 

\draw [ultra thick, draw=none] (e3) to [out=270, in=90] (r0); 
\draw [ultra thick, draw=none] (r0)
.. controls ($ (r0) + (270:8) $) and ($ (e0) + (345:10) $) .. (e0); 
\draw [ultra thick, draw=none] (v0)
.. controls ($ (v0) + (150:6) $) and ($ (r0) + (180:15) $) .. (r0); 
\draw [ultra thick, draw=none] (r0) to [out=0, in=270] (v3); 

\draw [ultra thick, blue] (e3) to [out=270, in=90] (r0); 
\draw [ultra thick, blue] (r0)
.. controls ($ (r0) + (270:8) $) and ($ (e0) + (345:10) $) .. (e0); 
\draw [ultra thick, green!50!black] (v0)
.. controls ($ (v0) + (150:6) $) and ($ (r0) + (180:15) $) .. (r0); 
\draw [ultra thick, green!50!black] (r0) to [out=0, in=270] (v3); 

\foreach \x/\word in {(r0)/r0, (r1)/r1, (r2)/r2, (r3)/r3}
{
\draw [red, fill=red] \x circle  (10pt);
}

\foreach \x/\word in {(e0)/e0, (e1)/e1, (e2)/e2, (e3)/e3}
{
\draw [green!50!black, fill=green!50!black] \x circle  (10pt);
}

\foreach \x/\word in {(v0)/v0, (v1)/v1, (v2)/v2, (v3)/v3, (v4)/v4}
{
\draw [blue, fill=blue] \x circle (10pt);
}

\end{tikzpicture}

%% file: trinity_with_dual.tikz.tex
\begin{tikzpicture}[scale = 0.25]

\clip (-13,-10) rectangle (11,12);

\coordinate (r3) at (1,6);
\coordinate (e1) at (-1,4); 
\coordinate (v1) at (-1.5, 2);
\coordinate (e2) at (-2,0); 
\coordinate (v2) at (-2.3,-2); 
\coordinate (e3) at (-2.5,-4); 
\coordinate (r0) at (-3,-7); 
\coordinate (v4) at (3,4);
\coordinate (r2) at (2.5,0);
\coordinate (v3) at (4,-3);
\coordinate (e0) at (7,3);
\coordinate (v0) at (-6,6); 
\coordinate (r1) at (-6, 0); 

\fill [black!10!white] (e0)
.. controls ($ (e0) + (60:7) $) and ($ (v0) + (60:7) $) .. (v0)
to [out=30, in=120] (r3)
to [out=30, in=120] (e0);

\fill [black!10!white] (r3)
to[out=210,in=70] (e1)
to [out=10, in=180] (v4)
to [out=120, in=300] (r3);

\fill [black!10!white] (v0)
.. controls ($ (v0) + (240:3) $) and ($ (r1) + (150:3) $) .. (r1)
to [out=90, in=190] (e1)
to [out=150, in=330] (v0);

\fill [black!10!white] (e1)
to [out = 250, in = 70] (v1)
to [out=0, in=120] (r2)
to [out=90, in=330] (e1);

\fill [black!10!white] (v1)
to (e2)
to [out=180, in=330] (r1)
to [out=60, in=180] (v1);

\fill [black!10!white] (e2)
to (v2)
to [out=350, in=210] (r2)
to [out=180, in=0] (e2);

\fill [black!10!white] (v2)
to (e3)
.. controls ($ (e3) + (120:2) $)  and ($ (r1) + (240:5) $) .. (r1)
to [out=270, in=170] (v2);

\fill [black!10!white] (v0) 
.. controls ($ (v0) + (210:8) $) and ($ (e3) + (200:8) $) .. (e3)
to [out=270, in=90] (r0)
.. controls ($ (r0) + (180:15) $) and ($ (v0) + (150:6) $) .. (v0);

\fill [black!10!white] (v4)
to [out=0, in=165] (e0)
to [out=240, in=0] (r2)
to [out=30, in=300] (v4);

\fill [black!10!white] (v3) 
to [out=90, in=300] (r2)
to [out=270, in=20] (e3)
to [out=300, in=210] (v3);

\fill [black!10!white] (v3) 
to [out=30, in=300] (e0)
.. controls ($ (e0) + (345:10) $)  and ($ (r0) + (270:8) $) .. (r0)
to [out=0, in=270] (v3);

\begin{scope}[dotted, ultra thick, decoration={
    markings,
    mark=at position 0.5 with {\arrow{>}}}
    ] 
\draw [blue, postaction={decorate}] (v4) to [bend right=10] (v0);
\draw [blue, postaction={decorate}] (v0) to (v1);
\draw [blue, postaction={decorate}] (v1) to (v4);
\draw [blue, postaction={decorate}] (v0) .. controls ($ (v0) + (50:4) $) and ($ (v4) + (45:7) $) .. (v4);
\draw [blue, postaction={decorate}] (v4) to [bend left=30] (v3);
\draw [blue, postaction={decorate}] (v3) to (v2);
\draw [blue, postaction={decorate}] (v2) .. controls ($ (v2) + (225:8) $) and ($ (v0) + (225:6) $) ..  (v0);
\draw [blue, postaction={decorate}] (v0) .. controls ($ (v0) + (205:15) $) and ($ (v3) + (235:11) $) .. (v3);
\draw [blue, postaction={decorate}] (v2) to [bend right=80] (v1);
\draw [blue, postaction={decorate}] (v1) to [bend right=80] (v2);
\draw [blue, postaction={decorate}] (v3) .. controls ($ (v3) + (20:15) $) and ($ (v0) + (70:15) $) .. (v0);
\end{scope}

\draw [ultra thick, blue] (r3) to[out=210,in=70] (e1); 
\draw [ultra thick, red] (e1) to [out = 250, in = 70] (v1); 
\draw [ultra thick, red] (v1) to (e2); 
\draw [ultra thick, red] (e2) to (v2); 
\draw [ultra thick, red] (v2) to (e3); 
\draw [ultra thick, blue] (e3) to [out=270, in=90] (r0); 
\draw [ultra thick, blue] (r0)
.. controls ($ (r0) + (270:8) $) and ($ (e0) + (345:10) $) .. (e0); 
\draw [ultra thick, red] (e0) to [out=165, in=0] (v4); 
\draw [ultra thick, red] (v4) to [out=180, in=10] (e1); 
\draw [ultra thick, blue] (e1) to [out=190, in=90] (r1); 
\draw [ultra thick, green!50!black] (r1) to [out=270, in=170] (v2); 
\draw [ultra thick, green!50!black] (v2) to [out=350, in=210] (r2); 
\draw [ultra thick, green!50!black] (r2) to [out=30, in=300] (v4); 
\draw [ultra thick, green!50!black] (v4) to [out=120, in=300] (r3); 
\draw [ultra thick, green!50!black] (r3) to [out=120, in=30] (v0); 
\draw [ultra thick, red] (v0)
.. controls ($ (v0) + (210:8) $) and ($ (e3) + (200:8) $) .. (e3); 
\draw [ultra thick, blue] (e3) to [out=20, in=270] (r2); 
\draw [ultra thick, blue] (r2) to [out=90, in=330] (e1); 
\draw [ultra thick, red] (e1) to [out=150, in=330] (v0); 
\draw [ultra thick, green!50!black] (v0)
.. controls ($ (v0) + (150:6) $) and ($ (r0) + (180:15) $) .. (r0); 
\draw [ultra thick, green!50!black] (r0) to [out=0, in=270] (v3); 
\draw [ultra thick, green!50!black] (v3) to [out=90, in=300] (r2); 
\draw [ultra thick, green!50!black] (r2) to [out=120, in=0] (v1); 
\draw [ultra thick, green!50!black] (v1) to [out=180, in=60] (r1); 
\draw [ultra thick, blue] (r1)
.. controls ($ (r1) + (240:5) $) and ($ (e3) + (120:2) $) .. (e3); 
\draw [ultra thick, red] (e3) to [out=300, in=210] (v3); 
\draw [ultra thick, red] (v3) to [out=30, in=300] (e0); 
\draw [ultra thick, blue] (e0) to [out=120, in=30] (r3); 

\draw [ultra thick, red] (e0) 
.. controls ($ (e0) + (60:7) $) and ($ (v0) + (60:7) $) .. (v0); 
\draw [ultra thick, green!50!black] (v0)
.. controls ($ (v0) + (240:3) $) and ($ (r1) + (150:3) $) .. (r1); 
\draw [ultra thick, blue] (r1) to [out=330, in=180] (e2); 
\draw [ultra thick, blue] (e2) to [out=0, in=180] (r2); 
\draw [ultra thick, blue] (r2) to [out=0, in=240] (e0); 

\foreach \x/\word in {(r0)/r0, (r1)/r1, (r2)/r2, (r3)/r3}
{
\draw [red, fill=red] \x circle  (10pt);
}

\foreach \x/\word in {(e0)/e0, (e1)/e1, (e2)/e2, (e3)/e3}
{
\draw [green!50!black, fill=green!50!black] \x circle  (10pt);
}

\foreach \x/\word in {(v0)/v0, (v1)/v1, (v2)/v2, (v3)/v3, (v4)/v4}
{
\draw [blue, fill=blue] \x circle (10pt);
}
\end{tikzpicture}

%% file: red_graph_with_disc.tikz.tex
\begin{tikzpicture}[scale = 0.2]
\coordinate (r3) at (1,6);
\coordinate (e1) at (-1,4); 
\coordinate (v1) at (-1.5, 2);
\coordinate (e2) at (-2,0); 
\coordinate (v2) at (-2.3,-2); 
\coordinate (e3) at (-2.5,-4); 
\coordinate (r0) at (-3,-7); 
\coordinate (v4) at (3,4);
\coordinate (r2) at (2.5,0);
\coordinate (v3) at (4,-3);
\coordinate (e0) at (7,3);
\coordinate (v0) at (-6,6); 
\coordinate (r1) at (-6, 0); 

\fill [black!10!white] 
(v0)
to [out=330, in=150] (e1)
to [out=250, in=70] (v1)
to (e2)
to (v2)
to (e3)
.. controls ($ (e3)+(200:8) $) and ($ (v0)+(210:8) $) .. (v0);

\draw [ultra thick, draw=none] (r3) to[out=210,in=70] (e1); 
\draw [ultra thick, red] (e1) to [out = 250, in = 70] (v1); 
\draw [ultra thick, red] (v1) to (e2); 
\draw [ultra thick, red] (e2) to (v2); 
\draw [ultra thick, red] (v2) to (e3); 
\draw [ultra thick, draw=none] (e3) to [out=270, in=90] (r0); 
\draw [ultra thick, draw=none] (r0)
.. controls ($ (r0) + (270:8) $) and ($ (e0) + (345:10) $) .. (e0); 
\draw [ultra thick, red] (e0) to [out=165, in=0] (v4); 
\draw [ultra thick, red] (v4) to [out=180, in=10] (e1); 
\draw [ultra thick, draw=none] (e1) to [out=190, in=90] (r1); 
\draw [ultra thick, draw=none] (r1) to [out=270, in=170] (v2); 
\draw [ultra thick, draw=none] (v2) to [out=350, in=210] (r2); 
\draw [ultra thick, draw=none] (r2) to [out=30, in=300] (v4); 
\draw [ultra thick, draw=none] (v4) to [out=120, in=300] (r3); 
\draw [ultra thick, draw=none] (r3) to [out=120, in=30] (v0); 
\draw [ultra thick, red] (v0)
.. controls ($ (v0) + (210:8) $) and ($ (e3) + (200:8) $) .. (e3); 
\draw [ultra thick, draw=none] (e3) to [out=20, in=270] (r2); 
\draw [ultra thick, draw=none] (r2) to [out=90, in=330] (e1); 
\draw [ultra thick, red] (e1) to [out=150, in=330] (v0); 
\draw [ultra thick, draw=none] (v0)
.. controls ($ (v0) + (150:6) $) and ($ (r0) + (180:15) $) .. (r0); 
\draw [ultra thick, draw=none] (r0) to [out=0, in=270] (v3); 
\draw [ultra thick, draw=none] (v3) to [out=90, in=300] (r2); 
\draw [ultra thick, draw=none] (r2) to [out=120, in=0] (v1); 
\draw [ultra thick, draw=none] (v1) to [out=180, in=60] (r1); 
\draw [ultra thick, draw=none] (r1)
.. controls ($ (r1) + (240:5) $) and ($ (e3) + (120:2) $) .. (e3); 
\draw [ultra thick, red] (e3) to [out=300, in=210] (v3); 
\draw [ultra thick, red] (v3) to [out=30, in=300] (e0); 
\draw [ultra thick, draw=none] (e0) to [out=120, in=30] (r3); 

\draw [ultra thick, red] (e0) 
.. controls ($ (e0) + (60:7) $) and ($ (v0) + (60:7) $) .. (v0); 
\draw [ultra thick, draw=none] (v0)
.. controls ($ (v0) + (240:3) $) and ($ (r1) + (150:3) $) .. (r1); 
\draw [ultra thick, draw=none] (r1) to [out=330, in=180] (e2); 
\draw [ultra thick, draw=none] (e2) to [out=0, in=180] (r2); 
\draw [ultra thick, draw=none] (r2) to [out=0, in=240] (e0); 


\draw (r1) node {$\tilde{D}_r$};

\foreach \x/\word in {(e0)/e0, (e1)/e1, (e2)/e2, (e3)/e3}
{
\draw [green!50!black, fill=green!50!black] \x circle  (10pt);
}

\foreach \x/\word in {(v0)/v0, (v1)/v1, (v2)/v2, (v3)/v3, (v4)/v4}
{
\draw [blue, fill=blue] \x circle (10pt);
}

\end{tikzpicture}

%% file: M_G.tikz.tex
\begin{tikzpicture}[scale = 0.4]

\clip (-12,-7) rectangle (11,11); 

\coordinate (r3) at (1,6);
\coordinate (e1) at (-1,4); 
\coordinate (v1) at (-1.5, 2);
\coordinate (e2) at (-2,0); 
\coordinate (v2) at (-2.3,-2); 
\coordinate (e3) at (-2.5,-4); 
\coordinate (r0) at (-3,-7); 
\coordinate (v4) at (3,4);
\coordinate (r2) at (2.5,0);
\coordinate (v3) at (4,-3);
\coordinate (e0) at (7,3);
\coordinate (v0) at (-6,6); 
\coordinate (r1) at (-6, 0); 

\coordinate (v0a) at ($ (v0) + (15:1) $);
\coordinate (v0b) at ($ (v0) + (135:1) $);
\coordinate (v0c) at ($ (v0) + (270:1) $);

\coordinate (v1a) at ($ (v1) + (340:1) $);
\coordinate (v1b) at ($ (v1) + (160:1) $);

\coordinate (v2a) at ($ (v2) + (350:1) $);
\coordinate (v2b) at ($ (v2) + (170:1) $);

\coordinate (v3a) at ($ (v3) + (120:1) $);
\coordinate (v3b) at ($ (v3) + (300:1) $);

\coordinate (v4a) at ($ (v4) + (90:1) $);
\coordinate (v4b) at ($ (v4) + (270:1) $);

\coordinate (e0a) at ($ (e0) + (0:1) $);
\coordinate (e0b) at ($ (e0) + (112:1) $);
\coordinate (e0c) at ($ (e0) + (232:1) $);

\coordinate (e1a) at ($ (e1) + (80:1) $);
\coordinate (e1b) at ($ (e1) + (200:1) $);
\coordinate (e1c) at ($ (e1) + (310:1) $);

\coordinate (e2a) at ($ (e2) + (345:1) $);
\coordinate (e2b) at ($ (e2) + (165:1) $);

\coordinate (e3a) at ($ (e3) + (12:1) $);
\coordinate (e3b) at ($ (e3) + (142:1) $);
\coordinate (e3c) at ($ (e3) + (250:1) $);

\fill [black!10!white] (v0c) 
.. controls ($ (v0c) + (210:6) $) and ($ (e3b) + (200:6) $) .. (e3b)
to [bend left=12] (e1b)
to (v0c);

\begin{knot}[consider self intersections, 
clip width=5,
flip crossing=2,
flip crossing=4,
flip crossing=6,
flip crossing=8,
flip crossing=11
]
\strand [ultra thick]
(e1b)
to [out = 250, in = 70] (v1a)
to [out=250, in=75] (e2b) 
to [out=255, in=80] (v2a) 
to [out=260, in=85] (e3b) 
.. controls ($ (e3b) + (200:8) $) and ($ (v0b) + (210:8) $) .. (v0b)
.. controls ($ (v0b) + (60:7) $) and ($ (e0b) + (60:5) $) .. (e0b) 
to [out=165, in=0] (v4b)
to [out=180, in=10] (e1a)
to [out=150, in=330] (v0c)
.. controls ($ (v0c) + (210:8) $) and ($ (e3c) + (200:6) $) .. (e3c) 
to [out=330, in=220] (v3a) 
to [out=40, in=300] (e0a) 
.. controls ($ (e0a) + (60:7) $) and ($ (v0a) + (60:6) $) .. (v0a) 
to [out=330, in=150] (e1b);
\strand[ultra thick]
(e1c)
to [out=10, in=180] (v4a)
to [out=0, in=165] (e0c)
to [out=300, in=20] (v3b) 
to [out=220, in=320] (e3a) 
to [out=85, in=260] (v2b) 
to [out=80, in=255] (e2a) 
to [out=75, in=250] (v1b) 
to [out=70, in=250] (e1c); 
\end{knot}




\fill [blue, opacity=0.2] (v0a)
to [out=135, in=30] (v0b)
to [out=240, in=180] (v0c)
to [out=30, in=270] (v0a);

\fill [brown, opacity=0.2] (v0a)
to [out=270, in=30] (v0c)
to (e1b)
to [out=90, in=210] (e1a)
to (v0a);
\draw [brown] (v0a) to [out=270, in=30] (v0c);
\draw [brown] (e1b) to [out=90, in=210] (e1a);

\fill [brown, opacity=0.2] (v0c)
to [out=180, in=240] (v0b)
.. controls ($ (v0b) + (210:10) $) and ($ (e3c) + (200:9) $) .. (e3c)
to [out=150, in=240] (e3b)
.. controls ($ (e3b) + (200:6) $) and ($ (v0c) + (210:6) $) .. (v0c);
\draw [brown] (v0c) to [out=180, in=240] (v0b);
\draw [brown] (e3c) to [out=150, in=240] (e3b);


\draw (r1) node {$D_r$};

\foreach \x/\word in {(e0)/e0, (e1)/e1, (e2)/e2, (e3)/e3}
{
\draw [green!50!black, fill=green!50!black] \x circle  (10pt);
}

\foreach \x/\word in {(v0)/v0, (v1)/v1, (v2)/v2, (v3)/v3, (v4)/v4}
{
\draw [blue, fill=blue] \x circle (10pt);
}

\foreach \x/\word in {(v0a)/a, (v0b)/b, (v0c)/c, (v1a)/a, (v1b)/b, (v2a)/a, (v2b)/b, (v3a)/a, (v3b)/b, (v4a)/a, (v4b)/b, (e0a)/a, (e0b)/b, (e0c)/c, (e1a)/a, (e1b)/b, (e1c)/c, (e2a)/a, (e2b)/b, (e3a)/a, (e3b)/b, (e3c)/c}
{
}


\draw [ultra thick, brown] (v0c) 
.. controls ($ (v0c) + (210:6) $) and ($ (e3b) + (200:6) $) .. (e3b)
to [bend left=12] (e1b)
to (v0c);

\draw [ultra thick, brown] (e1c)
to [bend left=2] (e3a)
to [bend right=5] (v3a)
to [bend right=10] (e0c)
to (v4b)
to (e1c);

\draw [ultra thick, brown] (v0a)
to (e1a)
to (v4a)
to (e0b)
.. controls ($ (e0b) + (60:5) $) and ($ (v0a) + (60:5) $) .. (v0a);

\draw [ultra thick, brown] (v0b)
.. controls ($ (v0b) + (210:10) $) and ($ (e3c) + (200:9) $) .. (e3c)
to [out=300, in=220] (v3b) 
to [out=20, in=300] (e0a)
.. controls ($ (e0a) + (60:8) $) and ($ (v0b) + (60:8) $) .. (v0b);

\end{tikzpicture}

%% file: configuration.tikz.tex
\begin{tikzpicture}[scale = 0.4]

\clip (-12,-8) rectangle (11,11);

\coordinate (r3) at (1,6);
\coordinate (e1) at (-1,4); 
\coordinate (v1) at (-1.5, 2);
\coordinate (e2) at (-2,0); 
\coordinate (v2) at (-2.3,-2); 
\coordinate (e3) at (-2.5,-4); 
\coordinate (r0) at (-3,-7); 
\coordinate (v4) at (3,4);
\coordinate (r2) at (2.5,0);
\coordinate (v3) at (4,-3);
\coordinate (e0) at (7,3);
\coordinate (v0) at (-6,6); 
\coordinate (r1) at (-6, 0); 

\coordinate (v0a) at ($ (v0) + (15:1) $);
\coordinate (v0b) at ($ (v0) + (135:1) $);
\coordinate (v0c) at ($ (v0) + (270:1) $);

\coordinate (v1a) at ($ (v1) + (340:1) $);
\coordinate (v1b) at ($ (v1) + (160:1) $);

\coordinate (v2a) at ($ (v2) + (350:1) $);
\coordinate (v2b) at ($ (v2) + (170:1) $);

\coordinate (v3a) at ($ (v3) + (120:1) $);
\coordinate (v3b) at ($ (v3) + (300:1) $);

\coordinate (v4a) at ($ (v4) + (90:1) $);
\coordinate (v4b) at ($ (v4) + (270:1) $);

\coordinate (e0a) at ($ (e0) + (0:1) $);
\coordinate (e0b) at ($ (e0) + (112:1) $);
\coordinate (e0c) at ($ (e0) + (232:1) $);

\coordinate (e1a) at ($ (e1) + (80:1) $);
\coordinate (e1b) at ($ (e1) + (200:1) $);
\coordinate (e1c) at ($ (e1) + (310:1) $);

\coordinate (e2a) at ($ (e2) + (345:1) $);
\coordinate (e2b) at ($ (e2) + (165:1) $);

\coordinate (e3a) at ($ (e3) + (12:1) $);
\coordinate (e3b) at ($ (e3) + (142:1) $);
\coordinate (e3c) at ($ (e3) + (250:1) $);

\draw [ultra thick] ($ (e1c)!0.5!(v4b) $) 
.. controls ($ (e1c)!0.5!(v4b) + (250:3) $) and ($ (v4b)!0.5!(e0c) + (250:3) $) .. ($ (v4b)!0.5!(e0c) $);
\draw [ultra thick] ($ (e0c)!0.5!(v3a) + (0.3,0) $) .. controls ($ (e0c)!0.5!(v3a) + (135:2) $) and ($ (v3a)!0.5!(e3a) + (135:2) $) .. ($ (v3a)!0.5!(e3a) $);
\draw [ultra thick] ($ (e3a)!0.4!(v2a) $) 
to [out=350, in=330] ($ (v1a)!0.5!(e1c) $);
\draw [ultra thick] ($ (v2a)!0.5!(e2a) $) 
to [out=330, in=330] ($ (e2a)!0.5!(v1a) $);

\draw [ultra thick] ($ (v1b)!0.5!(e2b) $)
to [out=150, in=60] ($ (e3b) + (-4.8,1) $);
\draw [ultra thick] ($ (e2b)!0.5!(v2b) $) 
.. controls ($ (e2b)!0.5!(v2b) + (150:2) $) and ($ (v2b)!0.5!(e3b) + (150:2) $) .. ($ (v2b)!0.5!(e3b) $);
\draw [ultra thick] ($ (v0c)!0.5!(e1b) $) to [out=225, in=225] ($ (e1b)!0.5!(v1b) $);

\draw [ultra thick] ($ (v4a)!0.5!(e0b) $) to [bend left=15] ($ (e0b) + (-3.5, 4.3) $);
\draw [ultra thick] ($ (v4a)!0.3!(e1a) $) to [bend right=75] ($ (e1a)!0.5!(v0a) $);

\draw [ultra thick] ($ (v3b) + (-3.8, -2.2) $) to [bend left=60] ($ (e3c) + (-6.8, 0.6) $);
\draw [ultra thick] ($ (e0a) + (-0.2, -4.1) $) .. controls ($ (e0a) + (-0.2, -4.1) + (0:6) $) and ($ (e0a) + (-4.3, 7) + (30:6) $) .. ($ (e0a) + (-4.3, 7) $);

\begin{knot}[consider self intersections, 
clip width=5,
flip crossing=2,
flip crossing=4,
flip crossing=6,
flip crossing=8,
flip crossing=11
]
\strand [ultra thick]
(e1b)
to [out = 250, in = 70] (v1a)
to [out=250, in=75] (e2b) 
to [out=255, in=80] (v2a) 
to [out=260, in=85] (e3b) 
.. controls ($ (e3b) + (200:8) $) and ($ (v0b) + (210:8) $) .. (v0b)
.. controls ($ (v0b) + (60:7) $) and ($ (e0b) + (60:5) $) .. (e0b) 
to [out=165, in=0] (v4b)
to [out=180, in=10] (e1a)
to [out=150, in=330] (v0c)
.. controls ($ (v0c) + (210:8) $) and ($ (e3c) + (200:6) $) .. (e3c) 
to [out=330, in=220] (v3a) 
to [out=40, in=300] (e0a) 
.. controls ($ (e0a) + (60:7) $) and ($ (v0a) + (60:6) $) .. (v0a) 
to [out=330, in=150] (e1b);
\strand[ultra thick]
(e1c)
to [out=10, in=180] (v4a)
to [out=0, in=165] (e0c)
to [out=300, in=20] (v3b) 
to [out=220, in=320] (e3a) 
to [out=85, in=260] (v2b) 
to [out=80, in=255] (e2a) 
to [out=75, in=250] (v1b) 
to [out=70, in=250] (e1c); 
\end{knot}

\draw [ultra thick, brown] (v0c) 
.. controls ($ (v0c) + (210:6) $) and ($ (e3b) + (200:6) $) .. (e3b)
to [bend left=12] (e1b) 
to (v0c);

\draw [ultra thick, brown] (e1c)
to (v1a) to (e2a) to (v2a) to (e3a)
to [bend right=5] (v3a)
to [bend right=10] (e0c)
to (v4b)
to (e1c);

\draw [ultra thick, brown] (v0a)
to (e1a)
to (v4a)
to (e0b)
.. controls ($ (e0b) + (60:5) $) and ($ (v0a) + (60:5) $) .. (v0a);

\draw [ultra thick, brown] (v0b)
.. controls ($ (v0b) + (210:10) $) and ($ (e3c) + (200:9) $) .. (e3c)
to [out=300, in=220] (v3b) 
to [out=20, in=300] (e0a)
.. controls ($ (e0a) + (60:8) $) and ($ (v0b) + (60:8) $) .. (v0b);


\foreach \x/\word in {(e0)/e0, (e1)/e1, (e2)/e2, (e3)/e3}
{
\draw [green!50!black, fill=green!50!black] \x circle  (10pt);
}

\foreach \x/\word in {(v0)/v0, (v1)/v1, (v2)/v2, (v3)/v3, (v4)/v4}
{
\draw [blue, fill=blue] \x circle (10pt);
}

\foreach \x/\word in {(v0a)/a, (v0b)/b, (v0c)/c, (v1a)/a, (v1b)/b, (v2a)/a, (v2b)/b, (v3a)/a, (v3b)/b, (v4a)/a, (v4b)/b, (e0a)/a, (e0b)/b, (e0c)/c, (e1a)/a, (e1b)/b, (e1c)/c, (e2a)/a, (e2b)/b, (e3a)/a, (e3b)/b, (e3c)/c}
{
\draw [black, fill=black] \x circle (3pt);
}


\end{tikzpicture}

%% file: configuration_rounding_1.tikz.tex
\begin{tikzpicture}[scale = 0.3]

\clip (-12,-8) rectangle (11,11);

\coordinate (r3) at (1,6);
\coordinate (e1) at (-1,4); 
\coordinate (v1) at (-1.5, 2);
\coordinate (e2) at (-2,0); 
\coordinate (v2) at (-2.3,-2); 
\coordinate (e3) at (-2.5,-4); 
\coordinate (r0) at (-3,-7); 
\coordinate (v4) at (3,4);
\coordinate (r2) at (2.5,0);
\coordinate (v3) at (4,-3);
\coordinate (e0) at (7,3);
\coordinate (v0) at (-6,6); 
\coordinate (r1) at (-6, 0); 

\coordinate (v0a) at ($ (v0) + (15:1) $);
\coordinate (v0b) at ($ (v0) + (135:1) $);
\coordinate (v0c) at ($ (v0) + (270:1) $);

\coordinate (v1a) at ($ (v1) + (340:1) $);
\coordinate (v1b) at ($ (v1) + (160:1) $);

\coordinate (v2a) at ($ (v2) + (350:1) $);
\coordinate (v2b) at ($ (v2) + (170:1) $);

\coordinate (v3a) at ($ (v3) + (120:1) $);
\coordinate (v3b) at ($ (v3) + (300:1) $);

\coordinate (v4a) at ($ (v4) + (90:1) $);
\coordinate (v4b) at ($ (v4) + (270:1) $);

\coordinate (e0a) at ($ (e0) + (0:1) $);
\coordinate (e0b) at ($ (e0) + (112:1) $);
\coordinate (e0c) at ($ (e0) + (232:1) $);

\coordinate (e1a) at ($ (e1) + (80:1) $);
\coordinate (e1b) at ($ (e1) + (200:1) $);
\coordinate (e1c) at ($ (e1) + (310:1) $);

\coordinate (e2a) at ($ (e2) + (345:1) $);
\coordinate (e2b) at ($ (e2) + (165:1) $);

\coordinate (e3a) at ($ (e3) + (12:1) $);
\coordinate (e3b) at ($ (e3) + (142:1) $);
\coordinate (e3c) at ($ (e3) + (250:1) $);

\draw [ultra thick] ($ (e1c)!0.5!(v4b) $) 
.. controls ($ (e1c)!0.5!(v4b) + (250:3) $) and ($ (v4b)!0.5!(e0c) + (250:3) $) .. ($ (v4b)!0.5!(e0c) $);
\draw [ultra thick] ($ (e0c)!0.5!(v3a) + (0.3,0) $) .. controls ($ (e0c)!0.5!(v3a) + (135:2) $) and ($ (v3a)!0.5!(e3a) + (135:2) $) .. ($ (v3a)!0.5!(e3a) $);
\draw [ultra thick] ($ (e3a)!0.4!(v2a) $) 
to [out=350, in=330] ($ (v1a)!0.5!(e1c) $);
\draw [ultra thick] ($ (v2a)!0.5!(e2a) $) 
to [out=330, in=330] ($ (e2a)!0.5!(v1a) $);

\draw [ultra thick] ($ (v1b)!0.5!(e2b) $)
to [out=150, in=60] ($ (e3b) + (-4.8,1) $);
\draw [ultra thick] ($ (e2b)!0.5!(v2b) $) 
.. controls ($ (e2b)!0.5!(v2b) + (150:2) $) and ($ (v2b)!0.5!(e3b) + (150:2) $) .. ($ (v2b)!0.5!(e3b) $);
\draw [ultra thick] ($ (v0c)!0.5!(e1b) $) to [out=225, in=225] ($ (e1b)!0.5!(v1b) $);

\draw [ultra thick] ($ (v4a)!0.5!(e0b) $) to [bend left=15] ($ (e0b) + (-3.5, 4.3) $);
\draw [ultra thick] ($ (v4a)!0.3!(e1a) $) to [bend right=75] ($ (e1a)!0.5!(v0a) $);

\draw [ultra thick] ($ (v3b) + (-3.8, -2.2) $) to [bend left=60] ($ (e3c) + (-6.8, 0.6) $);
\draw [ultra thick] ($ (e0a) + (-0.2, -4.1) $) .. controls ($ (e0a) + (-0.2, -4.1) + (0:6) $) and ($ (e0a) + (-4.3, 7) + (30:6) $) .. ($ (e0a) + (-4.3, 7) $);

\begin{knot}[consider self intersections, 
clip width=10,
flip crossing=2,
flip crossing=4,
flip crossing=6,
flip crossing=8,
flip crossing=11
]
\strand [ultra thick]
(e1b)
to [out = 250, in = 70] (v1a);
\strand[ultra thick] (e2b)
to [out=255, in=80] (v2a); 
\strand[ultra thick] (e3b) 
.. controls ($ (e3b) + (200:8) $) and ($ (v0b) + (210:8) $) .. (v0b);
\strand[ultra thick](e0b)
to [out=165, in=0] (v4b);
\strand[ultra thick] (e1a)
to [out=150, in=330] (v0c);
\strand[ultra thick] (e3c) 
to [out=310, in=220] (v3a); 
\strand [ultra thick] (e0a) 
.. controls ($ (e0a) + (60:7) $) and ($ (v0a) + (60:6) $) .. (v0a); 
\strand[ultra thick] (e1c)
to [out=10, in=180] (v4a);
\strand[ultra thick] (e0c)
to [out=300, in=20] (v3b); 
\strand [ultra thick](e3a) 
to [out=85, in=260] (v2b); 
\strand [ultra thick] (e2a) 
to [out=75, in=250] (v1b); 
\end{knot}

\draw [ultra thick, brown] (v0c) 
.. controls ($ (v0c) + (210:6) $) and ($ (e3b) + (200:6) $) .. (e3b)
to [bend left=12] (e1b) 
to (v0c);

\draw [ultra thick, brown] (e1c)
to (v1a) to (e2a) to (v2a) to (e3a)
to [bend right=5] (v3a)
to [bend right=10] (e0c)
to (v4b)
to (e1c);

\draw [ultra thick, brown] (v0a)
to (e1a)
to (v4a)
to (e0b)
.. controls ($ (e0b) + (60:5) $) and ($ (v0a) + (60:5) $) .. (v0a);

\draw [ultra thick, brown] (v0b)
.. controls ($ (v0b) + (210:10) $) and ($ (e3c) + (200:9) $) .. (e3c)
to [out=300, in=220] (v3b) 
to [out=20, in=300] (e0a)
.. controls ($ (e0a) + (60:8) $) and ($ (v0b) + (60:8) $) .. (v0b);


\foreach \x/\word in {(e0)/e0, (e1)/e1, (e2)/e2, (e3)/e3}
{
\draw [green!50!black, fill=green!50!black] \x circle  (10pt);
}

\foreach \x/\word in {(v0)/v0, (v1)/v1, (v2)/v2, (v3)/v3, (v4)/v4}
{
\draw [blue, fill=blue] \x circle (10pt);
}

\foreach \x/\word in {(v0a)/a, (v0b)/b, (v0c)/c, (v1a)/a, (v1b)/b, (v2a)/a, (v2b)/b, (v3a)/a, (v3b)/b, (v4a)/a, (v4b)/b, (e0a)/a, (e0b)/b, (e0c)/c, (e1a)/a, (e1b)/b, (e1c)/c, (e2a)/a, (e2b)/b, (e3a)/a, (e3b)/b, (e3c)/c}
{
\draw [black, fill=black] \x circle (3pt);
}


\end{tikzpicture}

%% file: configuration_rounding_2.tikz.tex
\begin{tikzpicture}[scale = 0.3]

\clip (-12,-8) rectangle (11,11);

\coordinate (r3) at (1,6);
\coordinate (e1) at (-1,4); 
\coordinate (v1) at (-1.5, 2);
\coordinate (e2) at (-2,0); 
\coordinate (v2) at (-2.3,-2); 
\coordinate (e3) at (-2.5,-4); 
\coordinate (r0) at (-3,-7); 
\coordinate (v4) at (3,4);
\coordinate (r2) at (2.5,0);
\coordinate (v3) at (4,-3);
\coordinate (e0) at (7,3);
\coordinate (v0) at (-6,6); 
\coordinate (r1) at (-6, 0); 

\coordinate (v0a) at ($ (v0) + (15:1) $);
\coordinate (v0b) at ($ (v0) + (135:1) $);
\coordinate (v0c) at ($ (v0) + (270:1) $);

\coordinate (v1a) at ($ (v1) + (340:1) $);
\coordinate (v1b) at ($ (v1) + (160:1) $);

\coordinate (v2a) at ($ (v2) + (350:1) $);
\coordinate (v2b) at ($ (v2) + (170:1) $);

\coordinate (v3a) at ($ (v3) + (120:1) $);
\coordinate (v3b) at ($ (v3) + (300:1) $);

\coordinate (v4a) at ($ (v4) + (90:1) $);
\coordinate (v4b) at ($ (v4) + (270:1) $);

\coordinate (e0a) at ($ (e0) + (0:1) $);
\coordinate (e0b) at ($ (e0) + (112:1) $);
\coordinate (e0c) at ($ (e0) + (232:1) $);

\coordinate (e1a) at ($ (e1) + (80:1) $);
\coordinate (e1b) at ($ (e1) + (200:1) $);
\coordinate (e1c) at ($ (e1) + (310:1) $);

\coordinate (e2a) at ($ (e2) + (345:1) $);
\coordinate (e2b) at ($ (e2) + (165:1) $);

\coordinate (e3a) at ($ (e3) + (12:1) $);
\coordinate (e3b) at ($ (e3) + (142:1) $);
\coordinate (e3c) at ($ (e3) + (250:1) $);

\draw [ultra thick] ($ (e1c)!0.5!(v4b) $) 
.. controls ($ (e1c)!0.5!(v4b) + (250:3) $) and ($ (v4b)!0.5!(e0c) + (250:3) $) .. ($ (v4b)!0.5!(e0c) $);
\draw [ultra thick] ($ (e0c)!0.5!(v3a) + (0.3,0) $) .. controls ($ (e0c)!0.5!(v3a) + (135:2) $) and ($ (v3a)!0.5!(e3a) + (135:2) $) .. ($ (v3a)!0.5!(e3a) $);
\draw [ultra thick] ($ (e3a)!0.4!(v2a) $) 
to [out=350, in=330] ($ (v1a)!0.5!(e1c) $);
\draw [ultra thick] ($ (v2a)!0.5!(e2a) $) 
to [out=330, in=330] ($ (e2a)!0.5!(v1a) $);

\draw [ultra thick] ($ (v1b)!0.5!(e2b) $)
to [out=150, in=60] ($ (e3b) + (-4.8,1) $);
\draw [ultra thick] ($ (e2b)!0.5!(v2b) $) 
.. controls ($ (e2b)!0.5!(v2b) + (150:2) $) and ($ (v2b)!0.5!(e3b) + (150:2) $) .. ($ (v2b)!0.5!(e3b) $);
\draw [ultra thick] ($ (v0c)!0.5!(e1b) $) to [out=225, in=225] ($ (e1b)!0.5!(v1b) $);

\draw [ultra thick] ($ (v4a)!0.5!(e0b) $) to [bend left=15] ($ (e0b) + (-3.5, 4.3) $);
\draw [ultra thick] ($ (v4a)!0.3!(e1a) $) to [bend right=75] ($ (e1a)!0.5!(v0a) $);

\draw [ultra thick] ($ (v3b) + (-3.8, -2.2) $) to [bend left=60] ($ (e3c) + (-6.8, 0.6) $);
\draw [ultra thick] ($ (e0a) + (-0.2, -4.1) $) .. controls ($ (e0a) + (-0.2, -4.1) + (0:6) $) and ($ (e0a) + (-4.3, 7) + (30:6) $) .. ($ (e0a) + (-4.3, 7) $);

\draw [ultra thick] ($ (e1c)!0.5!(v4b) $) -- ($ (v4a)!0.3!(e1a) $);
\draw [ultra thick] ($ (v4b)!0.5!(e0c) $) -- ($ (v4a)!0.5!(e0b) $);
\draw [ultra thick] ($ (e0c)!0.5!(v3a) + (0.3,0) $) -- ($ (e0a) + (-0.2, -4.1) $);
\draw [ultra thick] ($ (v3a)!0.5!(e3a) $) -- ($ (v3b) + (-3.8, -2.2) $) ;
\draw [ultra thick] ($ (e3a)!0.4!(v2a) $) -- ($ (v2b)!0.5!(e3b) $) ;
\draw [ultra thick] ($ (v2a)!0.5!(e2a) $) -- ($ (e2b)!0.5!(v2b) $);
\draw [ultra thick] ($ (e2a)!0.5!(v1a) $) -- ($  (v1b)!0.5!(e2b) $);
\draw [ultra thick] ($ (v1a)!0.5!(e1c) $) -- ($ (e1b)!0.5!(v1b) $);
\draw [ultra thick] ($ (e3b) + (-4.8,1) $) -- ($ (e3c) + (-6.8, 0.6) $);
\draw [ultra thick] ($ (e0b) + (-3.5, 4.3) $) -- ($ (e0a) + (-4.3, 7) $);
\draw [ultra thick] ($ (v0c)!0.5!(e1b) $) -- ($ (e1a)!0.5!(v0a) $);

\draw [ultra thick, brown] (v0c) 
.. controls ($ (v0c) + (210:6) $) and ($ (e3b) + (200:6) $) .. (e3b)
to [bend left=12] (e1b) 
to (v0c);

\draw [ultra thick, brown] (e1c)
to (v1a) to (e2a) to (v2a) to (e3a)
to [bend right=5] (v3a)
to [bend right=10] (e0c)
to (v4b)
to (e1c);

\draw [ultra thick, brown] (v0a)
to (e1a)
to (v4a)
to (e0b)
.. controls ($ (e0b) + (60:5) $) and ($ (v0a) + (60:5) $) .. (v0a);

\draw [ultra thick, brown] (v0b)
.. controls ($ (v0b) + (210:10) $) and ($ (e3c) + (200:9) $) .. (e3c)
to [out=300, in=220] (v3b) 
to [out=20, in=300] (e0a)
.. controls ($ (e0a) + (60:8) $) and ($ (v0b) + (60:8) $) .. (v0b);


\foreach \x/\word in {(e0)/e0, (e1)/e1, (e2)/e2, (e3)/e3}
{
\draw [green!50!black, fill=green!50!black] \x circle  (10pt);
}

\foreach \x/\word in {(v0)/v0, (v1)/v1, (v2)/v2, (v3)/v3, (v4)/v4}
{
\draw [blue, fill=blue] \x circle (10pt);
}

\foreach \x/\word in {(v0a)/a, (v0b)/b, (v0c)/c, (v1a)/a, (v1b)/b, (v2a)/a, (v2b)/b, (v3a)/a, (v3b)/b, (v4a)/a, (v4b)/b, (e0a)/a, (e0b)/b, (e0c)/c, (e1a)/a, (e1b)/b, (e1c)/c, (e2a)/a, (e2b)/b, (e3a)/a, (e3b)/b, (e3c)/c}
{
\draw [black, fill=black] \x circle (3pt);
}


\end{tikzpicture}

%% file: spanning_tree_1.tikz.tex
\begin{tikzpicture}[scale = 0.25]

\clip (-12,-10) rectangle (11,11);

\coordinate (r3) at (1,6);
\coordinate (e1) at (-1,4); 
\coordinate (v1) at (-1.5, 2);
\coordinate (e2) at (-2,0); 
\coordinate (v2) at (-2.3,-2); 
\coordinate (e3) at (-2.5,-4); 
\coordinate (r0) at (-3,-7); 
\coordinate (v4) at (3,4);
\coordinate (r2) at (2.5,0);
\coordinate (v3) at (4,-3);
\coordinate (e0) at (7,3);
\coordinate (v0) at (-6,6); 
\coordinate (r1) at (-6, 0); 

\draw [line width=5mm, blue!20, rounded corners=1mm, line cap=round] (r3)
to [out=210, in=70] (e1)
to [out=330, in=90] (r2)
to [out=270, in=20] (e3)
to [out=120, in=240] (r1)
to [out=330, in=180] (e2);

\draw [line width=5mm, blue!20, rounded corners=1mm, line cap=round] (e2)
to [out=180, in=330] (r1)
to [out=240, in=120] (e3)
to [out=270, in=90] (r0);

\draw [line width=5mm, blue!20, rounded corners=1mm, line cap=round] (r0)
to [out=90, in=270] (e3)
to [out=20, in=270] (r2)
to [out=0, in=240] (e0);

\draw [line width=5mm, blue!20, rounded corners=1mm, line cap=round] (e0)
to [out=240, in=0] (r2)
to [out=90, in=330] (e1)
to [out=70, in=210] (r3);

\draw [thick, blue, dotted] (e1) to [out=190, in=90] (r1); 
\draw [ultra thick, blue] (r1) to [out=240, in=120] (e3); 
\draw [ultra thick, blue] (r1) to [out=330, in=180] (e2); 

\draw [ultra thick, blue] (e3) to [out=20, in=270] (r2); 
\draw [ultra thick, blue] (r2) to [out=90, in=330] (e1); 
\draw [thick, blue, dotted] (e2) to [out=0, in=180] (r2); 
\draw [ultra thick, blue] (r2) to [out=0, in=240] (e0); 

\draw [ultra thick, blue] (r3) to[out=210,in=70] (e1); 
\draw [thick, blue, dotted] (e0) to [out=120, in=30] (r3); 

\draw [ultra thick, blue] (e3) to [out=270, in=90] (r0); 
\draw [thick, blue, dotted] (r0)
.. controls ($ (r0) + (270:8) $) and ($ (e0) + (345:10) $) .. (e0); 

\draw [ultra thick, red] (e1) to [out = 250, in = 70] (v1); 
\draw [ultra thick, red] (v1) to (e2); 
\draw [ultra thick, red] (e2) to (v2); 
\draw [ultra thick, red] (v2) to (e3); 
\draw [ultra thick, red] (e0) to [out=165, in=0] (v4); 
\draw [ultra thick, red] (v4) to [out=180, in=10] (e1); 
\draw [ultra thick, red] (v0)
.. controls ($ (v0) + (210:8) $) and ($ (e3) + (200:8) $) .. (e3); 
\draw [ultra thick, red] (e1) to [out=150, in=330] (v0); 
\draw [ultra thick, red] (e3) to [out=300, in=210] (v3); 
\draw [ultra thick, red] (v3) to [out=30, in=300] (e0); 
\draw [ultra thick, red] (e0) 
.. controls ($ (e0) + (60:7) $) and ($ (v0) + (60:7) $) .. (v0); 

\foreach \x/\word in {(r0)/r0, (r1)/r1, (r2)/r2, (r3)/r3}
{
\draw [red, fill=red] \x circle  (10pt);
}

\foreach \x/\word in {(e0)/e0, (e1)/e1, (e2)/e2, (e3)/e3}
{
\draw [green!50!black, fill=green!50!black] \x circle  (10pt);
}

\foreach \x/\word in {(v0)/v0, (v1)/v1, (v2)/v2, (v3)/v3, (v4)/v4}
{
\draw [blue, fill=blue] \x circle (10pt);
}

\end{tikzpicture}

%% file: spanning_tree_2.tikz.tex
\begin{tikzpicture}[scale = 0.4]

\clip (-12,-7.5) rectangle (11,11);

\coordinate (r3) at (0.4,5.7); 
\coordinate (e1) at (-1,4); 
\coordinate (v1) at (-1.5, 2);
\coordinate (e2) at (-2,0); 
\coordinate (v2) at (-2.3,-2); 
\coordinate (e3) at (-2.5,-4); 
\coordinate (r0) at (-3,-6.5); 
\coordinate (v4) at (3,4);
\coordinate (r2) at (2.5,0);
\coordinate (v3) at (4,-3);
\coordinate (e0) at (7,3);
\coordinate (v0) at (-6,6); 
\coordinate (r1) at (-5, 1); 

\coordinate (v0a) at ($ (v0) + (15:1) $);
\coordinate (v0b) at ($ (v0) + (135:1) $);
\coordinate (v0c) at ($ (v0) + (270:1) $);

\coordinate (v1a) at ($ (v1) + (340:1) $);
\coordinate (v1b) at ($ (v1) + (160:1) $);

\coordinate (v2a) at ($ (v2) + (350:1) $);
\coordinate (v2b) at ($ (v2) + (170:1) $);

\coordinate (v3a) at ($ (v3) + (120:1) $);
\coordinate (v3b) at ($ (v3) + (300:1) $);

\coordinate (v4a) at ($ (v4) + (90:1) $);
\coordinate (v4b) at ($ (v4) + (270:1) $);

\coordinate (e0a) at ($ (e0) + (0:1) $);
\coordinate (e0b) at ($ (e0) + (112:1) $);
\coordinate (e0c) at ($ (e0) + (232:1) $);

\coordinate (e1a) at ($ (e1) + (80:1) $);
\coordinate (e1b) at ($ (e1) + (200:1) $);
\coordinate (e1c) at ($ (e1) + (310:1) $);

\coordinate (e2a) at ($ (e2) + (345:1) $);
\coordinate (e2b) at ($ (e2) + (165:1) $);

\coordinate (e3a) at ($ (e3) + (12:1) $);
\coordinate (e3b) at ($ (e3) + (142:1) $);
\coordinate (e3c) at ($ (e3) + (250:1) $);


\draw [line width=6mm, rounded corners=1mm, line cap=round] (r3)
to [out=210, in=70] (e1a)
to (e1b)
to (e1c)
to [out=330, in=90] (r2)
to [out=270, in=20] (e3a)
to (e3b)
to [out=120, in=240] (r1)
to [out=330, in=180] (e2b)
to (e2a);

\draw [line width=6mm, rounded corners=1mm, line cap=round] (e2a)
to (e2b)
to [out=180, in=330] (r1)
to [out=240, in=120] (e3b)
to (e3c)
to [out=270, in=90] (r0);

\draw [line width=6mm, rounded corners=1mm, line cap=round] (r0)
to [out=90, in=270] (e3c)
to (e3a)
to [out=20, in=270] (r2)
to [out=0, in=240] (e0c)
to (e0a)
to (e0b)
to (e0c)
to [out=240, in=0] (r2)
to [out=90, in=330] (e1c)
to (e1a);
to [out=70, in=210] (r3);

\draw [line width=5mm, white, rounded corners=1mm, line cap=round] (r3)
to [out=210, in=70] (e1a)
to (e1b)
to (e1c)
to [out=330, in=90] (r2)
to [out=270, in=20] (e3a)
to (e3b)
to [out=120, in=240] (r1)
to [out=330, in=180] (e2b)
to (e2a);

\draw [line width=5mm, white, rounded corners=1mm, line cap=round](e2a)
to (e2b)
to [out=180, in=330] (r1)
to [out=240, in=120] (e3b)
to (e3c)
to [out=270, in=90] (r0);

\draw [line width=5mm, white, rounded corners=1mm, line cap=round](r0)
to [out=90, in=270] (e3c)
to (e3a)
to [out=20, in=270] (r2)
to [out=0, in=240] (e0c)
to (e0a)
to (e0b)
to (e0c)
to [out=240, in=0] (r2)
to [out=90, in=330] (e1c)
to (e1a);	
to [out=70, in=210] (r3);

\draw [ultra thick, blue] (r1) to [out=240, in=120] (e3b); 
\draw [ultra thick, blue] (r1) to [out=330, in=180] (e2b); 
\draw [ultra thick, blue] (e3a) to [out=20, in=270] (r2); 
\draw [ultra thick, blue] (r2) to [out=90, in=330] (e1c); 
\draw [ultra thick, blue] (r2) to [out=0, in=240] (e0c); 
\draw [ultra thick, blue] (r3) to[out=210,in=70] (e1a); 
\draw [ultra thick, blue] (e3c) to [out=270, in=90] (r0);

\def\boundaries{(v0c) 
.. controls ($ (v0c) + (210:6) $) and ($ (e3b) + (200:6) $) .. (e3b)
to [bend left=12] (e1b) 
to (v0c)

(e1c)
to (v1a) to (e2a) to (v2a) to (e3a)
to [bend right=5] (v3a)
to [bend right=10] (e0c)
to (v4b)
to (e1c)

(v0a)
to (e1a)
to (v4a)
to (e0b)
.. controls ($ (e0b) + (60:5) $) and ($ (v0a) + (60:5) $) .. (v0a)

(v0b)
.. controls ($ (v0b) + (210:10) $) and ($ (e3c) + (200:9) $) .. (e3c)
to [out=300, in=220] (v3b) 
to [out=20, in=300] (e0a)
.. controls ($ (e0a) + (60:8) $) and ($ (v0b) + (60:8) $) .. (v0b)}

\begin{scope}[even odd rule]
\clip \boundaries;
\fill [white] \boundaries;
\end{scope}

\draw [draw=none, fill=green!50!black, opacity=0.2]
(e0a) to (e0b) to (e0c) to (e0a);

\draw [draw=none, fill=green!50!black, opacity=0.2]
(e1a) to (e1b) to (e1c) to (e1a);

\draw [draw=none, fill=green!50!black, opacity=0.2]
(e2a) to [bend right=30] (e2b) to [bend right=30] (e2a);

\draw [draw=none, fill=green!50!black, opacity=0.2]
(e3a) to (e3b) to (e3c) to (e3a);

\foreach \x/\word/\num in {(r0)/r0/0, (r1)/r1/1, (r2)/r2/2, (r3)/r3/3}
{
\draw [red, fill=red] \x circle  (10pt);
}

\foreach \x/\word in {(e0)/e0, (e1)/e1, (e2)/e2, (e3)/e3}
{
\draw [green!50!black, fill=green!50!black] \x circle  (10pt);
}

\foreach \x/\word in {(v0)/v0, (v1)/v1, (v2)/v2, (v3)/v3, (v4)/v4}
{
\draw [blue, fill=blue] \x circle (10pt);
}


\draw [ultra thick, brown] (v0c) 
.. controls ($ (v0c) + (210:6) $) and ($ (e3b) + (200:6) $) .. (e3b)
to [bend left=12] (e1b) 
to (v0c);

\draw [ultra thick, brown] (e1c)
to (v1a) to (e2a) to (v2a) to (e3a)
to [bend right=5] (v3a)
to [bend right=10] (e0c)
to (v4b)
to (e1c);

\draw [ultra thick, brown] (v0a)
to (e1a)
to (v4a)
to (e0b)
.. controls ($ (e0b) + (60:5) $) and ($ (v0a) + (60:5) $) .. (v0a);

\draw [ultra thick, brown] (v0b)
.. controls ($ (v0b) + (210:10) $) and ($ (e3c) + (200:9) $) .. (e3c)
to [out=300, in=220] (v3b) 
to [out=20, in=300] (e0a)
.. controls ($ (e0a) + (60:8) $) and ($ (v0b) + (60:8) $) .. (v0b);

\foreach \x/\word in {(v0a)/0a, (v0b)/0b, (v0c)/0c, (v1a)/1a, (v1b)/1b, (v2a)/2a, (v2b)/2b, (v3a)/3a, (v3b)/3b, (v4a)/4a, (v4b)/4b, (e0a)/0a, (e0b)/0b, (e0c)/0c, (e1a)/1a, (e1b)/1b, (e1c)/1c, (e2a)/2a, (e2b)/2b, (e3a)/3a, (e3b)/3b, (e3c)/3c, (r0)/0, (r1)/1, (r2)/2, (r3)/3}
{
}

\end{tikzpicture}

%% file: max_valence_increase.tikz.tex
\begin{tikzpicture}[
scale=0.7, 
suture/.style={thick},
boundary/.style={ultra thick},
attaching/.style={dotted},
vertex/.style={draw=red, fill=red}]

\foreach \x in {0, 8, -8}
{
\draw [suture, xshift= \x cm] (-2.5,-2) arc (180:0:0.5);
\draw [suture, xshift = \x cm] (-1,-2) arc (180:0:0.5);
\draw [xshift = \x cm] (0.75,-2) node {$\ldots$};
\draw [suture, xshift = \x cm] (1.5,-2) arc (180:0:0.5);
\draw [decoration = {brace, mirror}, decorate, xshift = \x cm] (-2.5,-2.5) -- (2.5,-2.5)
node [anchor = north, pos=0.5, yshift=-0.5 cm] {$\V-1$};
}

\draw [->] (3.5,1) -- (5,1);
\draw [->] (-3.5,1) -- (-5,1);

\draw [suture] (-2.5,0) -- (2.5,0);
\draw [suture] (-2.5,2) -- (2.5,2);
\draw [suture] (-1,5) arc (-180:0:1);
\draw [attaching] (0,0) -- (0,4);
\draw (1,-1) node {$c$};
\draw (1.5,1) node {$c^+_i$};
\draw (1.5,3) node {$c^-_j$};
\draw (-0.5,2.5) node {$a$};

\draw [suture, xshift = -8 cm] (-1,5) to [bend left=30] (-2.5,2);
\draw [suture, xshift = -8 cm] (1,5) to [bend left=30] (-2.5,0);
\draw [suture, xshift = -8 cm] (2.5,2) to [bend right=90] (2.5,0);

\draw [suture, xshift = 8 cm] (1,5) to [bend right=30] (2.5,2);
\draw [suture, xshift = 8 cm] (-1,5) to [bend right=30] (2.5,0);
\draw [suture, xshift = 8 cm] (-2.5,2) to [bend left=90] (-2.5,0);

\end{tikzpicture}

%% file: fig-8_trail_configuration.tikz.tex
\begin{tikzpicture}[scale = 0.3,
knot/.style={ultra thick},
marker/.style={fill=black},
violetedge/.style={ultra thick, blue},
emeraldedge/.style={ultra thick, green!50!black},
rededge/.style={ultra thick, red}
]
\coordinate (b) at (0,-9);
\coordinate (m) at (0,-3);
\coordinate (l) at (-3,0);
\coordinate (r) at (3,0);
\coordinate (vl) at (-5 , -5);
\coordinate (vr) at (5,-5);
\coordinate (vt) at (0, 2);
\coordinate (eb) at (0,-6);
\coordinate (em) at (0,-1);
\coordinate (et) at (0, 6);


\draw [knot, rounded corners=0.5 cm] (0,5)  
.. controls ($ (0,5) + (180:2) $) and ($ (l) + (120:4) $)  .. (l)
.. controls ($ (l) + (195:8) $) and ($ (b) + (210:8) $) .. (b)
.. controls ($ (b) + (150:3) $) and ($ (m) + (210:3) $) .. (m)
to [out=150, in=-60] (l)
to [out=15, in=165] (r)
to [out=240, in=30] (m)
.. controls ($ (m) + (-30:3) $) and ($ (b) + (30:3) $) .. (b)
.. controls ($ (b) + (-30:8) $) and ($ (r) + (-15:8) $) .. (r)
.. controls ($ (r) + (60:4) $) and ($ (0,5) + (0:2) $) .. (0,5);
;

\draw (4,-9) node {\Huge $*$};
\draw (6,-11) node {\Huge $*$};

\draw [rededge] (vl) -- (eb);
\draw [rededge] (vr) -- (eb);
\draw [rededge] (vl) -- (em);
\draw [rededge] (vr) -- (em);
\draw [rededge] (vt) -- (em);
\draw [rededge] (vt) -- (et);
\draw [rededge] (vl) .. controls ($ (vl) + (180:8) $) and ($ (et) + (180:8) $) .. (et);
\draw [rededge] (vr) .. controls ($ (vr) + (0:8) $) and ($ (et) + (0:8) $) .. (et);

\foreach \x/\word in {(vl)/vl, (vr)/vr, (vt)/vt}
{
\draw [blue, fill=blue] \x circle (10pt);
}

\foreach \x/\word in {(eb)/eb, (em)/em, (et)/et}
{
\draw [green!50!black, fill=green!50!black] \x circle  (10pt);
}

\foreach \x/\word in {(b)/b, (m)/m, (l)/l, (r)/r}
{
\draw [red, fill=red] \x circle  (10pt);
}

\end{tikzpicture}